\documentclass[11pt]{amsart}

\usepackage{amssymb}
\usepackage{amsmath}
\usepackage{amsthm}
\usepackage{graphicx,psfrag}
\usepackage{enumerate}

\usepackage[dvips]{hyperref}
\hypersetup{pdfborder={}} 

\theoremstyle{plain}
\newtheorem{theorem}{Theorem}[section]
\newtheorem{lemma}[theorem]{Lemma}
\newtheorem{corollary}[theorem]{Corollary}
\newtheorem{proposition}[theorem]{Proposition}

\theoremstyle{definition}
\newtheorem{definition}[theorem]{Definition}
\newtheorem*{notation}{Notation}
\newtheorem{claim}{Claim}
\newtheorem{claimA}{Claim}
\newtheorem{claimB}{Claim}
\newtheorem{claimC}{Claim}
\newtheorem{claimD}{Claim}
\newtheorem*{claim*}{Claim}

\theoremstyle{remark}

\newtheorem*{remark*}{Remark}

\newtheorem*{acknowledgments}{Acknowledgments}
\newtheorem{question}{Question}

\numberwithin{figure}{section}

\newcommand{\Int}{\mathrm{int}}

\usepackage{graphicx}

\begin{document}

\title{Rank and genus of 3-manifolds}
\author{Tao Li}
\thanks{Partially supported by NSF grants DMS-1005556 and DMS-0705285}

\address{Department of Mathematics \\
 Boston College \\
 Chestnut Hill, MA 02467}
\email{taoli@bc.edu}

\begin{abstract}
We construct a counterexample to the Rank versus Genus Conjecture, i.e.~a closed orientable hyperbolic 3-manifold with rank of its fundamental group smaller than its Heegaard genus.  Moreover, we show that the discrepancy between rank and Heegaard genus can be arbitrarily large for hyperbolic 3-manifolds.  We also construct toroidal such examples containing hyperbolic JSJ pieces.
\end{abstract}

\maketitle

\tableofcontents

\section{Introduction}\label{Sintro} 
A Heegaard splitting of a closed orientable 3-manifold $M$ is a decomposition of $M$ into two handlebodies along a closed surface $S$.  The genus of $S$ is the genus of the Heegaard splitting.  The Heegaard genus of $M$, which we denote by $g(M)$, is the minimal genus over all Heegaard splittings of $M$.  A Heegaard splitting of genus $g$ naturally gives a balanced presentation of the fundamental group $\pi_1(M)$: the core of one handlebody gives $g$ generators and the compressing disks of the other handlebody give a set of $g$ relators.

The rank of $M$, which we denoted by $r(M)$, is the rank of the fundamental group $\pi_1(M)$, that is the minimal number of elements needed to generate $\pi_1(M)$.  By the relation between Heegaard splitting and $\pi_1(M)$ above, it is clear that $r(M)\le g(M)$.   In the 1960s, Waldhausen asked whether $r(M)=g(M)$ for all $M$, see \cite{W2, H2}.  This was called the generalized Poincar\'{e} Conjecture in \cite{H2}, as the case $r(M)=0$ is the Poincar\'{e} conjecture.  

In \cite{BZ},  Boileau and Zieschang found a Seifert fibered space with $r(M)=2$ and $g(M)=3$.  Later, Schultens and Weidmann \cite{SW} showed that there are graph manifolds $M$ with discrepancy $g(M)-r(M)$ arbitrarily large.  A crucial ingredient in all these examples is that the fundamental group of a Seifert fibered space has an element commuting with other elements and, for a certain class of Seifert fibered spaces, one can use this property to find a smaller generating set of $\pi_1(M)$ than the one given by a Heegaard splitting.  However, these examples are very special and the fundamental group of a closed hyperbolic 3-manifold does not contain such commuting elements, so the modernized version of this old conjecture is whether $r(M)=g(M)$ for hyperbolic 3-manifolds, see \cite[Conjecture 1.1]{Sh}. This conjecture is sometimes called the Rank versus Genus Conjecture or the Rank Conjecture, as $r(M)$ can be viewed as the algebraic rank and $g(M)$ can be regarded as the geometric rank of $M$.  This conjecture is also related to the Fixed Price Conjecture in topological dynamics \cite{AN}.

Indeed, there are some positive evidences for this conjecture.  In \cite{So}, Souto proved $r(M)=g(M)$ for any fiber bundle whose monodromy is a high power of a pseudo-Anosov map.  In \cite{NS}, Namazi and Souto showed that rank equals genus if the gluing map of a Heegaard splitting is a high power of a generic pseudo-Anosov map.  This means that, in some sense, a sufficiently complicated hyperbolic 3-manifold satisfies this conjecture.  On the other hand, many simple hyperbolic 3-manifolds also satisfy the conjecture, e.g., if $g(M)=2$ then $\pi_1(M)$ cannot be cyclic and hence $r(M)=g(M)=2$.

In this paper, we give a negative answer to this conjecture.

\begin{theorem}\label{Tclosed}
There is a closed orientable hyperbolic 3-manifold with rank of its fundamental group smaller than its Heegaard genus.  Moreover, the discrepancy between its rank and Heegaard genus can be arbitrarily large.
\end{theorem}

The original question of Waldhausen  \cite{W2, H2} asks whether rank equals genus for both closed manifolds and manifolds with boundary.  The main construction in this paper is an example of manifold with boundary.  In fact, Theorem~\ref{Tclosed} follows from Theorem~\ref{Tmain} and a theorem in \cite{L5}.

\begin{theorem}\label{Tmain}
There is a compact irreducible atoroidal 3-manifold $M$ with connected boundary such that $r(M)<g(M)$. 
\end{theorem}

The examples in this paper can also be easily modified to produce the first such examples that contain hyperbolic JSJ pieces.

\begin{theorem}\label{TJSJ}
Every 2-bridge knot exterior can be a JSJ piece of a close 3-manifold $M$ with $r(M)<g(M)$.
\end{theorem}

Now we briefly describe our construction.  The main construction is a gluing of three 3-manifolds with boundary along annuli.  The first piece $X$ is obtained by drilling out a tunnel in a 2-bridge knot exterior.  The boundary of $X$ is a genus two surface and $\pi_1(X)$ is generated by three elements, two of which are conjugate.  The second piece $Y_s$ is obtained by first gluing together three twisted $I$-bundles over non-orientable surfaces, then adding a 1-handle and finally performing a Dehn surgery on a curve in the resulting manifold.  Our manifold $M$ in Theorem~\ref{Tmain} is obtained by gluing two copies of $X$ to $Y_s$ along a pair of annuli in $\partial Y_s$.  To get a closed 3-manifold in Theorem~\ref{Tclosed}, we glue a handlebody to $\partial M$ using a sufficiently complicated gluing map.

We organize the paper as follows.  In section~\ref{S2}, we briefly review some basics of Heegaard splitting and some results in \cite{L5}.  We also explain in section~\ref{S2} why Theorem~\ref{Tclosed} follows from Theorem~\ref{Tmain} and some results in \cite{L5}. In section~\ref{Sann}, we prove some useful facts on Heegaard splittings of an annulus sum, i.e., a 3-manifold obtained by gluing a pair of 3-manifolds with boundary along an annulus.  The lemmas in section~\ref{Sann} will be used in calculating the Heegaard genus of our manifold $M$ described above.  In section~\ref{SX}, we construct the first piece $X$.  We construct the second piece $Y_s$ in section~\ref{SY}.   The main technical part of the paper is to compute the Heegaard genus of our manifold $M$, and this is carried out in section~\ref{Sincomp} and section~\ref{Sgenus}.  We finish the proof of both Theorem~\ref{Tmain} and Theorem~\ref{TJSJ} in section~\ref{Sgenus}.  In section~\ref{Sopen}, we discuss some interesting open questions concerning rank and genus.

\begin{acknowledgments}
This research started during the author's visit to Princeton University in 2009.  The author would like to thank the math department of Princeton University for its hospitality. 
\end{acknowledgments}

\section{Heegaard splittings and amalgamation}\label{S2}

\begin{notation}
Throughout this paper, for any topological space $X$, we denote the number of components of $X$ by
$|X|$, the interior of $X$ by $\Int(X)$, the closure of $X$ by $\overline{X}$ and a small open neighborhood of $X$ by $N(X)$.  We denote the disjoint union of $X$ and $Y$ by  $X\coprod Y$, and use $I$ to denote the interval $[0,1]$.  If $X$ is a 3-manifold, $g(X)$ denotes the Heegaard genus of $X$, and if $X$ is a surface, $g(X)$ denotes the genus of the surface $X$.
\end{notation}

 A \emph{compression body} is a connected $3$-manifold $V$ obtained by adding 2-handles to a product $S\times I$ along $S\times\{0\}$, where $S$ is a closed and orientable surface, and then capping off any resulting 2-sphere boundary components by 3-balls.  The surface $S\times\{1\}$ is denoted by $\partial_+V$ and $\partial V - \partial_{+}V$ is denoted by $\partial_{-}V$.  The cases $V = S \times I$ and $\partial_{-}V = \emptyset $ are allowed. In the first case we say $V$ is a trivial compression body and in the second case $V$ is a handlebody.  One can also view a compression body with $\partial_-V\ne\emptyset$ as a manifold obtained by adding 1-handles on the same side of $\partial_-V\times I$.

 A \emph{Heegaard splitting} of a $3$-manifold $M$ is a decomposition  $M = V_1 \cup V_2$ where $V_1$ and $V_2$ are compression bodies and $\partial_{+}V_1 = V_1 \cap V_2 = \partial_{+}V_2$.  The surface $\Sigma = \partial_{+}V_1 =  \partial_{+}V_2$ is called the \emph{Heegaard surface} of the Heegaard splitting.  Every compact orientable $3$-manifold has a Heegaard splitting.  The \emph{Heegaard genus} of a 3-manifold is the minimal genus of all Heegaard surfaces of the 3-manifold.

A Heegaard splitting of $M$ is \emph{stabilized} if $M$ contains a 3-ball whose boundary 2-sphere intersects the Heegaard surface in a single non-trivial circle in the Heegaard surface. 
A Heegaard splitting $M=V_1\cup_\Sigma V_2$ is \emph{weakly reducible} if there is a compressing disk $D_i$ in $V_i$ ($i=1,2$) such that $\partial D_1\cap\partial D_2=\emptyset$.  If a Heegaard splitting is not weakly reducible, then it is said to be \emph{strongly irreducible}.   A stabilized Heegaard splitting is weakly reducible and is not of minimal genus, see \cite{S2} for more details.    
A theorem of Casson and Gordon \cite{CG} says that an unstabilized Heegaard splitting of an irreducible non-Haken 3-manifold is always strongly irreducible.  

For unstabilized and weakly reducible Heegaard splittings, there is an operation called \emph{untelescoping} of the Heegaard splitting, which is an rearrangement of the handles given by the Heegaard splitting.  An untelescoping of a Heegaard splitting gives a decomposition of $M$ into several blocks along incompressible surfaces, each block having a strongly irreducible Heegaard splitting, see \cite{ST, S2}.  We summarize this as the following theorem due to Scharlemann and Thompson \cite{ST} (except part (4) of the theorem is \cite[Lemma 2]{SS}).

\begin{theorem}\label{TST}
Let $M$ be a compact irreducible and orientable 3--manifold with incompressible boundary.  Let $S$ be an unstabilized Heegaard surface of $M$.  Then the untelescoping of the Heegaard splitting gives a decomposition of $M$ as follows.
\begin{enumerate}
\item $M=N_0\cup_{F_1}N_1\cup_{F_2}\dots\cup_{F_m}N_m$, where each $F_i$ is incompressible in $M$.  
\item Each $N_i=A_i\cup_{\Sigma_i}B_i$, where each $A_i$ and $B_i$ is a union of compression bodies with $\partial_+A_i=\Sigma_i=\partial_+B_i$ and $\partial_-A_i=F_i=\partial_-B_{i-1}$.  
\item Each component of $\Sigma_i$ is a strongly irreducible Heegaard surface of a component of $N_i$. 
\item $\chi(S)=\sum\chi(\Sigma_i)-\sum\chi(F_i)$, see \cite[Lemma 2]{SS}.
\end{enumerate}
\end{theorem}

Note that a decomposition of $M$ as above (without requiring $F_i$ to be incompressible and $\Sigma_i$ to be strongly irreducible) is called a \emph{generalized Heegaard splitting} of $M$.

The converse of untelescoping is amalgamation of Heegaard splittings, see \cite{Sch}. In fact, one can amalgamate the Heegaard surfaces $\Sigma_i$ in Theorem~\ref{TST} along the incompressible surfaces $F_i$ to get back the Heegaard surface $S$.  The genus calculation in the amalgamation follows from the formula in part (4) of Theorem~\ref{TST}.  

Let $M_1$ and $M_2$ be two compact orientable irreducible 3-manifolds with connected boundary and suppose $\partial M_1\cong \partial M_2\cong F$.  We can glue $M_1$ to $M_2$ via a homeomorphism  $\phi\colon\partial M_1\to\partial M_2$ and get a closed 3-manifold $M$.  As in \cite{L5}, we can define a complexity for the gluing map $\phi$ using curve complex, see \cite{L5, H} for the definition of curve complex.  First, we view $M_1$ and $M_2$ as sub-manifolds of $M$ and view $F=\partial M_1=\partial M_2$ as a surface in $M$. 
Let $\mathcal{U}_i$ be the set of vertices in the curve complex $\mathcal{C}(F)$ represented by the boundary curves of properly embedded essential orientable surfaces in $M_i$ of maximal Euler characteristic (among all such essential orientable surfaces in $M_i$).  In particular, if $\partial M_i$ is compressible in $M_i$, then $\mathcal{U}_i$ is the disk complex, i.e., vertices represented by the set of boundary curves of compressing disks for $\partial M_i$.  The complexity $d(M)=d(\mathcal{U}_1,\mathcal{U}_2)$ is the distance from $\mathcal{U}_1$ to $\mathcal{U}_2$ in the curve complex $\mathcal{C}(F)$.  This complexity can be viewed as a generalization of the distance defined by Hempel \cite{H}.  Note that $d(M)$ can be arbitrarily large if the gluing map $\phi$ is a sufficiently high power of a generic pseudo-Anosov map.  The following is a theorem in \cite{L5} which generalizes earlier results in \cite{ST1, La, L4, L3}.

\begin{theorem}[\cite{L5}]\label{TL5}
Let $M_1$ and $M_2$ be two compact orientable irreducible 3-manifolds with connected boundary and $\partial M_1\cong \partial M_2$.  Let $M$ be the closed manifold obtained by gluing $M_1$ to $M_2$ via a homeomorphism  $\phi\colon \partial M_1\to\partial M_2$.  If $d(M)$ is sufficiently large, then $g(M)=g(M_1)+g(M_2)-g(\partial M_i)$.
\end{theorem}

Note that if we amalgamate Heegaard surfaces $S_1$ and $S_2$ of $M_1$ and $M_2$ respectively along $\partial M_i$, then the resulting Heegaard surface of $M$ has genus $g(S_1)+g(S_2)-g(\partial M_i)$, same as the formula in Theorem~\ref{TL5}.

A special case of Theorem~\ref{TL5} is that if $M_2$ is a handlebody and the gluing map $\phi$ is sufficiently complicated, then the Heegaard genus of $M$ does not change i.e., $g(M)=g(M_1)$.  This observation and the results in \cite{L5} give the following corollary which says that if the rank conjecture fails for 3-manifolds with connected boundary, then it also fails for closed 3-manifolds.  

Note that the reason we assume our manifold $M$ in Corollary~\ref{CL5} has connected boundary is that its Heegaard splitting is a decomposition of $M$ into a handlebody and a compression body, so the Heegaard splitting also gives a geometric presentation of $\pi_1(M)$.  In particular, we also have $r(M)\le g(M)$.  For a manifold with more than one boundary component, it is possible that a Heegaard surface separates the boundary components and has genus smaller than the rank of the 3-manifold.

\begin{corollary}\label{CL5}
Let $M$ be a compact orientable 3-manifold with connected boundary and suppose $r(M)<g(M)$.  Then there is a closed 3-manifold $\widehat{M}$ obtained by gluing a handlebody to $M$ along the boundary and via a sufficiently complicated gluing map, such that $r(\widehat{M})<g(\widehat{M})$.  Moreover, if $M$ is irreducible and atoroidal, then $\widehat{M}$ is hyperbolic.  
\end{corollary}
\begin{proof}
By Theorem~\ref{TL5}, if we glue a handlebody to $M$ using a sufficiently complicated gluing map, then $g(\widehat{M})=g(M)$.  Moreover, in terms of fundamental group, adding a handlebody to $M$ is the same as adding some relators to $\pi_1(M)$.  Hence $r(\widehat{M})\le r(M)$.  Since $r(M)<g(M)$, we have $r(\widehat{M})<g(\widehat{M})$.

Now we suppose $M$ is irreducible and atoroidal.  
Since $M$ is irreducible, by \cite[Theorem 1.2]{L5}, if the gluing map is sufficiently complicated, $\widehat{M}$ is also irreducible.  Moreover, if $\widehat{M}$ contains an incompressible torus and the gluing map is sufficiently complicated, by \cite[Lemma 6.6]{L5}, the torus can be isotoped in $\widehat{M}$ to be disjoint from the surface $\partial M$.   Since $\widehat{M}-\Int(M)$ is a handlebody, this means that the incompressible torus lies in $M$, which contradicts our hypothesis that $M$ is atoroidal.  Thus $\widehat{M}$ must also be atoroidal.  By Perelman \cite{P1, P2, P3}, this means that $\widehat{M}$ is hyperbolic. 
\end{proof}
\begin{remark*}
The 3-manifold $\widehat{M}$ constructed in this paper is Haken, so one only needs Thurston's hyperbolization theorem for Haken manifolds \cite{T} to conclude that $\widehat{M}$ is hyperbolic.
\end{remark*}

Using Corollary~\ref{CL5}, we can immediately see that Theorem~\ref{Tclosed} follows from Theorem~\ref{Tmain}.

\begin{proof}[Proof of Theorem~\ref{Tclosed} using Theorem~\ref{Tmain}]
The first part of Theorem~\ref{Tclosed} follows directly from Theorem~\ref{Tmain} and Corollary~\ref{CL5}. Next we explain why the discrepancy between rank and genus can be arbitrarily large.

Let $M$ be an irreducible and atoroidal 3-manifold as in Theorem~\ref{Tmain}.  Let $M_n=M\#_\partial M\#_\partial\cdots\#_\partial M$ be the boundary connected sum of $n$ copies of $M$ (i.e., connecting $n$ copies of $M$ using 1-handles).  Although $M_n$ has compressible boundary, $M_n$ is still irreducible and atoroidal.

By Grushko's theorem \cite{G}, $r(M_n)=nr(M)$.  By \cite{H1} and \cite[Corollary 1.2]{CG}, Heegaard genus is additive under boundary connected sum.  So $g(M_n)=ng(M)$.  By Corollary~\ref{CL5}, we can glue a handlebody to $M_n$ via a sufficiently complicated gluing map so that the resulting closed 3-manifold $\widehat{M}_n$ satisfies $g(\widehat{M}_n)=g(M_n)$.  Since $r(\widehat{M}_n)\le r(M_n)$, we have $g(\widehat{M}_n)-r(\widehat{M}_n)\ge g(M_n)-r(M_n)=n(g(M)-r(M))\ge n$.  
Since $M_n$ is irreducible and atoroidal, $\widehat{M}_n$ is a closed hyperbolic 3-manifold with  $g(\widehat{M}_n)-r(\widehat{M}_n)\ge n$.  
\end{proof}

\section{Annulus sum}\label{Sann}

Let $M_1$ and $M_2$ be two 3-manifolds with boundary.  Let $A_i$ ($i=1,2$) be an annulus in $\partial M_i$.  We can glue $M_1$ and $M_2$ together via a homeomorphism between $A_1$ and $A_2$, and we call the resulting 3-manifold an \emph{annulus sum} of $M_1$ and $M_2$ along the annuli $A_1$ and $A_2$.

Our main construction in this paper is an annulus sum of three 3-manifolds with boundary.  A key part of the proof is a study of Heegaard surfaces in the annulus sum.  In this section, we prove some basic properties concerning incompressible surfaces and strongly irreducible surfaces in an annulus sum.  These are the basic tools in calculating the Heegaard genus of such manifolds.

\begin{definition}\label{Dbase}
Let $N$ be a compact 3-manifold with boundary and let $F$ be a surface properly embedded in $N$.  Let $D$ be a $\partial$-compressing disk for $F$ with $\partial D=\alpha\cup\beta$, $F\cap D=\beta$ and $D\cap\partial N=\alpha$.  We call $\alpha$ the \emph{base arc} of the $\partial$-compressing disk $D$.
\end{definition}

\begin{definition}\label{Dsmall}
Let $N$ be a compact orientable 3-manifold.  We say $N$ is \emph{small} if $N$ contains no closed orientable non-peripheral incompressible surface.  Let $A$ be a sub-surface of $\partial N$.  We say $N$ is \emph{$A$-small} if every properly embedded orientable incompressible surface in $N$ with boundary in $A$ is $\partial$-parallel in $N$.
\end{definition}

Lemma~\ref{Lmonogon} and Lemma~\ref{LIbundle} are well-known facts about incompressible surfaces.  For completeness, we give a proof.

\begin{lemma}\label{Lmonogon}
Let $N$ be an orientable irreducible 3-manifold and let $A$ be a collection of annuli in $\partial N$.  Let $F$ be a connected orientable incompressible surface properly embedded in $N$ with $\partial F\subset A$.  If $F$ admits a $\partial$-compressing disk with base arc in $A$, then $F$ is an annulus parallel to a sub-annulus of $A$.
\end{lemma}
\begin{proof}
First, since $F$ is incompressible in $N$, $\partial F$ is a collection of essential curves in $A$.  Let $D$ be a $\partial$-compressing disk for $F$ with its base arc $\alpha$ in $A$.  If $\partial\alpha$ lies in the same circle of $\partial F$, since $A$ consists of annuli, $\alpha$ is parallel to a subarc of $\partial F$ bounded by $\partial\alpha$.  After pushing $\alpha$ into $\partial F$, $D$ becomes a compressing disk for $F$, contradicting that $F$ is incompressible.  Thus the two endpoints of $\alpha$ must lie in different components of $\partial F$.  

Let $A'$ be the sub-annulus of $A$ bounded by the two curves of $\partial F$ containing $\partial\alpha$.  Let $D_1$ and $D_2$ be two parallel copies of $D$ on opposite sides of $D$ with $\partial D_i=\alpha_i\cup\beta_i$, where $\alpha_i$ is the base arc of $D_i$ and $\beta_i\subset F$.  The arcs $\alpha_1$ and $\alpha_2$ divide $A'$ into two rectangles $R$ and $R'$, where $R'$ contains $\alpha$.  As shown in Figure~\ref{Fmono}(a), $\Delta=D_1\cup R\cup D_2$ is a disk with $\Delta\cap F=\partial\Delta$.  Since $F$ is incompressible, $\partial\Delta$ must bound a disk $\Delta'$ in $F$.  This implies that $F$ is the union of $\Delta'$ and the small rectangle in $F$ between $\beta_1$ and $\beta_2$, and hence $F$ is an annulus.  Moreover, $\Delta\cup\Delta'$ is a 2-sphere. Since $N$ is irreducible, the 2-sphere $\Delta\cup\Delta'$ bounds a 3-ball in $N$.  The union of this 3-ball and the region between $D_1$ and $D_2$ is a solid torus bounded by $F\cup A'$.  Therefore $F$ is an annulus parallel to $A'\subset A$.
\end{proof}

\begin{figure}
\centering
\psfrag{a}{(a)}
\psfrag{b}{(b)}
\psfrag{S}{$S$}
\psfrag{A}{$A$}
\psfrag{+}{$S_+$}
\psfrag{-}{$S_-$}
\psfrag{W}{$W_-$}
\psfrag{V}{$W_+$}
\includegraphics[width=3.5in]{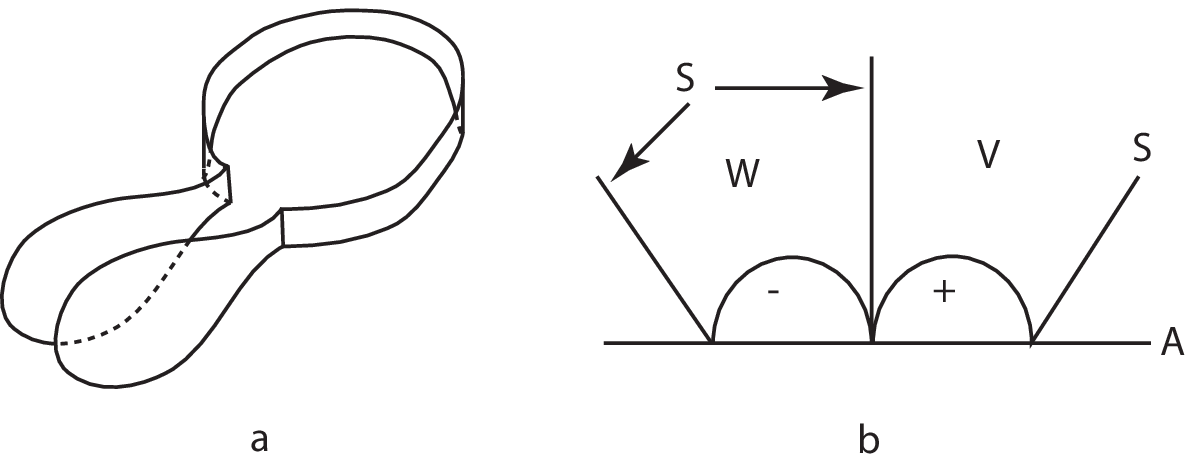}
\caption{}\label{Fmono}
\end{figure}

\begin{definition}\label{Dbundle}
Let $N$ be an $I$-bundle over a compact surface $F$, and let $\pi\colon  N\to F$ be the projection that collapses each $I$-fiber to a point.  We call $\pi^{-1}(\partial F)$ the \emph{vertical boundary} of $N$, denoted by $\partial_vN$, and call $\partial N-\Int(\partial_vN)$ the \emph{horizontal boundary} of $N$, denoted by  $\partial_hN$.  We say a surface $S$ in $N$ is \emph{horizontal} if $S$ is transverse to the $I$-fibers.  We say a surface $G$ in $N$ is \emph{vertical} if $G$ is the union of a collection of subarcs of the $I$-fibers of $N$.
\end{definition}

\begin{lemma}\label{LIbundle}
Let $N$ be a connected orientable 3-manifold and suppose $N$ is an $I$-bundle over a compact surface.  Suppose $F$ is a connected orientable incompressible surface properly embedded in $N$ with $\partial F\subset\partial_vN$.  Then $F$ either is an annulus parallel to a sub-annulus of $\partial_vN$, or can be isotoped to be horizontal in $N$.  Moreover, $N$ is $\partial_vN$-small.
\end{lemma}
\begin{proof}
Since $F$ is incompressible, $\partial F$ is essential in $\partial_vN$. So after isotopy, we may assume $\partial F$ is transverse to the $I$-fibers of $N$. 
Suppose $F$ is not an annulus parallel to a sub-annulus of $\partial_vN$.  By Lemma~\ref{Lmonogon}, this means that $F$ admits no $\partial$-compressing disk with base arc in $\partial_vN$.  

Let $R$ be a vertical rectangle properly embedded in $N$, and let $I_1$ and $I_2$ be the pair of opposite edges of $R$ that are $I$-fibers of $N$ in $\partial_vN$.  Since $F$ is incompressible and $N$ is irreducible, after isotopy, we may assume $F\cap R$ contains no closed curve.  As $\partial F\subset\partial_vN$, $F\cap R$ consists of arcs with endpoints in $I_1\cup I_2$.  Since $F$ admits no $\partial$-compressing disk with base arc in $\partial_vN$, after isotopy, $F\cap R$ contains no arc with both endpoints in the same arc $I_j$ ($j=1,2$).  Thus $F\cap R$ consists of arcs connecting $I_1$ to $I_2$.  So after isotopy, $F\cap R$ is transverse to the $I$-fibers.

There are a collection of vertical rectangles in $N$ such that if we cut $N$ open along these vertical rectangles, the resulting manifold is of the form $D\times I$ where $D$ is a disk.  The conclusion above on $F\cap R$ implies that after isotopy, we may assume that the restriction of $F$ to $\partial D\times I$ is a collection of curves transverse to the $I$-fibers.  Since $F$ is incompressible, after isotopy, each component of $F\cap(D\times I)$ is a disk transverse to the $I$-fibers.  Therefore $F$ is transverse to the $I$-fibers in $N$.

The last part of the lemma follows from the conclusion above immediately.  Suppose $F$ is horizontal in $N$ as above.  
We can cut $N$ open along $F$ and obtain a manifold $N'$.  As $F$ is transverse to the $I$-fibers, $N'$ has an $I$-bundle structure induced from that of $N$.  Let $N_G$ be a component of $N'$ that contains a component $G$ of $\partial_hN$.  By our construction of $N'$, $G$ cannot be the whole of $\partial_hN_G$ and hence the $I$-bundle $N_G$ must be a product of the form $G\times I$ with $G=G\times\{0\}$.  As $F$ is connected and orientable, $F$ can be viewed as $G\times\{1\}$ and $F$ is parallel to $G$ in $N$.  Thus, in any case, $F$ is $\partial$-parallel in $N$, which means that $N$ is $\partial_vN$-small.
\end{proof}

In the next lemma, we study how a strongly irreducible Heegaard surface intersects essential annuli.  The result is well-known to experts and it is similar to \cite[Lemma 3.3]{BSS}, where the authors consider the intersection of a strongly irreducible Heegaard surface with a closed incompressible surface.

\begin{definition}\label{D33} 
Let $N$ be a compact orientable 3-manifold and let $P$ be an orientable surface properly embedded in $N$.  
We say $P$ is \emph{strongly irreducible} if $P$ has compressing disks on both sides, and each compressing disk on one side meets every compressing disk on the other side.  
\end{definition}

\begin{lemma}\label{Lann}
Let $N$ be a compact orientable irreducible 3-manifold with incompressible boundary.  Let $A$ be a collection of essential annuli properly embedded in $N$.  Suppose $A$ divides $N$ into sub-manifolds $N_1,\dots, N_k$.  Let $S$ be a strongly irreducible Heegaard surface of $N$.  Then after isotopy, $S$ is transverse to $A$, and either
\begin{enumerate}
  \item $S\cap N_i$ is incompressible in $N_i$ for each $i$, or
  \item Exactly one component of $\coprod_{i=1}^k (S\cap N_i)$ is strongly irreducible and all other components are incompressible in the corresponding sub-manifolds $N_i$.  Moreover, no component of $S\cap N_i$ is an annulus parallel to a sub-annulus of $A$ in $N_i$. 
\end{enumerate}
\end{lemma}
\begin{proof}
The proof is similar to that of \cite[Lemma 3.3]{BSS}.  
Let $V$ and $W$ be the two compression bodies in the Heegaard splitting of $N$ along $S$. 
Let $G_V$ and $G_W$ be the core graphs of the compression bodies $V$ and $W$ respectively, and let $\Sigma_V=G_V\cup\partial_-V$ and $\Sigma_W=G_W\cup\partial_-W$, such that $N-(\Sigma_V\cup\Sigma_W)\cong S\times(0,1)$.  We may assume the graph $G_V\cup G_W$ is transverse to $A$. Next we consider the sweepout $f\colon S\times I\to N$ such that $f|_{S\times(0,1)}$ is an embedding, $f(S\times\{0\})=\Sigma_V$ and $f(S\times\{1\})=\Sigma_W$.  We denote $f(S\times\{t\})$ by $S_t$.  Each $S_t$ is isotopic to $S$ if $t\in(0,1)$, and we use $V_t$ and $W_t$ to denote the two compression bodies in the Heegaard splitting along $S_t$ that contain $\Sigma_V$ and $\Sigma_W$ respectively.

Since $A$ is a collection of essential annuli properly embedded in $N$,  $A$ cannot be totally inside a compression body $V$ or $W$.  This means that $S$ cannot be isotoped disjoint from $A$.  In particular, $A\cap \Sigma_V\ne\emptyset$ and $A\cap \Sigma_W\ne\emptyset$.

The sweepout $f$ induces a height function $h\colon  A\to I$ as follows. Define $h(x) = t$ if
$x\in S_t$. We can perturb $A$ so that $h$ is a Morse function on $A-(\Sigma_V\cup\Sigma_W)$. Let $t_0<t_1<\cdots<t_n$ denote the critical values of $h$.  Since $A\cap \Sigma_V\ne\emptyset$ and $A\cap \Sigma_W\ne\emptyset$, $t_0 = 0$ and $t_n = 1$.  

For each regular value $t\in I$ of $h$, we label $t$ with the letter $V$ (resp. $W$) if there is a simple closed curve in $S_t$ which is disjoint from $A$ and bounds a compressing disk for $S_t$ in $V_t$ (resp. $W_t$).  A regular value $t$ may have no label and may be labelled both $V$ and $W$.

\begin{claimA}\label{claimA1}
For a sufficiently small $\epsilon>0$, $\epsilon$ is labelled $V$ and $1-\epsilon$ is labelled $W$.
\end{claimA}
\begin{proof}[Proof of Claim~\ref{claimA1}]
This claim follows immediately from the assumption that $G_V\cup G_W$ is transverse to $A$.
\end{proof}

\begin{claimA}\label{claimA2}
If a regular value $t$ has no label, then part (1) of the lemma holds.
\end{claimA}
\begin{proof}[Proof of Claim~\ref{claimA2}]
Suppose $t$ has no label.  We first show that a curve in $S_t\cap A$ is either essential in both $S_t$ and $A$ or trivial in both $S_t$ and $A$.  To see this, let $C$ be a curve in $S_t\cap A$.  Since $A$ is incompressible, $C$ cannot be essential in $A$ but trivial in $S_t$.  If $C$ is essential in $S_t$ but trivial in $A$, then $C$ bounds an embedded disk in $N$.  Since the Heegaard surface $S_t$ is strongly irreducible, by Scharlemann's no-nesting lemma \cite[Lemma 2.2]{S}, $C$ bounds a compressing disk in either $V_t$ of $W_t$.  Moreover, we can move $C$ slightly off $A$, and this means that $t$ is labelled $V$ or $W$, a contradiction to the hypothesis of the claim.  Thus each curve in $S_t\cap A$ is either essential or trivial in both $S_t$ and $A$.

Let $c$ be a curve in $S_t\cap A$ that is an innermost trivial curve in $A$.  So $c$ bounds a disk $d_c$ in $A$ with $d_c\cap S_t=c$.  By our conclusion above, $c$ also bounds a disk $d_c'$ in $S_t$.  Since $c$ is innermost, $d_c\cup d_c'$ is an embedded 2-sphere.  As $N$ is irreducible, $d_c\cup d_c'$ bounds a 3-ball.  Hence we can isotope $S_t$ by pushing $d_c'$ across this 3-ball and eliminate this intersection curve $c$.  After finitely many such isotopies, we get surface $S_t'$ such that $S_t'$ is isotopic to $S_t$, and $S_t'\cap A$ consists of curves essential in both $S_t'$ and $A$.  Next we show that $S_t'\cap N_i$ is incompressible in $N_i$ for each $i$.

Suppose a component of $S_t'\cap N_i$ is compressible in $N_i$ and let $\gamma$ be an essential curve in $S_t'\cap N_i$ bounding an embedded disk in $N_i$.  Since $S_t'\cap A$ consists of essential curves in $S_t'$, $\gamma$ is an essential curve in the Heegaard surface $S_t'$.  Moreover, since the isotopy from $S_t$ to $S_t'$ is simply pushing some disks across $A$, we may view $\gamma$ as an essential curve in $S_t$ disjoint from these disks in the isotopy changing $S_t$ to $S_t'$.  In particular, we may assume $\gamma\subset S_t$ and $\gamma\cap A=\emptyset$.  As $\gamma$ bounds an embedded disk in $N_i$, by the no-nesting lemma \cite[Lemma 2.2]{S}, $\gamma$ bounds a compressing disk in $V_t$ or $W_t$, which contradicts the hypothesis that $t$ has no label.  Therefore, $S_t'\cap N_i$ is incompressible in $N_i$ for each $i$ and part (1) of the lemma holds.
\end{proof}

\begin{claimA}\label{claimA3}
If a regular value $t$ is labelled both $V$ and $W$, then part (2) of the lemma holds.
\end{claimA}
\begin{proof}[Proof of Claim~\ref{claimA3}]
Suppose $t$ is labelled both $V$ and $W$. 
We first show that a curve in $S_t\cap A$ is either essential in both $S_t$ and $A$ or trivial in both $S_t$ and $A$.  To see this, let $C$ be a curve in $S_t\cap A$.  Since $A$ is incompressible in $N$, $C$ cannot be essential in $A$ but trivial in $S_t$.  If $C$ is essential in $S_t$ but trivial in $A$, then $C$ bounds an embedded disk in $N$.  Since the Heegaard surface $S_t$ is strongly irreducible, by Scharlemann's no-nesting lemma, $C$ bounds a compressing disk in either $V_t$ or $W_t$.  Suppose $C$ bounds a compressing disk in $V_t$.  Since $t$ is also labelled $W$, $S_t$ has a compressing disk $\Delta$ in $W_t$ such that $\partial \Delta$ is disjoint from $A$.  As $C\subset A$, $C\cap\partial \Delta=\emptyset$ and this contradicts our hypothesis that $S_t$ is strongly irreducible.  Thus a curve in $S_t\cap A$ is either essential in both $S_t$ and $A$ or trivial in both $S_t$ and $A$. 

Now same as the proof of Claim~\ref{claimA2}, we can isotope $S_t$ to eliminate all the curves in $S_t\cap A$ that are trivial in both $S_t$ and $A$.  We use $S_t'$ to denote the surface after the isotopy.  This isotopy also changes $V_t$ and $W_t$ to  $V_t'$ and $W_t'$ which are the two compression bodies in the splitting of $N$ along $S_t'$.   
Since the isotopy only moves disks in $S_t$ across $A$ and since $t$ has both labels, there must be essential curves $\gamma_V$ and $\gamma_W$ in $S_t'$ that are disjoint from $A$ and bound compressing disks in $V_t'$ and $W_t'$ respectively.

As the Heegaard splitting is strongly irreducible, $\gamma_V\cap\gamma_W\ne\emptyset$ and this means that $\gamma_V$ and $\gamma_W$ lie in the same component $\Sigma$ of $S_t'\cap N_i$ for some $i$.  Moreover, since $A$ is incompressible in $N$, the compressing disks bounded by $\gamma_V$ and $\gamma_W$ can be isotoped disjoint from $A$.  Hence $\Sigma$ has compressing disks on both sides in $N_i$.

If a component $A'$ of $S_t'\cap N_j$ is an annulus in $N_j$ parallel to a sub-annulus of $A$, then we can isotope $S_t'$ by pushing $A'$ to the other side.  Note that the strongly irreducible component $\Sigma$ above cannot be an annulus.  So, even though $\Sigma$ may be changed by this isotopy, it still has compressing disks on both sides after the isotopy.  Since $S_t'\cap A$ consists of essential curves in $A$, every compressing disk for $\Sigma$ is also a compressing disk for $S_t'$.  Since the Heegaard surface $S_t'$ is strongly irreducible, this means that $\Sigma$ must be strongly irreducible in $N_i$.  Thus after isotopy, we may assume that (1) for any $j$, no component of $S_t'\cap N_j$ is an  annulus in $N_j$ parallel to a sub-annulus of $A$, and (2) $\coprod_{i=1}^k (S_t'\cap N_i)$ has a strongly irreducible component $\Sigma$.

Let $P$ ($P\ne\Sigma$) be any other component of $S_t'\cap N_j$.  If $P$ is compressible in $N_j$, then there is a curve $\gamma_P\subset P$ which is essential in $P$ and bounds a compressing disk for $P$ in $N_j$.  As $S_t'\cap A$ consists of curves essential in $S_t'$, $\gamma_P$ is essential in $S_t'$.  So $\gamma_P$ bounds a compressing disk for $S_t'$ in $V_t'$ or $W_t'$.  As $\gamma_P$ and $\gamma_V\cup\gamma_W$ lie in different components of $\coprod_{i=1}^k (S_t'\cap N_i)$, $\gamma_P\cap(\gamma_V\cup\gamma_W)=\emptyset$ and this contradicts our hypothesis that $S_t'$ is strongly irreducible.  Thus $P$ is incompressible and part (2) of the lemma holds.
\end{proof}

By Claim~\ref{claimA1}, as a regular value $t$ changes from 0 to 1, its label changes from $V$ to $W$.  Suppose the lemma is false, then by Claim~\ref{claimA2} and Claim~\ref{claimA3}, each regular value has exactly one label, and this implies that there must be a critical value $t_k$ such that $t_k-\epsilon$ is labelled $V$ and $t_k+\epsilon$ is labelled $W$ for a sufficiently small $\epsilon>0$.  So $S_{t_k}\cap A$ contain a single tangency.  

If this tangency is a center tangency, then the change from $S_{t_k-\epsilon}$ to  $S_{t_k+\epsilon}$ is an isotopy that eliminates or creates an intersection curve with $A$ that is trivial in both surfaces.  This means that $t_k-\epsilon$ and  $t_k+\epsilon$ have the same label, contradicting our assumption above.  Thus $S_{t_k}\cap A$ contains a saddle tangency.  So one component of $A\cap S_{t_k}$, denoted by $\Gamma$, is a figure 8 curve, and all other components are simple closed curves.  Note that since $A$ consists of annuli, at least one component of $A-\Gamma$ is a disk.

Let $Q(A)$ be a small product neighborhood of $A$ in $N$, where $Q(A)=A\times I$ with $A=A\times\{\frac{1}{2}\}$ and $\partial A\times I\subset\partial N$.  Let $N_i'$ be the component of $\overline{N-Q(A)}$ that lies in $N_i$.    

\begin{claimA}\label{claimA4}
For each $i$, $S_{t_k}\cap N_i'$ does not contain a curve that is essential in $S_{t_k}$ but bounds an embedded disk in $N$.  
\end{claimA}
\begin{proof}[Proof of Claim~\ref{claimA4}]
Suppose on the contrary that there is such a curve $\gamma$.  By the no-nesting lemma \cite[Lemma 2.2]{S}, $\gamma$ bounds a compressing disk in either $V_{t_k}$ or $W_{t_k}$.  Note that if $\epsilon$ is sufficiently small, then $S_{t_k\pm\epsilon}\cap N_i'$ is parallel to $S_{t_k}\cap N_i'$ in $N_i'$.  So $S_{t_k\pm\epsilon}\cap N_i'$ also contains such a curve $\gamma$ and this means that $t_k-\epsilon$ and $t_k+\epsilon$ are both labelled $V$ or $W$, contradicting the assumption above that $t_k-\epsilon$ and $t_k+\epsilon$ have different labels.
\end{proof}

Next we consider $S_{t_k}\cap Q(A)$.  Let $P$ be the component of $S_{t_k}\cap Q(A)$ that contains the saddle tangency.  By choosing the product neighborhood $Q(A)$ to be sufficiently small, we may assume $P$ is a pair of pants and any other component of $S_{t_k}\cap Q(A)$ is a vertical annulus in the product $A\times I=Q(A)$.  

If a curve in $S_{t_k}\cap (A\times\partial I)$ is trivial in $A\times\partial I$, then by Claim~\ref{claimA4}, it must be trivial in $S_{t_k}$.  If a curve in $S_{t_k}\cap (A\times\partial I)$ is essential in $A\times\partial I$, since $A$ is incompressible, it must also be essential in $S_{t_k}$. 

Now we isotope $S_{t_k}$ to eliminate the curves in $S_{t_k}\cap (A\times\partial I)$ that are trivial in both $A\times\partial I$ and $S_{t_k}$.  This isotopy is equivalent to the operation that first cuts $S_{t_k}$ open along curves in $S_{t_k}\cap (A\times\partial I)$ that are trivial in $A\times\partial I$, then caps off the boundary curves  using disks parallel to the subdisks of $A\times\partial I$ bounded by these curves and discards any resulting 2-sphere components.  As $N$ is irreducible, after a small perturbation, we obtain a surface $S_{t_k}'$ isotopic to $S_{t_k}$, and $S_{t_k}'\cap (A\times\partial I)$  consists of curves essential in both $A\times\partial I$ and $S_{t_k}'$.

Recall that, since $A$ is an annulus, at least one component of $A-\Gamma$ is a disk, where $\Gamma$ is the figure 8 curve in $A\cap S_{t_k}$ containing the saddle tangency.  So at least one boundary curve of the pair of pants $P$ is trivial in $A\times\partial I$.  Hence the operation above either eliminates $P$, or changes $P$ to a vertical annulus in $A\times I=Q(A)$, or changes $P$ to a $\partial$-parallel annulus in $A\times I$.  Thus $S_{t_k}'\cap Q(A)$ consists of essential vertical annuli in $Q(A)$ and at most one $\partial$-parallel annulus $P'$ in $Q(A)$ which comes from $P$.  Moreover, the possible $\partial$-parallel annulus $P'$ is incompressible in $Q(A)$.

Next we show $S_{t_k}'\cap N_i'$ is incompressible in $N_i'$ for all $i$. 
Suppose on the contrary that $S_{t_k}'\cap N_i'$ is compressible in $N_i'$ and let $\gamma$ is a curve in $S_{t_k}'\cap N_i'$ bounding a compressing disk. As $S_{t_k}'\cap (A\times\partial I)$ consists of curves essential in both $S_{t_k}'$ and $A\times\partial I$, $\gamma$ is essential in $S_{t_k}'$.  Since the operation of getting  $S_{t_k}'$ from $S_{t_k}$ is simply replacing disks in $S_{t_k}$ by other disks, we may view $\gamma$ as an essential curve in $S_{t_k}$ which is disjoint from these disks.  This contradicts Claim~\ref{claimA4}.  So $S_{t_k}'\cap N_i'$ is incompressible in $N_i'$ for all $i$.

If every component of $S_{t_k}'\cap Q(A)$ is a vertical annulus in $Q(A)$, then $S_{t_k}'\cap N_i$ is isotopic to $S_{t_k}'\cap N_i'$ and hence is incompressible in $N_i$ for each $i$, which means that part (1) of the lemma holds.

Suppose $S_{t_k}'\cap Q(A)$ contains a $\partial$-parallel annulus $P'$ as above.  Without loss of generality, we may suppose $\partial P'\subset A\times\{0\}=A_0$.  Clearly $A_0$ is isotopic to $A$ and $A_0$ divides $N$ into a collection of sub-manifolds $N_1'',\dots N_k''$ with each $N_i''$ isotopic to $N_i'$ and $N_i$.  So a component of $S_{t_k}'\cap N_i''$ is either $P'$ (which is a $\partial$-parallel but incompressible annulus), or isotopic to a component of $S_{t_k}'\cap N_i'$ which is also incompressible.  Therefore, by regarding $A_0$ and $N_i''$ as $A$ and $N_i$ respectively, we see that part (1) of the lemma also holds in this case.
\end{proof}

\begin{definition}\label{Dtubing}
Let $N$ be a compact 3-manifold with boundary, and let $A$ be either a torus component of $\partial N$ or an annulus in $\partial N$.  Let $F$ be a properly embedded surface in $N$ and suppose $\partial F\cap A$ is a collection of essential curves in $A$.  For any two adjacent curves $\gamma_1$ and $\gamma_2$ in $\partial F\cap A$, as shown in the schematic picture in Figure~\ref{Ftubing}, we can first glue the sub-annulus of $A$ bounded by $\gamma_1\cup\gamma_2$ to $F$ and then push it into $\Int(N)$ to make the resulting surface $\widehat{F}$ properly embedded in $N$.  We say that $\widehat{F}$ is obtained by \emph{tubing} along $A$.  If $|\partial F\cap A|$ is even, then one can apply the tubing operation on $F$ multiple times to obtain a closed surface in $N$.   Conversely, given a surface $\widehat{F}$, if there is an embedded annulus $\Gamma$ in $N$ with one boundary circle an essential curve in $A$, the other boundary circle in $\widehat{F}$ and $\Int(\Gamma)\cap\widehat{F}=\emptyset$, then we can cut $\widehat{F}$ open along the curve $\Gamma\cap\widehat{F}$ and glue two parallel copies of $\Gamma$ along the resulting pair of boundary curves.  The resulting surface has two more boundary circles in $A$ than $\widehat{F}$.  This is the converse operation of tubing, see Figure~\ref{Ftubing}, and we say the resulting surface is obtained by an \emph{annulus-compression} on $\widehat{F}$ along $A$. 
\end{definition}

\begin{figure}
  \centering
\psfrag{t}{tubing}
\psfrag{a}{annulus-compression}
\psfrag{A}{$A$}
\psfrag{F}{$F$}
\psfrag{h}{$\widehat{F}$}
  \includegraphics[width=4in]{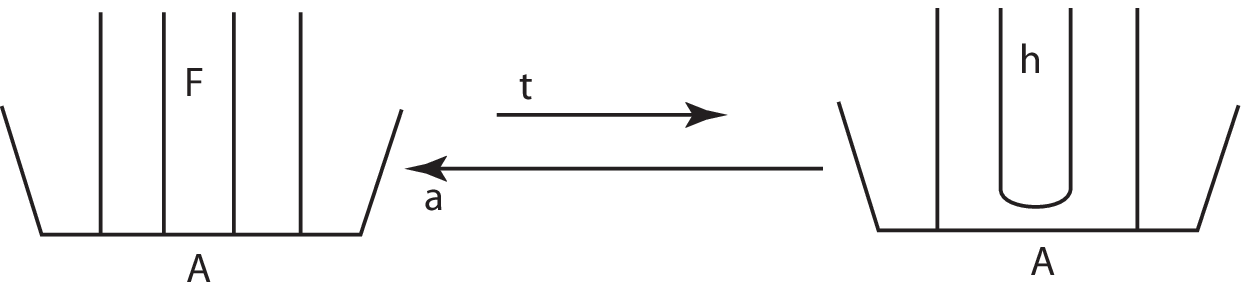}
  \caption{}\label{Ftubing}
\end{figure}

For any $\partial$-parallel surface $R$ in $N$, we use $P(R)$ to denote the region of isotopy between $R$ and $\partial N$, i.e., (1) if $R$ is a closed surface, then $P(R)$ is a product region of the form $R\times I$ in $N$ with $R\times\{0\}=R$ and $R\times\{1\}\subset\partial N$, and (2) if $R$ has boundary, then $P(R)$ is the pinched product region bounded by $R$ and the sub-surface of $\partial N$ which is bounded by $\partial R$ and isotopic to $R$ relative to $\partial R$.  We say a collection of disjoint $\partial$-parallel surfaces $R_1,\dots, R_m$ in $N$ are \emph{non-nested} if $P(R_1),\dots,P(R_m)$ are disjoint in $N$.

In the next lemma, we show that strongly irreducible surfaces in a small manifold have some nice properties.

\begin{lemma}\label{Lstr}
Let $N$ be a compact orientable irreducible 3-manifold.  Let $H$ be a component of $\partial N$ and let $A$ be an annulus in $H$ such that $H-A$ is connected.  Suppose $N$ is both small and $A$-small.  Let $S$ be a connected strongly irreducible surface properly embedded in $N$ with $\partial S\subset A$, and suppose $S$ does not lie in a collar neighborhood of $\partial N$ in $N$.  We use plus and minus signs to denote the two sides of $S$ and suppose $H-A$ is on the plus side of $S$.  Let $S_+$ and $S_-$ be the surfaces obtained by maximally compressing $S$ on the plus and minus sides of $S$ respectively and deleting any resulting 2-sphere components.  Then 
\begin{enumerate}
  \item $S_\pm$ is a collection of non-nested $\partial$-parallel surfaces in $N$
  \item The closure of each component of $\partial N-\partial S$ on the $\pm$-side of $S$ is parallel to a component of $S_\pm$, see Figure~\ref{F2sides}.
\end{enumerate} 
Furthermore, the closed surface obtained by tubing $\partial S$ on the minus-side using sub-annuli of $A$, as illustrated in Figure~\ref{F2sides}, is a Heegaard surface of $N$. In particular, $\chi(S)\le 2-2g(N)$.
\end{lemma}

\begin{figure} 
  \centering
\psfrag{S}{$S$}
\psfrag{A}{$A$}
\psfrag{+}{$S_+$}
\psfrag{-}{$S_-$}
\psfrag{h}{$\widehat{S}$}
\psfrag{c}{\shortstack{compressing on \\ the plus side}}
\psfrag{d}{\shortstack{compressing on \\ the minus side}} 
\psfrag{t}{\shortstack[l]{tubing \\ on the \\ minus side}}
  \includegraphics[width=3.5in]{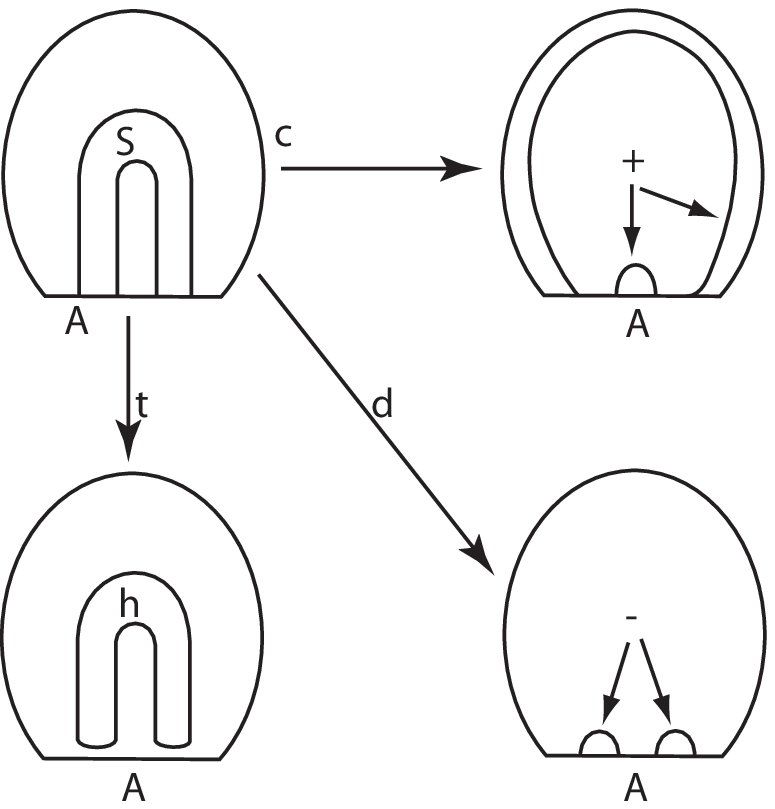}
  \caption{}
  \label{F2sides}
\end{figure}

\begin{proof}
If $S$ is non-separating, after maximally compressing $S$ on both sides, we get an incompressible surface in $N$ which has a non-separating component.  This contradicts that $N$ is both small and $A$-small.  So $S$ must be separating. Since $H-A$ is connected, this means that $|\partial S|$ is an even number.

Now we consider $S_\pm$.  A surface after a compression on $S$ naturally inherits plus and minus sides from $S$. 
Since we perform a maximal number of compressions, the surface $S_+$ (resp. $S_-$) is incompressible on the plus side (resp. minus side).  Moreover, since $S$ is strongly irreducible, by \cite{CG} (also see \cite[Proof of Lemma 5.5]{S1}), $S_+$ (resp. $S_-$) is also incompressible on the minus side (resp. plus side).  Thus $S_\pm$ is incompressible in $N$.
As $N$ is both small and $A$-small, each component of $S_\pm$ is $\partial$-parallel in $N$.  

Next we consider $S$ and $S_\pm$ at the same time.  
As the compressions on $S$ occur on curves in $\Int(S)$, after some isotopy on $S_\pm$, we may suppose $\partial S_\pm=\partial S=S\cap S_\pm$.
  Since $S$ is connected, by our construction, there is a connected region $W_\pm$ between $S$ and $S_\pm$ such that $\partial W_\pm=S\cup S_\pm$, see Figure~\ref{Fmono}(b) for a schematic picture, and $W_\pm$ can be viewed as a region obtained by adding 2-handles and 3-handles on the $\pm$-side of $S$.  Let $W=W_+\cup W_-$.  So $W$ is the component of $\overline{N-(S_+\cup S_-)}$ such that $\partial W=S_+\cup S_-$ and $\partial W\cap\partial N=\partial S=\partial S_\pm$.

Let $R$ be a component of $S_\pm$ and let $P(R)$ be the (pinched) product region between $R$ and $\partial N$ described before Lemma~\ref{Lstr}.  By our construction of $W$, $W$ lies either inside $P(R)$ or outside $P(R)$.  If $W\subset P(R)$, then $S\subset P(R)$ and hence $S$ lies in a collar neighborhood of $\partial N$ in $N$, contradicting our hypothesis on $S$.  Thus $W$ lies outside $P(R)$ and hence $W\cap P(R)=R$.  Since $\partial W=S_+\cup S_-$, this means that the union of $W$ and all the (pinched) product regions $P(R)$ (for all the components $R$ of $S_+$ and $S_-$) is the whole of $N$.  Therefore, both part (1) and part(2) of the lemma hold.

Now we prove the last part of the lemma.  Recall that $H-A$ lies on the plus side of $S$ and $|\partial S|$ is even. 
So, as illustrated in Figure~\ref{F2sides}, we can obtain a closed surface $\widehat{S}$ by tubing $\partial S$ on the minus side using sub-annuli of $A$ bounded by $\partial S$ on the minus side of $S$.   By part (2) of the lemma, each component of $S_+$ is either a closed surface parallel to a component of $\partial N$ on the plus side of $S$ or a surface parallel to $H-\Int(A)$ or an annulus parallel to a sub-annulus of $A$. Thus the same tubing operation on $S_+$ (i.e., tubing $\partial S_+$ on the minus-side of $S_+$ using sub-annuli of $A$ bounded by $\partial S_+$ on the minus side) yields a closed surface parallel to $H$ (plus all the closed-surface components of $S_+$).  Hence the maximal compression on $\widehat{S}$ on the plus side changes $\widehat{S}$ into a collection of non-nested $\partial$-parallel surfaces.

Since $H-A$ lies on the plus side of $S$, the components of $S_-$ with boundary in $A$ are a collection of non-nested $\partial$-parallel annuli, and all other components of $S_-$ are closed surfaces parallel to the components of $\partial N$ on the minus side of $S$.  So, tubing $S_-$ on the minus side of $S_-$ changes the annulus components of $S_-$ into a collection of tori bounding disjoint solid tori.  Thus the maximal compression of $\widehat{S}$ on the minus side of $\widehat{S}$ changes $\widehat{S}$ into a collection of non-nested $\partial$-parallel surfaces parallel to those components of $\partial N$ on the minus side of $S$.  This means that $\widehat{S}$ is a Heegaard surface of $N$.

As $\widehat{S}$ is a Heegaard surface of $N$, we have $\chi(S)=\chi(\widehat{S})=2-2g(\widehat{S})\le 2-2g(N)$.
\end{proof}

We conclude this section with the following technical lemma on the Heegaard genus of an annulus sum of two manifolds.  This lemma will be used in section~\ref{Sgenus} to estimate Heegaard genus.

\begin{lemma}\label{L2}
Let $M$ be a compact orientable irreducible 3-manifold with incompressible boundary, and let $A$ be an essential separating annulus properly embedded in $M$.  Suppose $A$ divides $M$ into two sub-manifolds $M_1$ and $M_2$.  We view $A\subset \partial M_1$ and $A\subset\partial M_2$.  Let $H$ be the component of $\partial M_1$ that contains $A$ and suppose $H-A$ is connected.  Suppose $M_1$ is both small and $A$-small.  Let $S$ be a minimal-genus Heegaard surface of $M$ and let $\Sigma_S$ be the union of all the incompressible surfaces and strongly irreducible Heegaard surfaces in an untelescoping of $S$.  Suppose 
\begin{enumerate}
  \item for each incompressible surface $F_s$ in $\Sigma_S$, $F_s\cap M_i$ ($i=1,2$) is incompressible in $M_i$ and no component of $F_s\cap M_i$ is an annulus parallel in $M_i$ to a sub-annulus of $A$.
  \item each component of $\Sigma_S\cap M_1$ is either incompressible or strongly irreducible in $M_1$, 
  \item $\Sigma_S\cap M_1$ contains a strongly irreducible component $\Sigma$ such that $\partial\Sigma$ consists of two circles in $A$ and $\Sigma$ does not lie in a product neighborhood of $\partial M_1$ in $M_1$. 
\end{enumerate} 
Then $g(M)\ge g(M_1)+g(M_2)-1$.
\end{lemma}

The basic idea of the proof of Lemma~\ref{L2} is to use $S$ to construct a Heegaard surface for $M_2$ and then use Lemma~\ref{Lstr} to compute the genus.  For example, if $S\cap M_1=\Sigma$, then by tubing the two boundary circles of $S\cap M_i$ ($i=1,2$) along $A$, we obtain a Heegaard surface for $M_i$ (see below for detailed proof), and the inequality in the lemma follows from a simple calculation of the Euler characteristic.  However, in a more general situation, we will need to modify the splitting.

\begin{proof}[Proof of Lemma~\ref{L2}]
Suppose $M=N_0\cup_{F_1}\cup\cdots\cup_{F_{m}} N_m$ is the decomposition in the untelescoping of $S$ as in Theorem~\ref{TST}, where each $F_i$ is incompressible and each component of $N_i$ has a strongly irreducible Heegaard surface which is a component of $\Sigma_i$.  For simplicity, we assume each block $N_i$ in the untelescoping is connected and hence each $\Sigma_i$ is a strongly irreducible Heegaard surface of $N_i$.  If $N_i$ is not connected, then we can simply use the components of the $N_i$'s and $\Sigma_i$'s and the proof is the same.  We may also suppose no (component of) $N_i$ is a product $F_i\times I$ with $\Sigma_i=F_i\times\{\frac{1}{2}\}$.

Let $\Sigma_S$ be the union of the $F_i$'s and $\Sigma_i$'s in the hypotheses of the lemma.  If some $F_i$ lies totally in $M_1$, then since $M_1$ is small, $F_i$ is parallel to a component of $\partial M_1$.  If $F_i$ is parallel to $H$ ($A\subset H$), then the special strongly irreducible component $\Sigma$ lies in the product neighborhood of $H$ bounded by $H\cup F_i$, contradicting our hypotheses.  If $F_i$ is parallel to a component of $\partial M_1-H$, then by our assumption on the untelescoping above, some $N_i$ must be a product $F_i\times I$ with a non-trivial Heegaard splitting of $F_i\times I$, and this contradicts our assumption that the Heegaard surface $S$ is of minimal-genus.  Thus no $F_i$ lies totally in $M_1$.  Since $M_1$ is $A$-small, each component of $F_i\cap M_1$ must be parallel to $H-\Int(A)$.  

If a component of $\Sigma_i\cap M_1$ lies in a product neighborhood of $A$ in $M_1$, then we push this component across $A$ and into $M_2$.  So after this isotopy, we may assume no component of $\Sigma_i\cap M_1$ lies in a product neighborhood of $A$ in $M_1$.   Note that the special strongly irreducible component $\Sigma$ of $(\cup \Sigma_i)\cap M_1$ in the hypotheses is unchanged by the isotopy, since $\Sigma$ does not lie in a product neighborhood of $\partial M_1$ in $M_1$.

By Lemma~\ref{Lstr}, $\Sigma$ is separating in $M_1$.  We use plus and minus sides to denote the two sides of $\Sigma$ and suppose $H-A$ is on the plus side of $\Sigma$. 
Let $\Sigma_\pm$ be the surface obtained by maximally compressing $\Sigma$ in $M_1$ on the $\pm$-side and deleting any resulting 2-sphere component.  
Since $\partial \Sigma$ has only two boundary circles, by Lemma~\ref{Lstr}, a component of $\Sigma_-$, denoted by $\Sigma_-'$, is an annulus parallel to the sub-annulus $A_\Sigma$ of $A$ bounded by $\partial\Sigma$, and a component of $\Sigma_+$, denoted by $\Sigma_+'$, is parallel to $H-\Int(A)$.  By Lemma~\ref{Lstr}, all other components of $\Sigma_\pm$ are closed surfaces parallel to the surfaces in $\partial M_1-H$.
By our assumption on $(\cup \Sigma_i)\cap M_1$ above, no component of $(\cup \Sigma_i)\cap M_1$ lies in the solid torus bounded by $\Sigma_-'\cup A_\Sigma$, and hence each component of $(\cup \Sigma_i)\cap M_1$ (except for $\Sigma$) must lie in the product region $Q$ between $\Sigma_+'$ and $H-\Int(A)$.  We view $Q=F'\times I$ where $F'\times\{0\}=\Sigma_+'$ and $F'\times\{1\}=H-\Int(A)$.  Since $H-A$ is connected, $F'$ is connected and has two boundary circles.  By our conclusion on the $F_i$'s above, we may suppose each component of $(\cup F_i)\cap M_1$ is of the form $F'\times\{t\}$. 

Next we will show that if $(\cup \Sigma_i)\cap Q\ne\emptyset$, then we can replace a component of  $(\cup \Sigma_i)\cap Q$ by a nicer surface without increasing the genus.

Let $S'$ be a component of $(\cup\Sigma_i)\cap Q$.  By our assumption on $(\cup \Sigma_i)\cap M_1$ at the beginning, $S'$ is either incompressible or strongly irreducible, and $S'$ does not lie in a product neighborhood of $A$ in $M_1$.  So, if $S'$ is incompressible, then $S'$ is parallel to a surface of the form $F'\times\{t\}$ in $Q=F'\times I$.  Suppose $S'$ is strongly irreducible.  Similar to the proof of Lemma~\ref{Lstr} and by Lemma~\ref{LIbundle}, $S'$ must be separating in $Q$.  We can assign plus and minus sides for $S'$.  Let $S_\pm'$ be the surface obtained by maximally compressing $S'$ in $Q$ on the $\pm$-side of $S'$ and deleting any resulting 2-sphere component.  Similar to the proof of Lemma~\ref{Lstr}, $S_\pm'$ consists of incompressible surfaces in $Q=F'\times I$.  By Lemma~\ref{LIbundle}, each component of $S_\pm'$ is either an annulus parallel to a sub-annulus of $A$ or a horizontal surface isotopic to $F'\times\{t\}$.  
In particular, each component of $S_\pm'$ is $\partial$-parallel in $Q$.  

Similar to the proof of Lemma~\ref{Lstr}, after isotopy, we may assume $\partial S_\pm'=\partial S'=S'\cap S_\pm'$ and there is a component $W$ of $\overline{Q-(S_+'\cup S_-')}$ containing $S'$, such that $S_+'\cup S_-'=\partial W$, as illustrated in Figure~\ref{Fmono}(b).  By the construction, $W\cap\partial Q=\partial S'=\partial S_\pm'\subset A$.  In particular, $W$ is disjoint from $F'\times\partial I$.

Since we have assumed at the beginning of the proof that no component of $\Sigma_i\cap M_1$ lies in a product neighborhood of $A$ in $M_1$, $W$ does not lie in a product neighborhood of $A$.  This means that at least one component of $S_+'\coprod S_-'$ is horizontal in $Q=F'\times I$ (i.e., it can be isotoped into the form $F'\times\{t\}$). 
 Since $W\cap(F'\times\partial I)=\emptyset$, $S_+'\coprod S_-'$ contains at least two  horizontal components in $Q=F'\times I$.  
As each surface $F'\times\{t\}$ separates the two components of $F'\times\partial I$ and since $W$ is connected, $S_+'\coprod S_-'$ cannot have three or more horizontal components in $Q$.  Thus $S_+'\coprod S_-'$ contains exactly two horizontal components in $Q=F'\times I$ and all other components of $S_+'\coprod S_-'$ are $\partial$-parallel annuli lying between the two horizontal components of $S_+'\coprod S_-'$.  Similar to the proof of Lemma~\ref{Lstr}, since $W$ is connected, these $\partial$-parallel annuli in  $S_\pm'$  are non-nested, and as in part (2) of Lemma~\ref{Lstr}, the closure of each component of $\partial Q-\partial S'$ on the $\pm$-side of $S'$ is parallel to a component of $S_\pm'$.

Next we construct a new surface $S''$ to replace $S'$.

Without loss of generality, we may suppose $S_+'$ contains a horizontal component of the form $F'\times\{t\}$.  
Let $k$ be the number of components of $S_+'$. 
As shown in Figure~\ref{Fmove}(a, b), we can add $k-1$ tubes to $S_+'$ along $k-1$ unknotted arcs which can be isotoped into $\partial F'\times I$, and the resulting surface, which we denote by $S''$, is a connected sum of all the components of $S_+'$.  It follows from our construction of $S''$ and the properties of $S_\pm'$ that if we maximally compress $S''$ on the $\pm$-side of $S''$, we get a surface isotopic to $S_\pm'$.  Moreover, since $S'$ is connected, one has to compress $S'$ at least $k-1$ times to get a surface with $k$ components.  This implies that $g(S'')\le g(S')$.  Now we replace the component $S'$ of $(\cup \Sigma_i)\cap Q$ by $S''$.  Since we get the same surface $S_\pm'$ after maximally compressing $S'$ and $S''$ on the $\pm$-side, the resulting surface (from replacing $S'$ by $S''$) is also a Heegaard surface of the corresponding block $N_i$.  As $g(S'')\le g(S')$, the genus of the new Heegaard surface is no larger than the genus $g(\Sigma_i)$ of the corresponding $\Sigma_i$.

\begin{figure}
  \centering
\psfrag{a}{(a)}
\psfrag{b}{(b)}
\psfrag{c}{(c)}
\psfrag{A}{$A$}
\psfrag{S}{$S_+'$}
\psfrag{2}{$S''$}
  \includegraphics[width=5in]{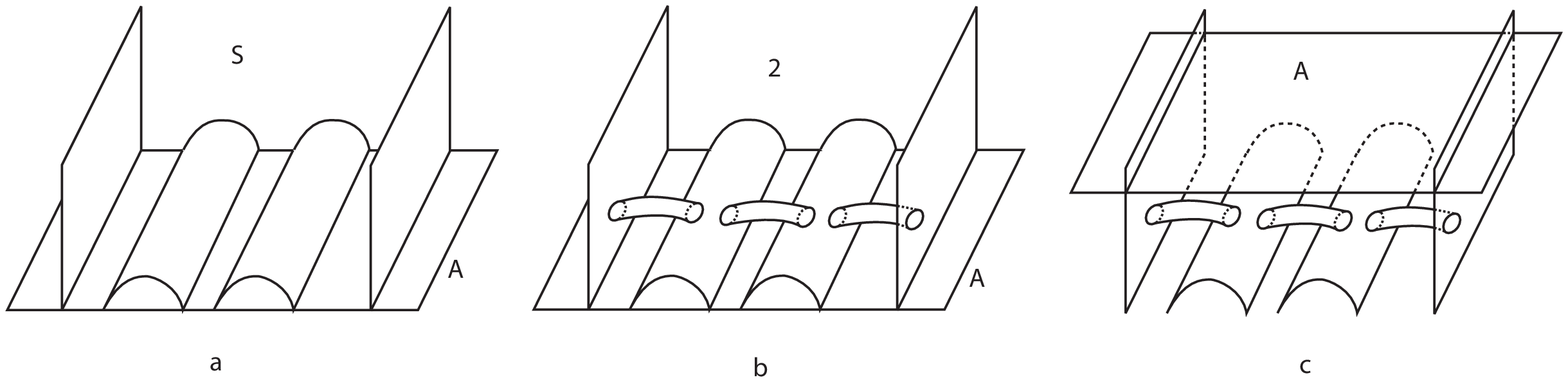}
  \caption{}
  \label{Fmove}
\end{figure}

As shown in Figure~\ref{Fmove}(c), we can isotope $S''$ by pushing the $\partial$-parallel annuli in $S_+'$ and the added tubes across $A$ and into $M_2$, and after this isotopy, $S''\cap M_1$ consists of either one horizontal surface in $Q=F'\times I$ (in the case that $S_+'$ has one horizontal component) or two horizontal surfaces (in the case that $S_+'$ has two horizontal components).  

After we perform such operations on all strongly irreducible components of $(\cup \Sigma_i)\cap Q$ as above, we obtain a new set of Heegaard surfaces $\{P_i\}$ of the blocks $N_i$'s with $g(P_i)\le g(\Sigma_i)$.  Moreover, the special surface $\Sigma$ in the hypotheses of the lemma remains a component of $(\cup P_i)\cap M_1$ and all the components of $(\cup P_i)\cap Q$ are of the form $F'\times\{t\}$ in $Q=F'\times I$.

As in part (4) of Theorem~\ref{TST}, if we amalgamate the new Heegaard surfaces $P_i$'s of the blocks $N_i$'s along the $F_i$'s, we get a new Heegaard surface of $M$ whose Euler characteristic is $\sum\chi(P_i)-\sum\chi(F_i)$.  Since $S$ can be obtained by amalgamating the $\Sigma_i$'s along the $F_i$'s, $\chi(S)=\sum\chi(\Sigma_i)-\sum\chi(F_i)$.    Since $g(P_i)\le g(\Sigma_i)$ and since $S$ is a minimal-genus Heegaard surface of $M$, we have $\chi(S)=\sum\chi(\Sigma_i)-\sum\chi(F_i)=\sum\chi(P_i)-\sum\chi(F_i)$.

Recall that $\partial\Sigma$ consists of two circles in $A$ bounding a sub-annulus $A_\Sigma$ of $A$.  Moreover, $\Sigma_-'$ is an annulus parallel to $A_\Sigma$.  So if $\Int(A_\Sigma)$ intersects some $F_i$ or $\Sigma_i$, then the solid torus bounded by $\Sigma_-'\cup A_\Sigma$ must contain a component of $F_i\cap M_1$ or $\Sigma_i\cap M_1$, which contradicts our assumption at the beginning of the proof that no component of $(\cup F_i)\cap M_1$ and $(\cup \Sigma_i)\cap M_1$ lies in a product neighborhood of $A$.  Thus $\Int(A_\Sigma)$ is disjoint from all the $F_i$'s, $\Sigma_i$'s and $P_i$'s.  Furthermore, by our assumptions on the $P_i$'s above and the conclusion on $(\cup F_i)\cap M_1$ at the beginning of the proof, every component of $(\cup F_i)\cap M_1$ and $(\cup P_i)\cap M_1$, except for $\Sigma$, is parallel to $H-\Int(A)$.   This implies that the boundary of each component of $(\cup F_i)\cap M_1$ and $(\cup P_i)\cap M_1$ is a pair of circles in $A$, and the sub-annuli of $A$ bounded by these pairs of circles are pairwise nested, with $A_\Sigma$ being the innermost annulus.

Next we use $(\cup F_i)\cap M_2$ and $(\cup P_i)\cap M_2$ to construct a generalized Heegaard splitting for $M_2$. 
As shown in Figure~\ref{Fdisk}(a), for each component $\Gamma$ of $(\cup F_i)\cap M_1$ and $(\cup P_i)\cap M_1$, we replace $\Gamma$ by an annulus that is parallel to the sub-annulus of $A$ bounded by $\partial\Gamma$.  In particular, we replace $\Sigma$ by the $\partial$-parallel annulus $\Sigma_-'$.  After pushing these annuli into $M_2$, the resulting surfaces are closed orientable surfaces in $M_2$.  This operation changes each $F_i$ and each $P_i$ into a surface in $M_2$ which we denote by $F_i'$ and $P_i'$ respectively.  By the discussion on the boundary curves of $(\cup F_i)\cap M_1$ and $(\cup P_i)\cap M_1$ above and as shown in Figure~\ref{Fdisk}(a), these $F_i'$'s and $P_i'$'s are disjoint in $M_2$.

\begin{figure}
  \centering
\psfrag{1}{$M_1$}
\psfrag{2}{$M_2$}
\psfrag{a}{(a)}
\psfrag{b}{(b)}
\psfrag{c}{(c)}
\psfrag{A}{$A$}
\psfrag{S}{$\Sigma$}
  \includegraphics[width=4in]{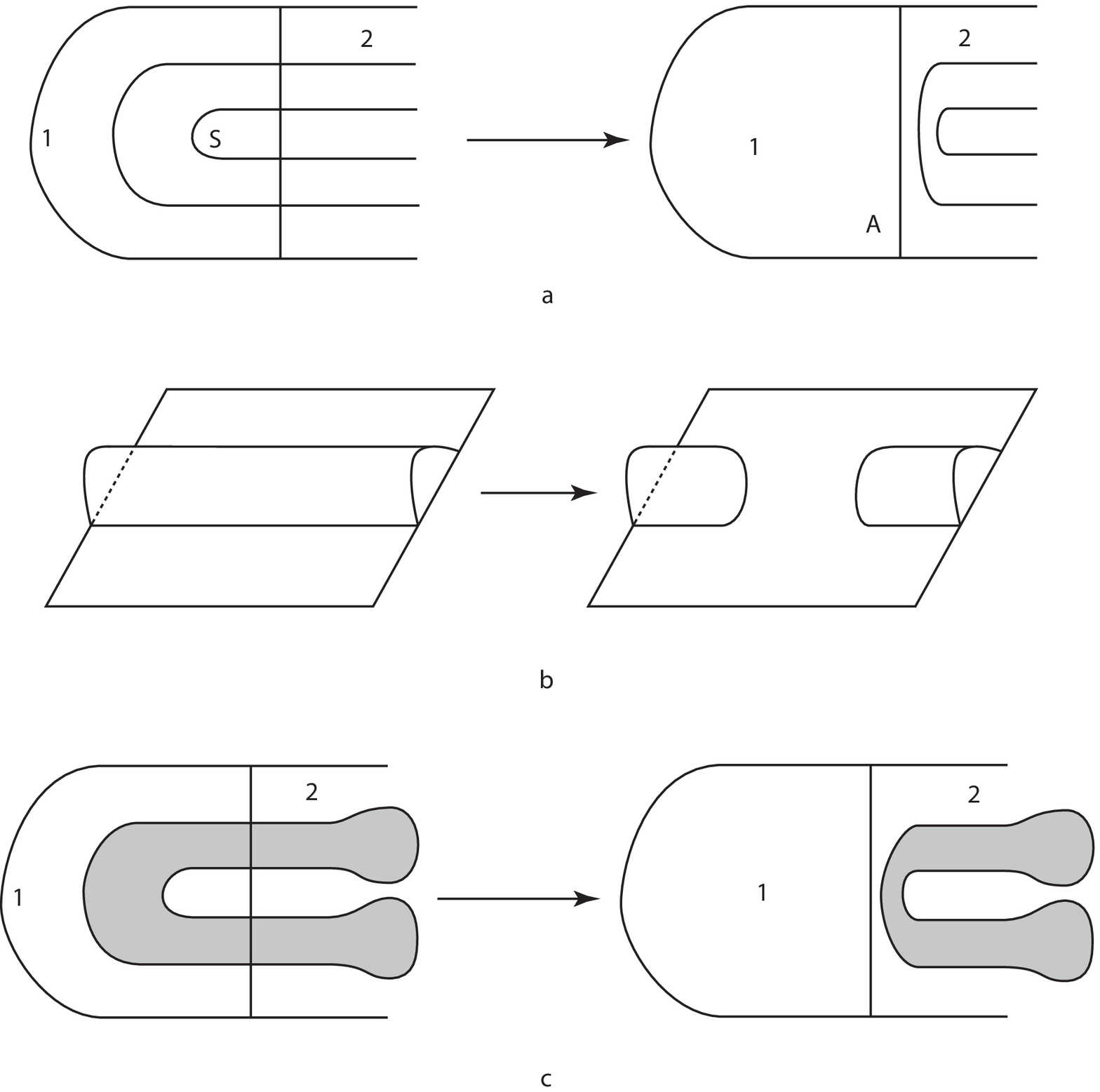}
  \caption{}
  \label{Fdisk}
\end{figure}

We claim that these $F_i'$'s and $P_i'$'s give a generalized Heegaard splitting for $M_2$.  Recall that the region $Q$ between $\Sigma_+'$ and $H-\Int(A)$ is a product $F'\times I$, and the intersection of the $F_i$'s and $P_i$'s with $Q$ are surfaces of the form $F'\times\{t\}$ which cut $Q$ into a collection of product regions of the form $F'\times J$ ($J\subset I$).  These product regions can be viewed as product regions in the compression bodies in the splittings of the $N_i$'s along the Heegaard surfaces $P_i$'s.  The reason why these $F_i'$'s and $P_i'$'s give a generalized Heegaard splitting for $M_2$ is that if one replaces a product region $F'\times J$ as above in a compression body by another product region $annulus\times I$, the resulting manifold is still a compression body.  

To prove this claim, we first consider how a compressing disk of $P_i$ intersects these product regions $F'\times J\subset Q$.  Given a compressing disk $D$ of $P_i$, if $D\cap (\partial F'\times J)$ contains a trivial arc in the annuli $\partial F'\times J$, then as shown in Figure~\ref{Fdisk}(b), we can perform a $\partial$-compression on $D$ which changes $D$ into a pair of new compressing disks for $P_i$.  This implies that, after some $\partial$-compressions as above and some isotopy, we can choose a maximal set of compressing disks for the $P_i$'s on both sides such that the intersection of each compressing disk in this set with the product regions $Q$ (if not empty) is a collection of vertical rectangles in products $F'\times J$ ($J\subset I$) above. 

The operations changing $F_i$ and $P_i$ to $F_i'$ and $P_i'$ above basically change each product region in $Q$ between the components of $F_i\cap Q$ and $P_i\cap Q$ into a new product region $annulus\times I$ in $M_2$.  As illustrated in Figure~\ref{Fdisk}(c) (the shaded region in Figure~\ref{Fdisk}(c) denotes a compressing disk), by switching vertical rectangles in $Q$ to vertical rectangles in the $annulus\times I$ regions, we can use the compressing disks for $P_i$ above to construct a maximal set of compressing disks for $P_i'$.  
This plus our construction of $\Sigma_\pm'$ implies that after compressing the $P_i'$'s on either side in $M_2$, we obtain surfaces parallel to the corresponding $F_i'$'s.  Therefore, the $P_i'$'s are Heegaard surfaces of the regions between the $F_i'$'s.  This means that these $P_i'$'s and $F_i'$'s give a generalized Heegaard splitting of $M_2$.  Note that because of replacing $\Sigma$ by $\Sigma_-'$, a compression body in this generalized Heegaard splitting may be a trivial compression body.

Similar to the calculation in \cite[Lemma 2]{SS} (see part (4) of Theorem~\ref{TST}),  the Euler characteristic of the Heegaard surface of $M_2$ obtained by amalgamating the $P_i'$'s along the $F_i'$'s above is $\sum\chi(P_i')-\sum\chi(F_i')$.  So we have $2-2g(M_2)\ge \sum\chi(P_i')-\sum\chi(F_i')$.     

Note that since each block $N_i$ in the untelescoping has incompressible boundary, $A\cap N_i$ consists of essential annuli in $N_i$.  Since an essential annulus in $N_i$ intersects every Heegaard surface of $N_i$, this implies that there is at least one component of $(\cup P_i)\cap Q$ lying between each pair of components of $((\cup F_i)\cap Q)\cup (H-\Int(A))$.  Thus $|(\cup F_i)\cap Q|\le |(\cup P_i)\cap Q|$.  

Let $P_j$ be the surface that contains $\Sigma$.  In the operation of getting the surface $P_j'$ from $P_j$ above, we replace $\Sigma$ by the annulus $\Sigma_-'$.  So $\chi(P_j)-\chi(P_j')\le\chi(\Sigma)$.  In our construction of $F_i'$ and $P_i'$, we replace each component of $(\cup P_i)\cap Q$ and $(\cup F_i)\cap Q$ by an annulus.  Since $|(\cup F_i)\cap Q|\le |(\cup P_i)\cap Q|$ and the by the operation on $\Sigma$, we have 
$$\left(\sum\chi(P_i)-\sum\chi(P_i')\right)-\left(\sum\chi(F_i)-\sum\chi(F_i')\right)\le\chi(\Sigma).$$   
Thus we have 
$(\sum\chi(P_i)-\sum\chi(F_i))-(\sum\chi(P_i')-\sum\chi(F_i'))=(\sum\chi(P_i)-\sum\chi(P_i'))-(\sum\chi(F_i)-\sum\chi(F_i'))\le\chi(\Sigma)$.

Recall we have shown earlier that $\chi(S)=\sum\chi(\Sigma_i)-\sum\chi(F_i)=\sum\chi(P_i)-\sum\chi(F_i)$ and $2-2g(M_2)\ge \sum\chi(P_i')-\sum\chi(F_i')$.   
So we have $\chi(S)-(2-2g(M_2))\le(\sum\chi(P_i)-\sum\chi(F_i))-(\sum\chi(P_i')-\sum\chi(F_i'))\le\chi(\Sigma)$ and hence $\chi(S)-(2-2g(M_2))=2g(M_2)-2g(S)\le\chi(\Sigma)$.  By Lemma~\ref{Lstr}, $\chi(\Sigma)\le 2-2g(M_1)$.  This means that $2g(M_2)-2g(S)\le 2-2g(M_1)$ and hence $g(S)\ge g(M_1)+g(M_2)-1$.
\end{proof}

\section{The construction of $X$}\label{SX}

Our first manifold $X$ is the exterior of a graph $G$ in $S^3$ constructed as follows.

We first take a 2-bridge knot $K$ in $S^3$ and let $S$ be a 2-bridge sphere with respect to $K$.  Let $B_+$ and $B_-$ be the two 3-balls in $S^3$ bounded by $S$.  So $K\cap B_\pm$ is a pair of trivial arcs in the 3-ball $B_\pm$.   Let $\alpha$ be a component of $K\cap B_+$.  Let $\beta$ be an arc in the bridge sphere $S$ with $\partial\beta=\partial\alpha=\beta\cap K$.  Note that if we slightly push $\beta$ into $B_+$, we get a 2-bridge presentation for the knot $(K-\alpha)\cup\beta$.  We may choose the slope of $\beta$ in the 4-punctured sphere $S-K$ so that 
\begin{enumerate}
  \item  $(K-\alpha)\cup\beta$ is a different 2-bridge knot from $K$, and  
  \item $\beta$ is not an unknotting tunnel of the 2-bridge knot $K$ (see \cite{K} for the classification of unknotting tunnels of 2-bridge knots).  
\end{enumerate} 
Let $G=K\cup\beta$ and $X=S^3-N(G)$, where $N(G)$ is an open neighborhood of $G$ in $S^3$.

Note that, although $\alpha\cup\beta$ is a trivial knot in $S^3$, $\alpha\cup\beta$ does not bound an embedded disk $D$ with $\Int(D)\cap K=\emptyset$, because if such a disk $D$ exists, then $\beta$ is isotopic to $\alpha$ (fixing $\partial\beta$), which implies that $(K-\alpha)\cup\beta$ and $K$ are the same 2-bridge knot, contradicting our assumptions on $\beta$.

Let $c\subset\partial X$ be a meridional curve for the arc $K-\alpha$, i.e., $c$ bounds a disk $D_c$ in $S^3$ such that $D_c\cap X=\partial D_c=c$ and $D_c\cap G$ is a single point in $K-\alpha$.  Let $A\subset\partial X$ be an annular neighborhood of $c$ in $\partial X$. 

By a theorem of Hatcher and Thurston \cite{HT}, a 2-bridge knot complement does not contain non-peripheral closed incompressible surfaces.  The next lemma shows that the complement of $G$ has similar properties.  

\begin{lemma}\label{LXsmall}
$X$ is small and $A$-small.
\end{lemma}
\begin{proof}
Similar to \cite[Proof of Theorem 1]{HT}, we picture the 2-bridge knot $K$ with respect to the natural height function $h\colon  S^3\to\mathbb{R}$.  We denote each level 2-sphere $h^{-1}(r)$ by $S_r$.  We may assume $h(K)=[0,1]$ and each $S_r$ $r\in (0,1)$ is a bridge 2-sphere for $K$.  Moreover, suppose the arc $\beta$ in the construction of $G$ lies in the level sphere $S_t$ ($t\in (0,1)$), and suppose $h(\alpha)=[t,1]$.

Let $F$ be a compact orientable incompressible surface properly embedded in $X$.  Suppose either $F$ is a closed surface or $\partial F\subset A$.  Our goal is to show that $F$ is $\partial$-parallel in $X$.  After shrinking $N(G)$ to $G$, we may view $F$ as a closed surface in $S^3$ possibly with some punctures at the arc $K-\alpha$.  

We first use an argument in \cite{HT}. 
As in \cite[Proof of Theorem 1]{HT}, we may suppose the height function $h$ on $F$ is a Morse function.  We may also suppose each puncture of $F\cap G$ is a center tangency of $F$ with a level 2-sphere and suppose the level sphere $S_t$ containing $\beta$ is not a critical level.  We may view $h^{-1}((0,1))$ as a product $S^2\times(0,1)$ with each arc $K\cap(S^2\times(0,1))$ being a vertical arc $\{x\}\times(0,1)$.  The isotopy class of each essential and non-peripheral simple closed curve in a 4-punctured sphere is referred as a slope which is a number in $\mathbb{Q}\cup\{\infty\}$.  By projecting $S^2\times(0,1)$ to the same $S^2$, we may use the slopes of a fixed 4-punctured sphere to denote the isotopy classes of essential non-peripheral loops at every level $S_r-K$ ($r\in (0,1)$).

For each non-critical level $S_r$ ($r\in (0,1)$), we define the slope of this level to be the slope of any essential and non-peripheral circle of $S_r\cap F$ in the 4-punctured sphere $S_r-K$, if there is such a circle (if there are several such circles, they must have the same slope).  The slope (if defined) at a level can change only at a level of saddle of $F$.  When passing a saddle, either one level curve splits into two level curves or two level curves are joined into one level curve.  In either case, the three level curves can be projected to be disjoint curves in a common level 2-sphere.  As there cannot be two disjoint circles of different slopes in a 4-punctured sphere, the slope cannot change at a saddle, except to become undefined.

Our difference from \cite{HT} is the special level $S_t$ which contains the arc $\beta$.  Let $r_\beta$ be the slope of an essential curve around $\beta$ in $S_t-K$.  As $F\cap\beta=\emptyset$, if $F\cap S_t$ contains a circle that is essential and non-peripheral in $S_t-K$, then the slope at $S_t$ must be $r_\beta$.

The first possibility is that the slope at every non-critical level $S_r$ with $t<r<1$ is defined.  Since slope does not change at saddles of $F$, this means that the slope at $S_{1-\epsilon}$ is $r_\beta$ for a small $\epsilon>0$.  

Let $B_+$ be the 3-ball bounded by $S_t$ containing $\alpha$.  Then the two arcs $K\cap B_+$ can be isotoped into a pair of disjoint arcs in $S_t$. 
Let $\pi(\alpha)$ be the arc in $S_t$ isotopic to $\alpha$ in $B_+$ as above and let $r_\alpha$ be the slope of an essential curve around $\pi(\alpha)$ in $S_t-K$.  If $\epsilon$ is sufficiently small ($\epsilon>0$), then any essential non-peripheral curve of $F\cap S_{1-\epsilon}$ must have slope $r_\alpha$ in the 4-puncture sphere $S_{1-\epsilon}-K$.  
This implies that, if the slope at every non-critical level $S_r$ with $t<r<1$ is defined, then the slope $r_\beta$ must be the same as the slope $r_\alpha$ and hence $K$ and $(K-\alpha)\cup\beta$ are the same 2-bridge knot, contradicting our assumptions on $\beta$.  Thus there must be a non-critical level $S_a$ with $t<a<1$ such that $F\cap S_a$ consists of trivial and peripheral curves in $S_a-K$.

Similarly, since $(K-\alpha)\cup\beta$ is a 2-bridge knot, as in \cite[Proof of Theorem 1]{HT} and the argument above, there is also a non-critical level $S_b$ with $0<b<t$ such that $F\cap S_b$ consists of trivial and peripheral curves in $S_b-K$.

Since $F-G$ is incompressible, if a circle in $F\cap S_a$ is trivial in $S_a-K$, then it must be trivial in $F-G$ and hence we can perform an isotopy on $F-G$ to remove this circle.  Thus after some isotopies, we may assume $F\cap S_a$ and $F\cap S_b$ consist of peripheral curves in $S_a-K$ and $S_b-K$ respectively around the punctures.  

Let $B_a$ and $B_b$ be the 3-balls in $S^3$ bounded by $S_a$ and $S_b$ respectively and that do not contain the level sphere $S_t$.  Since $K$ is a 2-bridge knot and $\beta$ lies outside $B_a$, $B_a\cap G$ is a pair of trivial arcs in $B_a$.  There is a disk $D$ properly embedded in $B_a-G$ that separates the two arcs in $B_a\cap G$.  Since $F\cap S_a$ consists of peripheral curves in $S_a-K$, we may assume $\partial D\cap F=\emptyset$ in $S_a$.  Since $F-G$ is incompressible and $\partial D\cap F=\emptyset$, after some isotopy, we may assume $D\cap F=\emptyset$.  The disk $D$ divides $B_a$ into a pair of 3-balls and $B_a\cap G$ is a pair of unknotted arcs in the pair of 3-balls.  So each component of $B_a-D$ may be viewed as a tubular neighborhood of a component of $B_a\cap G$ in $S^3$.  Since $F\cap S_a$ consists of peripheral circles in $S_a-K$ and since $F\cap D=\emptyset$, we can isotope $F\cap B_a$ into a small neighborhood of $G\cap B_a$ in $B_a$. 
Similarly, there is also such a disk in $B_b$ disjoint from $F$ and separating the pair of arcs in $G\cap B_b$.  So after isotopy, $F\cap B_b$ lies in a small neighborhood of $G\cap B_b$ in $B_b$.

Next we consider the region $W$ between the two spheres $S_a$ and $S_b$, and this is our main difference from \cite[Proof of Theorem 1]{HT}.  

By our construction, $W\cong S^2\times I$ and $G\cap W$ consists of two vertical arcs of the form $\{x\}\times I\subset S^2\times I$ and an H-shaped graph which is the union of two vertical arcs in $W$ and the horizontal arc $\beta$.  

There are a pair of essential non-peripheral simple closed curves $\gamma_1$ and $\gamma_2$ in the 4-punctured sphere $S_t-K$ such that $\gamma_1$ is disjoint from $\beta$ and $\gamma_2\cap\beta$ is a single point.  Note that this means that $\gamma_1\cap\gamma_2$ consists of two points, and $\gamma_1$ and $\gamma_2$ cut $S_t$ into 4 disks, each containing a puncture of $S_t\cap K$.  Let $A_1$ and $A_2$ be the two vertical annuli in $W$ containing $\gamma_1$ and $\gamma_2$ respectively.  By the construction, $A_1\cap G=\emptyset$ and $A_2\cap G=\gamma_2\cap\beta$ is a single point in $\beta$.  As $\gamma_1\cap\gamma_2$ has two points, $A_1\cap A_2$ in a pair of $I$-fibers of $W=S^2\times I$.  Moreover, $A_1$ and $A_2$ cut $W$ into four 3-balls, each containing an arc of $K\cap W$.

Since $F\cap \partial W$ consists of peripheral curves around the punctures in $K\cap \partial W$, after isotopy, we may assume $F\cap\partial A_1=\emptyset$ and $F\cap\partial A_2=\emptyset$. 
We may also assume $F$ is transverse to $A_1$ and $A_2$ and assume the number of intersection points of $F$ with the pair of arcs $A_1\cap A_2$ is minimal up to isotopy.  
This implies that any curve of $F\cap A_i$ that is essential in $A_i$ must intersect each component of $A_1\cap A_2$ in a single point, and any curve of $F\cap A_i$ that is trivial in $A_i$ must be disjoint from $A_1\cap A_2$.  If $F\cap A_i$ contains a trivial curve $c$ which bounds a disk in $A_i-G$, then since $F-G$ is incompressible, $c$ also bounds a disk in $F-G$.  So we can isotope $F$ to eliminate all such curves.  Thus, after isotopy, we may assume that 
\begin{enumerate}
  \item  $A_1\cap F$ consists of curves essential in $A_1$, 
  \item a curve in $A_2\cap F$ is either essential in $A_2$ or a curve around the puncture $A_2\cap\beta$ and disjoint from $A_1\cap A_2$, 
  \item  each curve of $A_i\cap F$ that is essential in $A_i$ intersects each component of $A_1\cap A_2$ in a single point.  
\end{enumerate} 
By our construction of $A_1$ and $A_2$, this implies that if a curve of $A_1\cap F$ meets a curve of $A_2\cap F$, then they intersect in two points, one in each component of $A_1\cap A_2$. 

We have the following two cases to consider.

\vspace{8pt}
\noindent
Case (1). $A_1\cap F\ne\emptyset$.  
\vspace{8pt}

Let $\gamma$ be a curve in $F\cap A_1$ and let $x$ be a point in $\gamma\cap(A_1\cap A_2)$.  Let $\gamma'$ be the curve in $F\cap A_2$ containing $x$.  By our assumptions on $F\cap A_i$ above, $\gamma$ and $\gamma'$ are essential curves in $A_1$ and $A_2$ respectively.  Moreover, as in the conclusion before Case (1), $\gamma\cap\gamma'$ has two intersection points, one in each component of $A_1\cap A_2$.  As $A_1\cap\beta=\emptyset$ and $A_2\cap\beta$ is a single point, we may isotope $F$ so that $\gamma$ and $\gamma'$ lie in the same level sphere $S_r$ ($r\ne t$).  Let $N_x$ be a small neighborhood of $\gamma\cup\gamma'$ in $F$.  So $N_x$ is a 4-hole sphere.  Since $F$ is transverse to $A_1$ and $A_2$, we may isotope $F$ so that $N_x\subset S_r$.  By our construction of $A_1$ and $A_2$, $S_r-N_x$ consists of 4 disks, each containing a puncture of $S_r\cap K$.  
   Let $B_r$ be the 3-ball that is bounded by $S_r$ and does not contain $\beta$.  So $B_r\cap G$ is a pair of trivial arcs in $B_r$.  This implies that there is a simple closed curve $C$ in $N_x$ bounding a disk in $B_r$ that separates the two arcs of $B_r\cap G$.   Since $F-G$ is incompressible and $N_x\subset F-G$, $C$ must bound a disk $\Delta$ in $F-G$.  The curve $C$ cuts $N_x$ into two pairs of pants $P_x$ and $P_x'$, one on each side of $C$.  Hence either $P_x$ or $P_x'$ lies in $\Delta$.  However, this implies that a component of $\partial N_x$ lies in $\Delta$ and hence bounds a subdisk $\delta$ of $\Delta$.  On the other hand, this component of $\partial N_x$ bounds a disk $d_r$ in $S_r$ which contains exactly one puncture of $K\cap S_r$.  As $S_r$ is disjoint from $\beta$, $\delta\cup d_r$ is a (possibly immersed) 2-sphere disjoint from $\beta$ and intersecting $K$ in a single point.  This means that a meridional curve of $K$ is homotopically trivial in $S^3-K$, which is impossible.

\vspace{8pt}
\noindent
Case (2). $A_1\cap F=\emptyset$.  
\vspace{8pt}

We cut $W$ open along the vertical annulus $A_1$.  The resulting manifold consists of two 3-balls $W_1$ and $W_2$.  Suppose $W_1$ is the 3-ball that contains the H-shaped graph in $G\cap W$.   By the construction of $A_1$, $W_2\cap G$ consists of two unknotted arcs that are $\partial$-parallel in $W_2$. 

We first consider $W_2$.  As $F\cap \partial W$ ($\partial W=S_a\cup S_b$) consists of peripheral circles around the punctures $G\cap\partial W$, there is a disk $D_2$ properly embedded in $W_2$ such that $\partial D_2\cap F=\emptyset$ and $D_2$ separates the two components of $G\cap W_2$.  Since $F-G$ is incompressible and $\partial D_2\cap F=\emptyset$, after isotopy, we may assume $F\cap D_2=\emptyset$.  Note that each component of $W_2-D_2$ can be viewed as a tubular neighborhood of a component of $G\cap W_2$ in $W_2$.  So, similar to the argument above on $F\cap B_a$ and $F\cap B_b$, we can isotope $F\cap W_2$ into a small neighborhood of $G\cap W_2$ in $W_2$. 

Next we consider $W_1$.  As $\beta$ lies in a level sphere $S_t$, the H-shaped graph $G\cap W_1$ is $\partial$-parallel in $W_1$, which means that we may view $W_1$ as a tubular neighborhood of $G\cap W_1$ in $S^3$.  Since $F\cap \partial W$ consists of peripheral curves around the punctures, we can isotope $F\cap W_1$ into a small neighborhood of the H-shaped graph $G\cap W_1$. 

Now we glue $W_1$, $W_2$, $B_a$ and $B_b$ together to get back $S^3-G$.  The conclusions on $F\cap B_a$, $F\cap B_b$, $F\cap W_1$ and $F\cap W_2$ above imply that $F$ can be isotoped into a small neighborhood of $G$.

Finally, we come back to view $F$ as properly embedded in $X=S^3-N(G)$.  The conclusion above implies that $F$ can be isotoped into a product neighborhood of $\partial X$ in $X$.  Since $F$ is incompressible, by a theorem of Waldhausen \cite[Proposition 3.1 and Corollary 3.2]{W}, $F$ must be parallel (in the product neighborhood of $\partial X$) to a sub-surface of $\partial X$.  Hence $F$ is $\partial$-parallel in $X$.
\end{proof}

\begin{lemma}\label{LXb}
$X$ has incompressible boundary.
\end{lemma}
\begin{proof}
First, as $X$ is a subspace of $S^3$ and $G$ is connected, $X$ is irreducible. 
Since we assumed $\beta$ is not an unknotting tunnel for the 2-bridge knot $K$, $X$ is not a handlebody.    

Suppose $\partial X$ is compressible in $X$ and let $P$ be the surface obtained by maximally compressing $\partial X$ in $X$ and discarding any resulting 2-sphere components.  Since $X$ is irreducible and $X$ is not a handlebody, $P\ne\emptyset$.  This means that if $\partial X$ is compressible, then $X$ contains a non-peripheral incompressible surface $P$, contradicting Lemma~\ref{LXsmall}.
\end{proof}

\begin{lemma}\label{LXgenus}
The Heegaard genus of $X$ is 3.  In other words, the tunnel number of the graph $G$ is one.
\end{lemma}
\begin{proof}
This lemma is similar to the fact that the tunnel number of a 2-bridge knot is one. First,  since $\partial X$ has genus 2 and since $X$ is not a handlebody, the Heegaard genus of $X$ is at least 3.  Next we construct a genus 3 Heegaard splitting of $X$.

Let $S$ be the bridge sphere of $K$ containing $\beta$ and let $B_\pm$ be the 3-balls bounded by $S$, as in the construction of $G$ at the beginning of this section.  The 2-bridge knot has an unknotting tunnel $\tau$ in $B_+$ connecting the two arcs of $K\cap B_+$, as shown in Figure~\ref{Fknot}(a).  Let $H=\tau\cup(K\cap B_+)$ (note that $H$ does not contain $\beta$).  The H-shaped graph $H$ is $\partial$-parallel in $B_+$ and the 4-hole sphere $S-N(H)$ can be isotoped into $\partial \overline{N(H)}$.  In particular, the arc $\beta-N(H)$ in $S$ is parallel to an arc in $\partial \overline{N(H)}$.  This implies that we can fix $H$ and slide the arc $\beta$ into a trivial unknotted circle, see Figure~\ref{Fknot}(a,b) for a picture.  Since $\tau$ is an unknotting tunnel for the 2-bridge knot $K$, $N(K\cup\tau)$ is a standard genus 2 handlebody in $S^3$.  After sliding $\beta$ into a trivial circle as above, we see that $N(G\cup \tau)$ is a standard genus 3 handlebody in $S^3$.  Hence $\tau$ is an unknotting tunnel for $G$ and $X$ has a genus 3 Heegaard surface.  
\end{proof}

\begin{figure}
  \centering
\psfrag{a}{(a)}
\psfrag{b}{(b)}
\psfrag{c}{(c)}
\psfrag{t}{$\tau$}
\psfrag{r}{$\alpha$}
\psfrag{d}{$\beta$}
\psfrag{h}{sliding $\beta$}
\psfrag{x}{$x$}
\psfrag{y}{$y$}
\psfrag{s}{$s$}
\psfrag{F}{$S$}
  \includegraphics[width=5in]{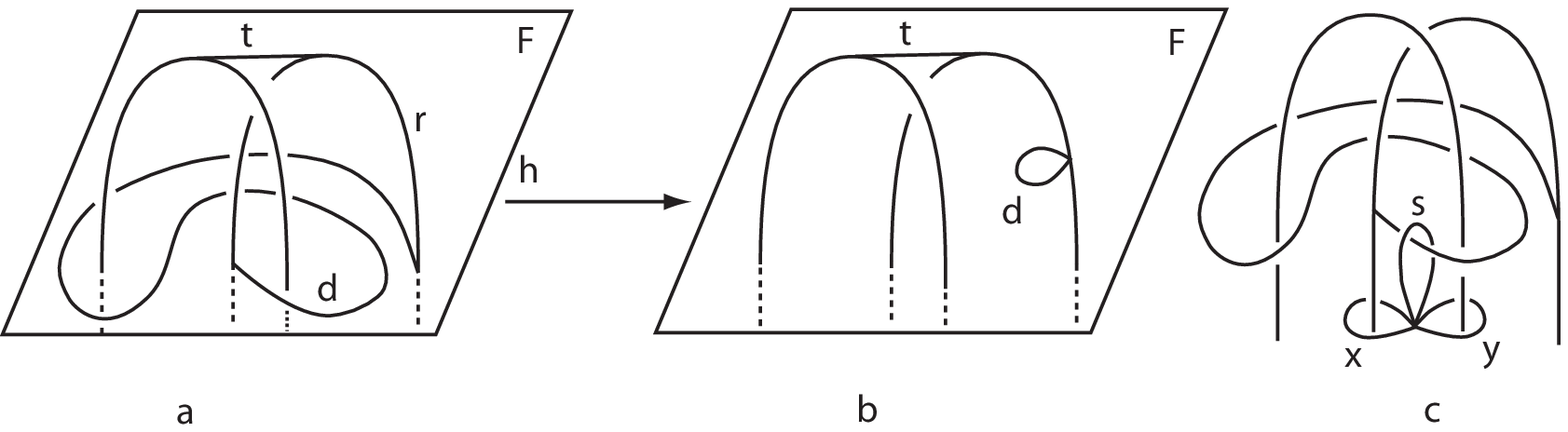}
  \caption{}
  \label{Fknot}
\end{figure}

\begin{lemma}\label{LXrank}
The rank of $\pi_1(X)$ is 3.  Moreover, $\pi_1(X)$ is generated by $x$, $h^{-1}xh$ and $s$, where $h\in\pi_1(X)$ and $x$ is represented by the core curve of the annulus $A\subset\partial X$ described  before Lemma~\ref{LXsmall}.
\end{lemma}
\begin{proof}
By Lemma~\ref{LXgenus}, $X$ has a genus 3 Heegaard splitting.  As $g(\partial X)=2$, the genus 3 Heegaard splitting of $X$ gives a presentation of $\pi_1(X)$ with three generators and one relator.  By a theorem of Whitehead \cite{Wh} (see \cite[Proposition II.5.11]{LS}), $\pi_1(X)$ cannot be generated by two elements unless the relator is an element of some basis of $\pi_1(X)$ and $\pi_1(X)$ is a free group.  Since $X$ is not a handlebody and $\pi_1(X)$ is not a free group, $\pi_1(X)$ cannot be generated by two elements and $rank(\pi_1(X))=3$.

Let $S$ be the bridge sphere containing $\beta$ and let $B_\pm$ be the 3-balls bounded by $S$ as in the proof of Lemma~\ref{LXgenus}.  We first push the arc $\beta$ slightly into $\Int(B_+)$.  
The fundamental group of the complement of the 2-bridge knot $K$ is generated by two elements $x$ and $y$ represented by two meridional loops in the 3-ball $B_-$, as shown in Figure~\ref{Fknot}(c).  We may choose $x$ and $y$ to be conjugate in $\pi_1(B_--K)$.  So $y=h^{-1}xh$ where $h$ is an element of $\pi_1(B_--K)$.  Let $s$ be an element represented by a loop around the arc $\beta$, as shown in Figure~\ref{Fknot}(c).  Since $x$ and $y$ generate $\pi_1(B_+-K)$ and since $\beta$ lies in a level 2-sphere, $\pi_1(B_+-G)$ can be generated by $x$, $y$ and $s$.  Notice that $\pi_1(B_--K)$ is a subgroup of $\pi_1(S-K)$.  So $\pi_1(S^3-G)$ is generated by $x$, $y$ and $s$.  By our construction, we may view $x$ as the element represented by the core curve of the annulus $A\subset\partial X$, and $y=h^{-1}xh$ where $h\in \pi_1(S^3-G)$.
\end{proof}

\begin{lemma}\label{Llive}
The dimension of $H_1(X;\mathbb{Z}_2)$ is 2.
\end{lemma}
\begin{proof}
First, by Lemma~\ref{LXrank}, $\pi_1(X)$ is generated by 3 elements $x$, $h^{-1}xh$ and $s$, two of which are conjugate.  This means that the dimension of $H_1(X;\mathbb{Z}_2)$ is at most 2.

By the ``half lives, half dies" Lemma (see \cite[Lemma 3.5]{Ha2}), the dimension of the kernel of $i_*\colon H_1(\partial X;\mathbb{Z}_2)\to H_1(X;\mathbb{Z}_2)$ is 2, since $g(\partial X)=2$.  As the dimension of $H_1(\partial X;\mathbb{Z}_2)$ is 4, the image of $i_*$ has dimension 2 and hence the dimension of $H_1(X;\mathbb{Z}_2)$ is at least 2.  So $dim(H_1(X;\mathbb{Z}_2))=2$.
\end{proof}

\section{The construction of $Y_s$}\label{SY}
In this section, we construct the second piece $Y_s$.  Our final manifold $M$ is an annulus sum of $Y_s$ and two copies of $X$ constructed in section~\ref{SX}.

Let $F_i$ ($i=1,2$) be a compact once-punctured non-orientable surface of genus $g$ ($g\ge 3$).  Note that the genus of a non-orientable surface is the cross-cap number.  Let $\alpha_i$ be an orientation-preserving non-separating simple closed curve in $F_i$.  As shown in Figure~\ref{Fsurface}, suppose $F_i$ is the surface obtained by gluing a M\"{o}bius band $\mu_i$ to a twice-punctured orientable surface $\Gamma_i$ ($\alpha_i\subset\Gamma_i$).  We view $\Gamma_i$ and $\mu_i$ as sub-surfaces of $F_i$.  As shown in Figure~\ref{Fsurface}, let $\gamma_i$ be an arc properly embedded in $F_i$ that winds around the core of the M\"{o}bius band $\mu_i$ once and intersects $\alpha_i$ in a single point.  In particular, the arc $\gamma_i$ is constructed to be orientation-reversing in $F_i$ in the sense that if we pinch $\partial\gamma_i$ to one point along $\partial F_i$ then the resulting closed curve is orientation-reversing.  Let $O_i=\alpha_i\cap\gamma_i$.  As shown in Figure~\ref{Fsurface}, let $p_i$ and $q_i$ be the two endpoints of $\gamma_i$ such that the subarc of $\gamma_i$ bounded by $O_i\cup p_i$ lies in $\Gamma_i$ and the subarc of $\gamma_i$ bounded by $O_i\cup q_i$ intersects $\mu_i$ ($i=1,2$).  

Now we consider the twisted $I$-bundle $N_i$ over $F_i$ ($i=1,2$).  As $F_i=\mu_i\cup\Gamma_i$, we may view $N_i$ as the union of $\Gamma_i\times I$ and a twisted $I$-bundle over $\mu_i$.  We may view $F_i$ as a section of the $I$-bundle $N_i$ and view the sub-surface $\Gamma_i$ of $F_i$ as $\Gamma_i\times\{\frac{1}{2}\}\subset\Gamma_i\times I\subset N_i$.

Let $\Gamma$ be an annulus.  We can form a closed non-orientable surface $F$ of genus $2g$ by gluing $F_1$ and $F_2$ to $\Gamma$ along boundary circles.

\begin{figure}
  \centering
\psfrag{g}{glue}
\psfrag{a}{$\alpha_1$}
\psfrag{b}{$\beta_1$}
\psfrag{c}{$\gamma_1$}
\psfrag{p}{$q_1$}
\psfrag{q}{$p_1$}
\psfrag{m}{$\mu_1$}
\psfrag{x}{$\alpha_2$}
\psfrag{y}{$\beta_2$}
\psfrag{z}{$\gamma_2$}
\psfrag{r}{$q_2$}
\psfrag{s}{$p_2$}
\psfrag{n}{$\mu_2$}
\psfrag{F}{$F_1$}
\psfrag{G}{$F_2$}
\psfrag{I}{$\Gamma_1$}
\psfrag{J}{$\Gamma_2$}
  \includegraphics[width=5in]{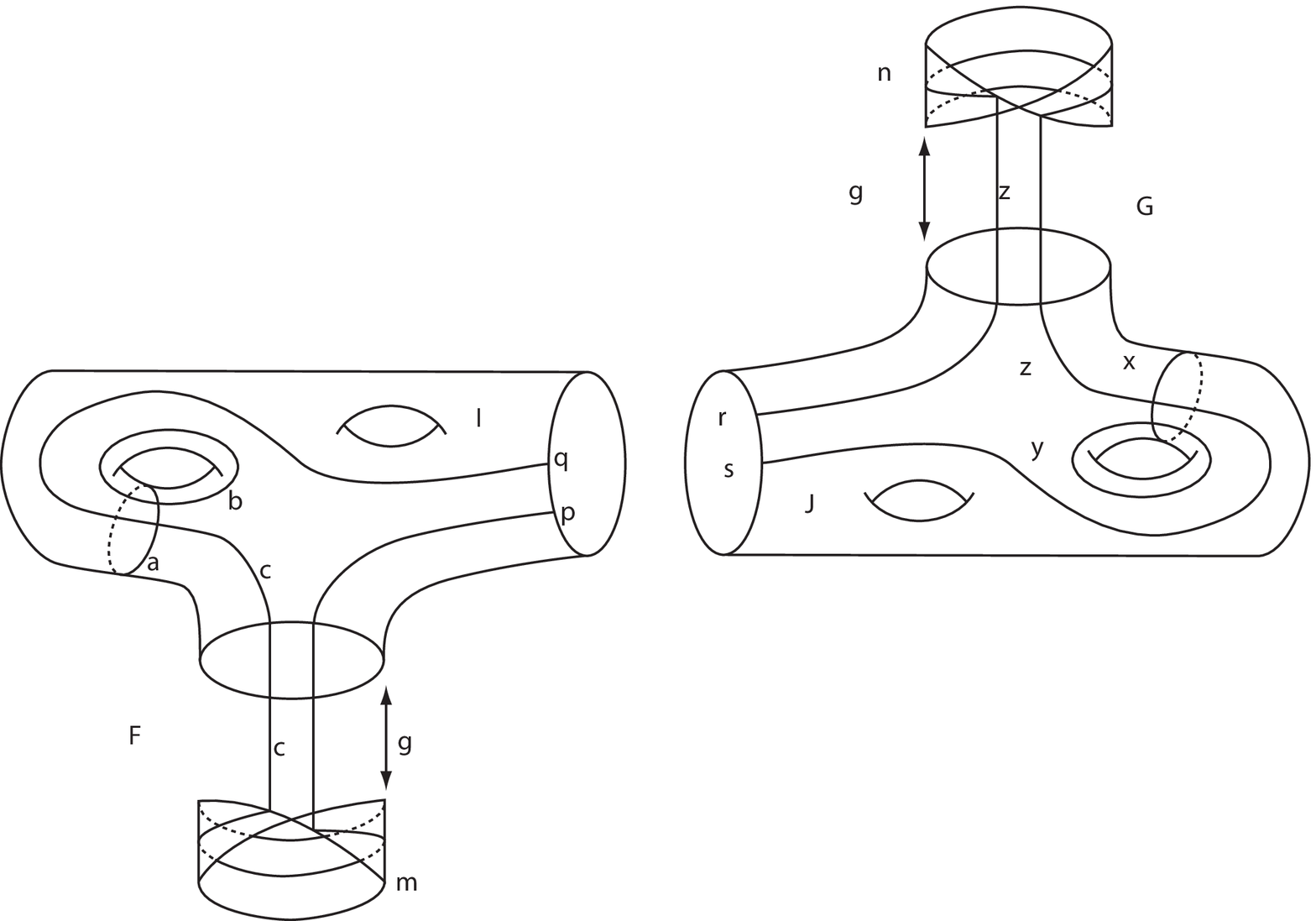}
  \caption{}
  \label{Fsurface}
\end{figure}

Let $N_F$ be the twisted $I$-bundle over $F=F_1\cup\Gamma\cup F_2$ and we may view $N_F$ as the manifold obtained by gluing the $I$-bundles $N_1$ and $N_2$ above to $\Gamma\times I$.  We view $N_1$, $N_2$ and $\Gamma\times I$ as sub-manifolds of $N_F$.  In particular, $\Gamma\times\{0\}$ and $\Gamma\times\{1\}$ are a pair of annuli in $\partial N_F$.  Let $F_K$ be another  once-punctured non-orientable surface of genus $g$ and let $N_K$ be a twisted $I$-bundle over $F_K$.  We denote the vertical boundary of $N_K$ by $\partial_vN_K$.  As $F_K$ has one boundary circle, $\partial_vN_K$ is an annulus.  Now we glue $N_K$ to $N_F$ by identifying $\partial_vN_K$ to $\Gamma\times\{1\}$.  We denote the resulting manifold by $Y_N$.  

Note that $Y_N$ is homotopy equivalent to the 2-complex obtained by gluing $F_1$, $F_2$ and $F_K$ together along their boundary circles.  Since $F_1$, $F_2$ and $F_K$ are all once-punctured non-orientable surfaces of genus $g$, it is easy to see that the dimension of $H_1(Y_N;\mathbb{Z}_2)$ is $3g$.  Since $rank(\pi_1(\mathcal{M}))\ge dim(H_1(\mathcal{M};\mathbb{Z}_2))$ for any manifold $\mathcal{M}$, $rank(\pi_1(Y_N))\ge 3g$.  On the other hand, the standard generators for $\pi_1(F_1)$, $\pi_1(F_2)$ and $\pi_1(F_K)$ form a generating set of $3g$ elements for $\pi_1(Y_N)$.  Hence $rank(\pi_1(Y_N))=3g$.

Now we add a 1-handle to $Y_N$ with the two ends of the 1-handle in the annulus $\Gamma\times\{0\}$.  More specifically, let $D_0$ and $D_1$ be a pair of disks in $\Gamma\times\{0\}$.  We glue a 1-handle $D^2\times I$ by identifying $D^2\times\{0\}$ and $D^2\times\{1\}$ to $D_0$ and $D_1$ respectively.  Let $Y$ be the resulting manifold.

Clearly $rank(\pi_1(Y))=3g+1$.  Next we describe a preferred generating set for $\pi_1(Y)$.  For any closed curve $x$ in $Y$, we use $[x]$ to denote the element in $\pi_1(Y)$ represented by $x$ (after pulling $x$ to pass through the basepoint for $\pi_1(Y)$).  To simplify notation and pictures, we do not specify a basepoint for $\pi_1(Y)$.

We may assume that the two endpoints $p_i$ and $q_i$ of $\gamma_i$ are close to each other in $\partial F_i$ and we use $[\gamma_i]$ to denote the element of $\pi_1(F_i)$ represented by the loop obtained by pinching $p_i$ to $q_i$ in a small neighborhood of $q_i$.  Let $m_i$ be the core curve of the M\"{o}bius band $\mu_i$.  As shown in Figure~\ref{Fsurface}, we may suppose there is a simple closed curve $\beta_i$ in $\Gamma_i\subset F_i$ such that $\beta_i\cap\alpha_i$ is a single point, $\beta_i\cap\gamma_i=\emptyset$ and $[\gamma_i]=[\beta_i]\cdot[m_i]$ in $\pi_1(F_i)$.    We may extend $[\alpha_i]$, $[\beta_i]$ and $[m_i]$ to a standard set of $g$ generators $\{[m_i], [\alpha_i], [\beta_i],[\theta_{i4}],\dots, [\theta_{ig}]\}$ for $\pi_1(F_i)$ ($i=1,2$), where each $\theta_{ik}$ is a simple closed curve in $\Gamma_i\subset F_i$ disjoint from $\alpha_i\cup\beta_i\cup\gamma_i$.  For an argument latter in the proof, we choose a slightly different set of curves representing this set of generators.   Recall that the twisted $I$-bundle $N_i$ ($i=1,2$) is the union of $\Gamma_i\times I$ and the twisted $I$-bundle over the M\"{o}bius band $\mu_i$.  Let $\alpha_i'=\alpha_i\times\{1\}\subset \Gamma_i\times\{1\}$, $\beta_i'=\beta_i\times\{1\}\subset \Gamma_i\times\{1\}$, and $\theta_{ik}'=\theta_{ik}\times\{1\}\subset\Gamma_i\times\{1\}$.  So the set of $g$ simple closed curves $m_i, \alpha_i', \beta_i', \theta_{i4}',\dots,\theta_{ig}'$  represents a standard set of $g$ generators for $\pi_1(N_i)=\pi_1(F_i)$.  We would like to emphasize that, except for $m_i$, every curve in this set lies in $\Gamma_i\times\{1\}$, and $\Gamma_i\times\{1\}$ is glued to $\Gamma\times\{1\}=\partial_vN_K$ along $\partial F_i\times\{1\}$.  We denote this set of generators of $\pi_1(N_i)$ by $\mathcal{F}_i$.

Let $\mathcal{F}_K$ be a set of $g$ generators for $\pi_1(F_K)$ and let $\gamma_0$ be a curve representing the core of the added 1-handle.  So $\mathcal{F}_1\cup\mathcal{F}_2\cup\mathcal{F}_K\cup\{[\gamma_0]\}$ is a set of $3g+1$ generators for $\pi_1(Y)$.  

Note that by connecting $\gamma_0$, the circles $\{m_i, \alpha_i, \beta_i,\theta_{i4},\dots,\theta_{ig}\}$ in $F_i$ and the circles in $F_K$ representing $\mathcal{F}_K$ together, we can form a graph $G$ such that $\overline{N(G)}$ is a handlebody of genus $3g+1$.  In this construction, the $g$ circles $\{m_i, \alpha_i, \beta_i,\theta_{i4},\dots,\theta_{ig}\}$ are connected to form a core graph of the handlebody $N_i$, and the $g$ circles in $F_K$ representing $\mathcal{F}_K$ form a core graph of the handlebody $N_K$.  So $N_1\cup N_2\cup N_K$ can be obtained by attaching two 2-handles to a neighborhood of the 3 core graphs (pinched together) of $N_1$, $N_2$ and $N_K$.  This means that $Y-N(G)$ is a compression body.  Thus $\overline{N(G)}\cup (Y-N(G))$ gives a (standard) Heegaard splitting of $Y$ corresponding to our generators of $\pi_1(Y)$ above.

Next we replace $\gamma_0$ by another curve and form a slightly different generating set of $\pi_1(Y)$.    
As shown in Figure~\ref{Fgamma}(a), we first connect $p_1$ to $p_2$ by an arc $\gamma_p$ in the annulus $\Gamma\times \{\frac{1}{2}\}$.   Then we connect $q_1$ to $q_2$ by an arc $\gamma_q$ that goes through the added 1-handle exactly once, see Figure~\ref{Fgamma}(a).  More specifically, as shown in Figure~\ref{Fgamma}(a), we first take a core curve $\gamma_h$ of the 1-handle (by our construction, $\partial\gamma_h$ is a pair of points in the two disks $D_0$ and $D_1$ in $\Gamma\times\{0\}$), and $\gamma_q$ is obtained by connecting the two endpoints of $\gamma_h$ to the two points $q_1$ and $q_2$ using a pair of unknotted trivial arcs $\delta_1$ and $\delta_2$ in $\Gamma\times I$ respectively.  So $\gamma_q=\delta_1\cup\gamma_h\cup\delta_2$.  The simple closed curve $\gamma=\gamma_1\cup\gamma_2\cup\gamma_p\cup\gamma_q$ represents the element $[\gamma_1]\cdot[\gamma_0]\cdot[\gamma_2]^{-1}$ in $\pi_1(Y)$.  Thus, after replacing $[\gamma_0]$ by $[\gamma]$, we get a new set of generators $\mathcal{F}_1\cup\mathcal{F}_2\cup\mathcal{F}_K\cup\{[\gamma]\}$ of $\pi_1(Y)$ consisting of $3g+1$ elements.

\begin{figure}
  \centering
\psfrag{X}{$\Gamma\times\{1\}$}
\psfrag{p}{$p_1$}
\psfrag{q}{$q_1$}
\psfrag{r}{$p_2$}
\psfrag{s}{$q_2$}
\psfrag{0}{\small{$D_0$}}
\psfrag{1}{\small{$D_1$}}
\psfrag{h}{$\gamma_h$}
\psfrag{g}{$\gamma_p$}
\psfrag{d}{$\delta_1$}
\psfrag{e}{$\delta_2$}
\psfrag{f}{$\underline{\gamma_q=\delta_1\cup\gamma_h\cup\delta_2}$}
\psfrag{R}{$R$}
\psfrag{D}{$\Delta_1$}
\psfrag{2}{$\Delta_2$}
\psfrag{a}{(a)}
\psfrag{b}{(b)}
\psfrag{K}{$K_1$}
\psfrag{j}{$K_2$}
  \includegraphics[width=5in]{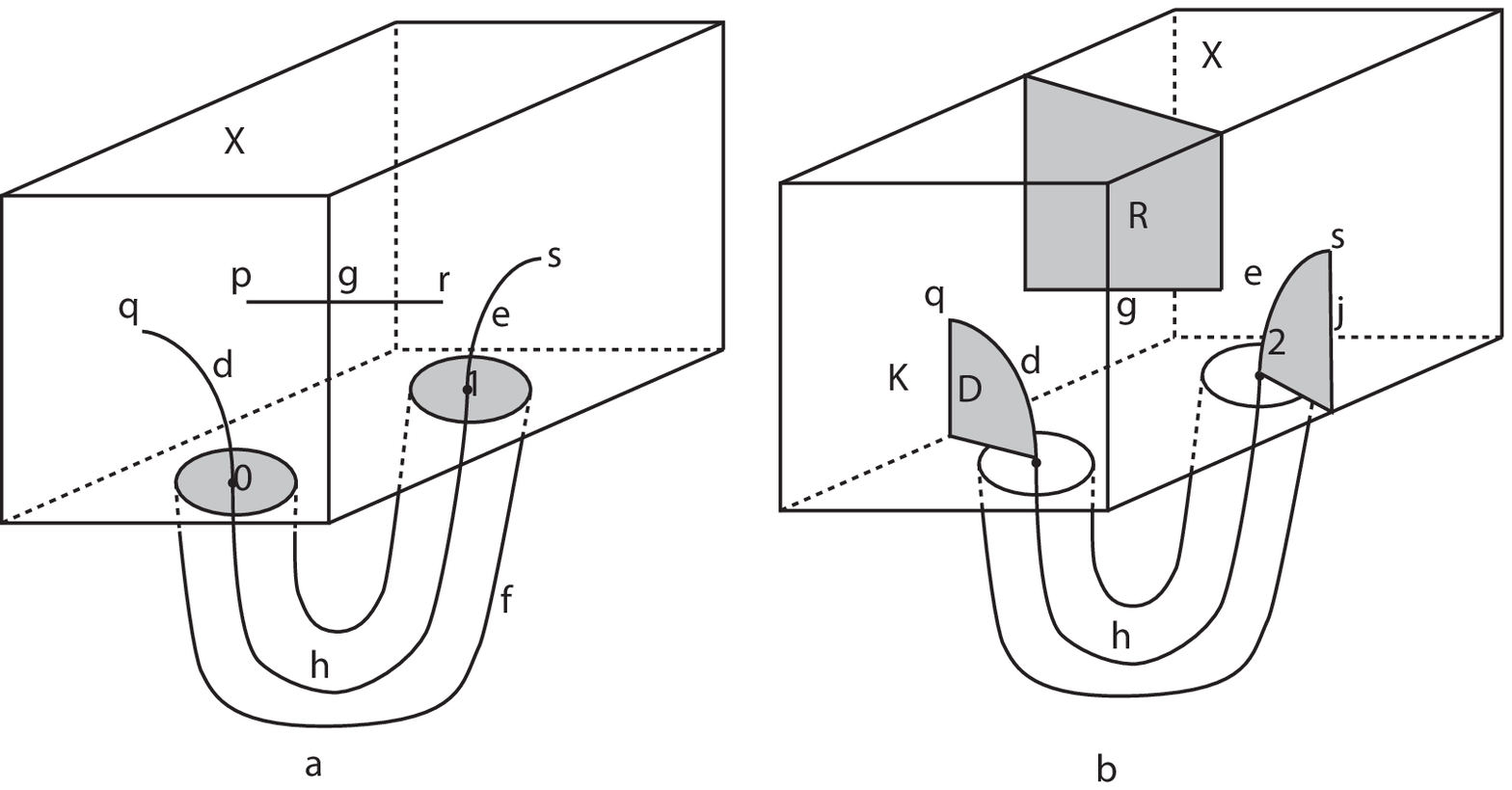}
  \caption{}
  \label{Fgamma}
\end{figure}

From the viewpoint of Heegaard splitting,  we can do a handle-slide on the handlebody $\overline{N(G)}$ described above, dragging the two ends of the 1-handle corresponding to $\gamma_0$ along $\gamma_1$ and $\gamma_2$.  The resulting 1-handle can be viewed as a small neighborhood of $\gamma$.  In particular, this means that $\gamma$ is a core curve of the handlebody $\overline{N(G)}$ in the standard Heegaard splitting of $Y$ described above (in other words, a compressing disk of $\overline{N(G)}$ divides $\overline{N(G)}$ into two components one of which is a solid torus containing $\gamma$ as its core curve).  This means that after any Dehn surgery on $\gamma$, $\overline{N(G)}$ remains a handlebody.  Let $Y_s$ be the manifold obtained from $Y$ by performing a Dehn surgery on $\gamma$ with surgery slope $s$ (with respect to a certain fixed framing in which the meridional slope is $\infty$).  The discussion above says that $\overline{N(G)}$ remains a handlebody after the Dehn surgery and hence $\partial \overline{N(G)}$ remains a Heegaard surface of $Y_s$.  Hence the Heegaard genus $g(Y_s)\le 3g+1$.

Let $\gamma'$ be the core of the surgery solid torus in $Y_s$.
Now we replace $[\gamma]$ in the generating set of $\pi_1(Y)$ above by $[\gamma']$ in $\pi_1(Y_s)$.  It is easy to see from our construction that $\mathcal{F}_1\cup\mathcal{F}_2\cup\mathcal{F}_K\cup\{[\gamma']\}$ is a generating set for $\pi_1(Y_s)$ containing $3g+1$ elements, where the elements in $\mathcal{F}_i$ are now viewed as elements in $\pi_1(Y_s)$ represented by the curves $m_i$, $\alpha_i'$, $\beta_i'$, $\theta_{i4}',\dots, \theta_{ig}'$.  Note that since the annulus in $Y$ bounded by $\alpha_i\cup\alpha_i'$ intersects $\gamma$ once, we may view $[\alpha_i']=[\alpha_i]\cdot [\gamma']^n$ in $\pi_1(Y_s)$, where $s=\frac{m}{n}$, as the meridional loop around $\gamma$ is a boundary curve of the surgery solid torus winding around the core $\gamma'$ $n$ times.

Now we are in position to construct our manifold $M$.  

Let $A_i\subset\Gamma_i\times\{1\}\subset\partial Y_s$ ($i=1,2$) be an annular neighborhood of the curve $\alpha_i'=\alpha_i\times\{1\}$ which represents a generator in $\mathcal{F}_i$.  Next we consider the pair $(X, A)$ where $X$ is the manifold constructed in section~\ref{SX} and $A\subset\partial X$ is the annulus in Lemma~\ref{LXsmall} and Lemma~\ref{LXrank}.  We take two copies of the pair, which we denote by $(X_1,A_1)$ and $(X_2,A_2)$, and then we glue $X_1$ and $X_2$ to $Y_s$ by identifying the annulus $A_i$ ($i=1,2$) in $\partial X_i$ to the annulus $A_i$ in $\partial Y_s$.  Denote resulting manifold by $M$.  We may view $A_1$ and $A_2$ as annuli properly embedded in $M$ that divide $M$ into $X_1$, $Y_s$ and $X_2$.

\begin{lemma}\label{LRi}
Let $\gamma'$ be the core of the surgery solid torus in $Y_s\subset M$.  
Let $W=M-N(\gamma')$, in other words, $W$ is obtained by gluing $X_1$ and $X_2$ to $Y-N(\gamma)$ along the annuli $A_1$ and $A_2$ respectively.  Let $T=\partial\overline{N(\gamma')}$ be the torus component of $\partial W$. Then,
\begin{enumerate}
  \item $W$ is irreducible and $\partial$-irreducible,
  \item  $W$ does not contain any non-peripheral incompressible torus, and 
  \item $W$ does not contain any essential annulus with both boundary curves in $T$.
\end{enumerate}
\end{lemma}
\begin{proof}
Let $M'=X_1\cup_{A_1} Y\cup_{A_2} X_2$ be the manifold obtained by gluing $X_1$ and $X_2$ to $Y$ along the annuli $A_1$ and $A_2$.  So $M$ is obtained from $M'$ by a Dehn surgery on $\gamma$ with surgery slope $s$.  We may view $W=M'-N(\gamma)$ and $T=\partial\overline{N(\gamma)}$.

Let $H=\partial W-T$ be the other boundary component of $W$. 
We prove most parts of the lemma at the same time and prove the last case in Claim~\ref{claimC2} below.  Suppose the lemma is false and suppose there is a surface $P\subset W$ that is either (1) an essential 2-sphere, or (2) a compressing disk with $\partial P\subset T$, or (3) a non-peripheral incompressible torus, or (4) an essential annulus with $\partial P\subset T$.  In particular $P\cap H=\emptyset$.  Next we show that such $P$ does not exist.

\begin{claimC}\label{claimC0}
Let $Q$ be an incompressible surface in $W$ with $\partial Q\subset H$.  Suppose $|P\cap Q|$ is minimal among all such surfaces $P$.  Then $P\cap Q$ (if not empty) consists of essential curves in both $P$ and $Q$.
\end{claimC}
\begin{proof}[Proof of Claim~\ref{claimC0}]
Since $P\cap H=\emptyset$ and $\partial Q\subset H$, $P\cap Q$ consists of simple closed curves.  Let $\delta$ be a component of $P\cap Q$ that is an innermost trivial curve in $Q$. Since $P$ is incompressible, $\delta$ also bounds a disk in $P$.  After compressing $P$ along the disk in $Q$ bounded by $\delta$, we obtain two surfaces $P'$ and $P''$ where $P''$ is a 2-sphere.  Since $P$ is essential, in all the cases for $P$, at least one of $P'$ and $P''$ is essential in $W$.  As both $P'$ and $P''$ intersect $Q$ in fewer curves than $|P\cap Q|$, this contradicts that $|P\cap Q|$ is minimal.  Thus each component of $P\cap Q$ must be essential in $Q$.  Furthermore, since $Q$ is incompressible in $W$ and $P\cap Q$ is essential in $Q$, each component of $P\cap Q$ must also be essential in $P$. 
\end{proof}

Recall that $Y_N$ is obtained by gluing the twisted $I$-bundles $N_1$, $N_2$ and $N_K$ to $\Gamma\times I$, and $Y$ is obtained by adding a 1-handle to $Y_N$.  Next we analyze how $P$ intersects the 3 twisted $I$-bundles. 

Suppose $|P\cap(\partial_vN_K\cup A_1\cup A_2)|$ is minimal among all such surfaces $P$.  

\begin{claimC}\label{claimC1}
$P\cap(\partial_vN_K\cup A_1\cup A_2)=\emptyset$.
\end{claimC}
\begin{proof}[Proof of Claim~\ref{claimC1}]
By our construction, the core curves of $A_1$, $A_2$ and $\partial_vN_K$ are essential in $Y_N$ and hence essential in $Y$.   Since the core curves of $A_1$ and $A_2$ are also  essential in $X_1$ and $X_2$ respectively, $\partial_vN_K$, $A_1$ and $A_2$ are all incompressible in $W$.

By Claim~\ref{claimC0}, $P\cap(\partial_vN_K\cup A_1\cup A_2)$ consists of simple closed curves that are essential in both $P$ and $\partial_vN_K\cup A_1\cup A_2$.  This immediately implies that $P\cap(\partial_vN_K\cup A_1\cup A_2)=\emptyset$ if $P$ is a sphere or disk.

Next we consider the cases that $P$ is a torus or an annulus.  If $P\cap(\partial_vN_K\cup A_1\cup A_2)\ne\emptyset$, then the conclusion above implies that $P\cap X_i$ and $P\cap N_K$ consist of incompressible annuli in $X_i$ and $N_K$ respectively.  By Lemma~\ref{LXsmall}, $X_i$ is $A_i$-small, so each component of $P\cap X_i$ is a $\partial$-parallel annulus in $X_i$.  By Lemma~\ref{LIbundle}, $N_K$ is $\partial_vN_K$-small, so each component of $P\cap N_K$ is a $\partial$-parallel annulus in $N_K$. 
Hence we can isotope the annuli in $P\cap X_i$ and $P\cap N_K$ across $A_i$ and $\partial_vN_K$ respectively and into $Y-N_K$.  This contradicts that $|P\cap(\partial_vN_K\cup A_1\cup A_2)|$ is minimal.  Therefore, $P\cap(\partial_vN_K\cup A_1\cup A_2)=\emptyset$ in all cases for $P$.
\end{proof}

By our construction of $X_i$ and $N_K$, Claim~\ref{claimC1} implies that $P\subset Y-N_K$. 
Now we consider the twisted $I$-bundle $N_i$ ($i=1,2$).  Recall that $\gamma\cap N_i=\gamma_i$ is an arc properly embedded in $F_i$.  Let $\pi\colon N_i\to F_i$ be the projection that collapses each $I$-fiber to a point.  Let $L_i$ be a small neighborhood of $\gamma_i$ in $F_i$ and let $R_i=\pi^{-1}(L_i)$. So $R_i$ can be viewed as an induced $I$-bundle over $L_i$.  We may assume $N(\gamma)\cap N_i$ lies in $R_i$. Hence $T\cap N_i$ is an annulus properly embedded in $R_i$.  Let $W_i$ the closure of $N_i-R_i$.  So $W_i$ is an $I$-bundle over a compact surface.  Moreover, by the construction of $\gamma_i$, $\partial_vW_i$ is a vertical essential annulus in $N_i$. 

We may also suppose $|P\cap\partial_vW_i|$ is minimal among all such surfaces $P$.  Similar to Claim~\ref{claimC0}, $P\cap\partial_vW_i$ consists of curves essential in both $P$ and $\partial_vW_i$.   This immediately implies that $P\cap\partial_vW_i=\emptyset$ if $P$ is a sphere or disk. Moreover,
if $P$ is a torus or an annulus, $P\cap W_i$ (if not empty) consists of annuli incompressible in $W_i$.  However, by Lemma~\ref{LIbundle} and since $g\ge 3$, an incompressible annulus in $W_i$ with boundary in $\partial_vW_i$ is parallel to a sub-annulus of $\partial_vW_i$.  Hence we can isotope $P\cap W_i$ out of $W_i$.  Since $|P\cap\partial_vW_i|$ is assumed to be minimal, we have $P\cap W_i=\emptyset$.  As $R_i$ can be viewed as a tubular neighborhood of $\gamma_i$, after isotopy, $P\cap N_i$ lies in a tubular neighborhood of $\gamma_i$ in $N_i$ ($i=1,2$).    

Let $W_\Gamma$ be the closure of $Y-(N_1\cup N_2\cup N_K)$.   By our construction, $W_\Gamma$ is the genus 2 handlebody obtained by adding a 1-handle to $\Gamma\times I$, and $\gamma\cap W_\Gamma$ consists of two arcs $\gamma_p$ and $\gamma_q$.  Moreover, $\gamma_p$ and $\gamma_q$ are $\partial$-parallel in $W_\Gamma$.  So there are a pair of embedded and disjoint disks $D_p$ and $D_q$ in $W_\Gamma$ such that $\partial D_p=\gamma_p\cup d_p$ and $\partial D_q=\gamma_q\cup d_q$, where $d_p=D_p\cap\partial W_\Gamma$ and $d_q=D_q\cap\partial W_\Gamma$.  Let $B_p$ and $B_q$ be small neighborhoods of $D_p$ and $D_q$ in $W_\Gamma$ respectively, so $B_p$ and $B_q$ are a pair of 3-balls containing $\gamma_p$ and $\gamma_q$ respectively.  In our construction of $W=M'-N(\gamma)$, where $M'=X_1\cup Y\cup X_2$, we may assume $N(\gamma)$ is so small that $N(\gamma)\cap W_\Gamma\subset B_p\cup B_q$ and hence $T\cap W_\Gamma$ is a pair of annuli lying in $B_p\cup B_q$. 

Since $P\cap N_i$ ($i=1, 2$) lies in a tubular neighborhood of $\gamma_i$ in $N_i$ and $P\cap\partial_vN_K=\emptyset$, after isotopy, we may assume $P\cap\partial W_\Gamma$ lies in the interior of the pair of disks $\partial B_P\cap\partial W_\Gamma$ and $\partial B_q\cap\partial W_\Gamma$.  
Let $\Delta_p=\overline{\partial B_p-\partial W_\Gamma}$ and $\Delta_q=\overline{\partial B_q-\partial W_\Gamma}$.  So $\Delta_p$ and $\Delta_q$ are disks properly embedded in $W_\Gamma$, cutting off the 3-balls $B_p$ and $B_q$ from $W_\Gamma$.  By our assumption on $P\cap\partial W_\Gamma$ above, $P\cap\partial\Delta_p=\emptyset$ and $P\cap\partial\Delta_q=\emptyset$.  This means that $P\cap\Delta_p$ and $P\cap\Delta_q$ (if not empty) consist of simple closed curves in $\Delta_p$ and $\Delta_q$ respectively.  As in the proof of Claim~\ref{claimC0}, after compressing $P$ along subdisks of $\Delta_p\cup\Delta_q$ bounded by curves in $P\cap\Delta_p$ and $P\cap\Delta_q$, we may assume $P\cap\Delta_p=\emptyset$ and $P\cap\Delta_q=\emptyset$.

By our assumptions on $P\cap N_i$,  $P\cap\Delta_p$ and $P\cap\Delta_q$, we can conclude that $P$ lies in either $R_1\cup R_2\cup B_p\cup B_q$ or in $W_\Gamma-(B_p\cup B_q)$.  If $P\subset W_\Gamma-(B_p\cup B_q)$, since $T\cap W_\Gamma$ lies inside $B_p\cup B_q$, $P$ can only be an essential 2-sphere or torus.  However, since $W_\Gamma-(B_p\cup B_q)$ is a handlebody, this cannot happen.  So $P\subset R_1\cup R_2\cup B_p\cup B_q$.  By our construction, we may choose $B_p$ and $B_q$ so that $R_1\cup R_2\cup B_p\cup B_q$ can be viewed as a tubular neighborhood of $\gamma$ in $M'=X_1\cup Y\cup X_2$ that contains $N(\gamma)$ and $T$.  So $(R_1\cup R_2\cup B_p\cup B_q)-N(\gamma)$ is a product neighborhood of $T$ in $W$.  However, there are no such surface $P$ in a product neighborhood of $T$ in $W$.  Thus such $P$ does not exist.

The argument on $P$ above proves most cases in the lemma.  It remains to show that the component $H$ of $\partial W$ is incompressible in $W$.

\begin{claimC}\label{claimC2}
$H$ is incompressible in $W$.
\end{claimC}
\begin{proof}[Proof of Claim~\ref{claimC2}]
We may view $H$ as the boundary of $M'$, where $M'=X_1\cup_{A_1} Y\cup_{A_2} X_2$.
Claim~\ref{claimC2} is equivalent to the statement that $H$ is incompressible in $M'-\gamma$.

Recall that $Y$ is obtained from $Y_N$ by adding a 1-handle along a pair of disks $D_0$ and $D_1$ in $\Gamma\times\{0\}$, where $Y_N$ is the manifold obtained by gluing $N_K$ to $N_F$ by identifying $\partial_vN_K$ to $\Gamma\times\{1\}$. 

Suppose the claim is false and let $\Delta$ be a compressing disk for $H$ in $M'-\gamma$.  We may view $D_0$ and $D_1$ as a pair of once-punctured disks in $Y-\gamma$.   We may assume $\Delta$ is transverse to $D_0\cup D_1$ and assume $|\Delta\cap(D_0\cup D_1)|$ is minimal among all the compressing disks for $H$.

We first show that $\Delta\cap D_i=\emptyset$ for both $i$. 
If $\Delta\cap D_i$ contains a closed curve $\delta$, then $\delta$ bounds a disk $\Delta_\delta$ in $\Delta$ and a disk $D_\delta$ in $D_i$.  Suppose $\delta$ is innermost in $D_i$, which means that $D_\delta\cup\Delta_\delta$ is an embedded 2-sphere.  If $D_\delta$ contains the puncture $\gamma\cap D_i$, then the 2-sphere $D_\delta\cup\Delta_\delta$ intersects $\gamma$ in one point, which implies that the boundary torus $T$ of $W$ is compressible in $W$, and this contradicts our conclusion above that $T$ is incompressible.  Thus $D_\delta$ does not contain the puncture $\gamma\cap D_i$.  So we can compress $\Delta$ along $D_\delta$ and get a new compressing disk for $H$ with fewer intersection curves with $D_i$.  Since $|\Delta\cap(D_0\cup D_1)|$ is minimal, this means that $\Delta\cap D_i$ contains no closed curves.  

If  $\Delta\cap D_i$ contains an arc $\delta'$, then $\delta'$ divides $D_i$ into two subdisks and let $D_\delta'$ be the subdisk that does not contain the puncture $\gamma\cap D_i$.  We may suppose $\delta'$ is outermost in the sense that $D_\delta'\cap \Delta=\delta'$.  Then we perform a $\partial$-compression on $\Delta$ along $D_\delta'$, see Figure~\ref{Fdisk}(b).  The $\partial$-compression changes $\Delta$ into two disks, at least one of which is a compressing disk for $H$ with fewer intersection curves with $D_i$ than $|\Delta\cap D_i|$.  Since $|\Delta\cap(D_0\cup D_1)|$ is minimal, no such arc $\delta'$ exists.  Thus $\Delta\cap D_i=\emptyset$ for both $i$. 

Similarly, after some compressions and $\partial$-compressions, we may assume that $\Delta\cap (A_1\cup A_2\cup \partial_vN_K)$ contains no arcs or closed curves that are trivial in the annuli $A_1\cup A_2\cup \partial_vN_K$.  As $\Delta$ is a disk and since $A_1\cup A_2\cup \partial_vN_K$ is incompressible, this means that, after isotopy, $\Delta\cap (A_1\cup A_2\cup \partial_vN_K)$ contains no closed curve, and each component of $\Delta\cap (A_1\cup A_2\cup \partial_vN_K)$ is an essential arc in $A_1\cup A_2\cup \partial_vN_K$.

If $\Delta\cap (A_1\cup A_2\cup \partial_vN_K)\ne\emptyset$, then there is an arc $\eta$ in $\Delta\cap (A_1\cup A_2\cup \partial_vN_K)$ that is outermost in $\Delta$, i.e. $\eta$ and a subarc of $\partial\Delta$ bound a bigon subdisk $E$ of $\Delta$ such that $E\cap (A_1\cup A_2\cup \partial_vN_K)=\eta$.  

The annuli $A_1$, $A_2$, $\partial_vN_K$ and the two disks $D_0$ and $D_1$ above divide $M'$ into several pieces: $X_1$, $X_2$, $N_K$, $N_F$ and the added 1-handle $D^2\times I$.  By our construction, $E$ lies in one of these pieces.   
As $\partial X_i$ is incompressible in $X_i$, $E\not\subset X_i$.  Similarly, since $\partial_hN_K$ is incompressible in the $I$-bundle $N_K$ and since $\partial_hN_K$ is not parallel to $\partial_vN_K$, by Lemma~\ref{Lmonogon}, $E\not\subset N_K$.  Moreover, since $\Delta\cap (D_0\cup D_1)=\emptyset$, these conclusions imply that $E\subset N_F$.  However, since $N_F$ is a twisted $I$-bundle over a closed surface, $\partial N_F$ is incompressible in $N_F$.  Hence $E\not\subset N_F$, a contradiction.  

If $\Delta\cap (A_1\cup A_2\cup \partial_vN_K)=\emptyset$, then by applying the argument on $E$ above to $\Delta$, we see that $\Delta$ cannot be in $X_1$, $X_2$, $N_K$, or $N_F$.  Moreover, since the intersection of $\gamma$ with the 1-handle $D^2\times I$ is a core curve of the 1-handle, $\Delta$ cannot lie inside $(D^2\times I)-\gamma$.   So $\Delta$ does not exist and the claim holds.
\end{proof}

Claim~\ref{claimC2} plus the discussion on $P$ before Claim~\ref{claimC2} proves Lemma~\ref{LRi}. 
\end{proof}

\begin{lemma}\label{LRM}
The rank of $\pi_1(M)$ is at most $3g+3$.
\end{lemma}
\begin{proof}
Let $\Omega_i\subset N_i$ be the induced $I$-bundle over the M\"{o}bius band $\mu_i$.  So we may view $N_i$ as the union of $\Omega_i$ and $\Gamma_i\times I$.   Recall that the curve $\gamma\subset Y$ before the Dehn surgery winds around the core curve $m_i$ of $\mu_i$ exactly once.  So we may assume $\Omega_i-N(\gamma)$ is a handlebody of genus 2 and $\pi_1(\Omega_i-\gamma)$ is generated by $[m_i]$ and the element represented by a meridional loop around $\gamma$.  Thus $\pi_1(\Omega_i-\gamma)$ can be generated by $[m_i]$ and $[\gamma']$ in $\pi_1(Y_s)$, where $\gamma'$ is the core of the surgery solid torus in $M$ as in Lemma~\ref{LRi}.  

In the discussion above, we have a set of $3g+1$ generators $\mathcal{F}_1\cup\mathcal{F}_2\cup\mathcal{F}_K\cup[\gamma']$ for $\pi_1(Y_s)$.  Let $x_i=[\alpha_i']$ and $b_i=[\beta_i']$ be the two special generators in $\mathcal{F}_i$ ($i=1,2$).  Since $\alpha_i'\cap\beta_i'$ is a single point, a neighborhood of $\alpha_i'\cup\beta_i'$ in $\Gamma_i\times\{1\}$ is a once-punctured torus $T_i$ and $\partial T_i$ represents the element $b_ix_ib_i^{-1}x_i^{-1}$.  

Let $\mathcal{F}_i^-=\mathcal{F}_i-\{x_i, b_i\}$ be the remaining set of $g-2$ generators in $\mathcal{F}_i$.  So $\mathcal{F}_i^-=\{[m_i], [\theta_{i4}'],\dots,[\theta_{ig}']\}$. 
The complement of $T_i$ in $\Gamma_i\times\{1\}$ is a surface $T_i^c$ with 3 boundary circles: one boundary circle is $\partial T_i$, the second boundary circle is $f_i=\partial F_i\times\{1\}$ and the third boundary circle $c_i$ is glued to $\Omega_i$.  
In our construction above, the elements in $\mathcal{F}_i^- -[m_i]$ are represented by the curves $\theta_{i4}',\dots,\theta_{ig}'$ which lie in the surface $T_i^c=(\Gamma_i\times\{1\})-\Int(T_i)$.  
Hence $[\partial T_i]$ can be generated by $\mathcal{F}_i^- -[m_i]$ plus the two elements $[f_i]$ and $[c_i]$ represented by the other two boundary circles of $T_i^c$ in $\pi_1(T_i^c)$. 
As the curve $f_i=\partial F_i\times\{1\}$ is also a boundary curve of the annulus $\Gamma\times\{1\}$ and $\partial_v N_K=\Gamma\times\{1\}$, $[f_i]$ lies in $\pi_1(F_K)$.  Since $c_i$ is a curve in $\partial \Omega_i-\gamma$ and since $\pi_1(\Omega_i-\gamma)$ can be generated by $[m_i]$ and $[\gamma']$, $[\partial T_i]$ lies in the subgroup of $\pi_1(M)$ generated by $\mathcal{F}_i^-\cup\mathcal{F}_K\cup[\gamma']$.  

Let $G_i$ be the subgroup of $\pi_1(M)$ generated by $x_i\cup \mathcal{F}_i^-\cup \mathcal{F}_K\cup [\gamma']$.  So $[\partial T_i]\in G_i$.  For simplicity, we do not include relators in any group presentation of a subgroup of $\pi_1(M)$ and write $G_i=\langle x_i, \mathcal{F}_i^-, \mathcal{F}_K, [\gamma']\rangle$.  Since $[\partial T_i]=b_ix_ib_i^{-1}x_i^{-1}\in G_i$ and $x_i\in G_i$, we have $b_ix_ib_i^{-1}\in G_i$. 

By Lemma~\ref{LXrank}, $\pi_1(X_i)$ is generated by 3 elements $x_i, h_i^{-1}x_ih_i, s_i$, where $h_i\in\pi_1(X_i)$ and $x_i$ is the element represented by $\alpha_i'$ which is the core of the annulus $A_i$ as above.  We consider the following subgroup $G$ of $\pi_1(M)$ generated by $3g+3$ elements. 
$$G=\langle x_1, b_1h_1, \mathcal{F}_1^-, \mathcal{F}_K, x_2, b_2h_2, \mathcal{F}_2^-, [\gamma'], s_1, s_2\rangle$$
Note that the generators for $G$ are simply obtained by first replacing $b_i$ by $b_ih_i$ in the $3g+1$ generators of $\pi_1(Y_s)$ and then adding in $s_1$ and $s_2$.  Our goal next is to show that $G=\pi_1(M)$.

By our construction, $G_i\subset G$ for both $i$.  
Let $y_i=b_ix_ib_i^{-1}$.  By our discussion above, $y_i\in G_i\subset G$.  Note that $x_i=b_i^{-1}y_ib_i$ and $h_i^{-1}x_ih_i=h_i^{-1}b_i^{-1}y_ib_ih_i=(b_ih_i)^{-1}y_i(b_ih_i)$.  Since $y_i\in G$ and $b_ih_i\in G$, we have $h_i^{-1}x_ih_i\in G$.  As $\pi_1(X_i)=\langle x_i, h_i^{-1}x_ih_i, s_i\rangle$ and $h_i^{-1}x_ih_i\in G$, we have $\pi_1(X_i)\subset G$ for both $i=1,2$.  

Since $h_i\in \pi_1(X_i)\subset G$ and $b_ih_i\in G$, we have $b_i\in G$.  This means that $\mathcal{F}_i\subset G$ and hence $\pi_1(Y_s)\subset G$.  Thus $\pi_1(M)=G$ and $\pi_1(M)$ can be generated by $3g+3$ elements.
\end{proof}

Although the inequality $r(M)\le 3g+3$ in Lemma~\ref{LRM} is all we need in proving Theorem~\ref{Tmain}, for certain slopes $s$, we can determine the exact rank of $\pi_1(M)$.  For completeness, we include the following proposition.

\begin{proposition}\label{Prank}
If $s=\frac{m}{2n}$ where $m$ is odd, then the rank of $\pi_1(M)$ and the dimension of $H_1(M;\mathbb{Z}_2)$ are both equal to $3g+3$.
\end{proposition}
\begin{proof}
Let $M'=X_1\cup_{A_1} Y\cup_{A_2} X_2$ as above.  So $M$ is obtained from $M'$ by Dehn surgery on $\gamma$ with slope $s$.

By Lemma~\ref{Llive}, the dimension of $H_1(X_i;\mathbb{Z}_2)$ is $2$.  Moreover, as $\pi_1(X_i)$ is generated by $x_i, h_i^{-1}x_ih_i, s_i$ and since $dim(H_1(X_i;\mathbb{Z}_2))=2$, we see that $x_i$ and $s_i$ represent two generators for $H_1(X_i;\mathbb{Z}_2)$.

By our construction, the dimension of $H_1(Y;\mathbb{Z}_2)$ is $3g+1$, and the core curves of $A_1$ and $A_2$ represent two elements in a basis of $H_1(Y;\mathbb{Z}_2)$.  Using Mayer-Vietoris sequence, it is easy to see that the dimension of $H_1(M';\mathbb{Z}_2)$ is $(3g+1)+2-1+2-1=3g+3$.  

Recall that in our framing, the $\infty$-slope is the meridional slope, so from Mayer-Vietoris sequence, we see that if the surgery slope $s=\frac{m}{2n}$, then $H_1(M;\mathbb{Z}_2)\cong H_1(M';\mathbb{Z}_2)$.  So $dim(H_1(M;\mathbb{Z}_2))$ is also $3g+3$.

As $rank(\pi_1(M))\ge dim(H_1(M;\mathbb{Z}_2))$, this means that the rank of $\pi_1(M)$ is at least $3g+3$.  By Lemma~\ref{LRM}, the rank of $\pi_1(M)$ must be equal to $3g+3$ if $s=\frac{m}{2n}$.
\end{proof}

\section{Incompressible surfaces in $Y-\gamma$}\label{Sincomp}

Our main task in the remainder of this paper is to show that the Heegaard genus of $M$ is at least $3g+4$ and by Lemma~\ref{LRM} this proves Theorem~\ref{Tmain}.  In this section, we study incompressible surfaces in $Y-\gamma$ with boundary in $A_1\cup A_2$.  This gives a major tool in calculating the Heegaard genus of $M$.

As before, we view $N_1$, $N_2$ and $N_K$ as sub-manifolds of $Y$.  We first describe a set of standard incompressible surfaces in $N_i-\gamma$ and $N_K$.

A standard incompressible surface in $N_K$ is simply a surface with boundary in $\partial_vN_K$ and that is parallel to $\partial_hN_K$.

There are a few different types of standard incompressible surfaces for $N_i-\gamma$.  

\vspace{8pt}
\noindent
\emph{Surfaces of type $A$ and type $A'$}: 
A surface of type $A$ in $N_i-\gamma$ is a surface with boundary in $\partial_vN_i$ and that is parallel in $N_i-\gamma$ to $\partial_hN_i$.  Moreover, starting from a surface of type $A$, we can perform an annulus-compression along the annulus $A_i$ (see Definition~\ref{Dtubing}) and obtain a surface with 4 boundary circles, two circles in $A_i$ and two circles in $\partial_vN_i$.  We call the surface after the annulus-compression a surface of type $A'$.  Furthermore, the boundary curves of a type $A$ (or type $A'$) surface in $\partial_vN_i$ is as shown in Figure~\ref{Ftype}(a).

\begin{figure}
  \centering
\psfrag{a}{(a)}
\psfrag{b}{(b)}
\psfrag{c}{(c)}
\psfrag{d}{(d)}
\psfrag{p}{$q_i$}
\psfrag{q}{$p_i$}
\psfrag{0}{\Small{$\partial F_i\times\{0\}$}}
\psfrag{1}{\Small{$\partial F_i\times\{1\}$}}
  \includegraphics[width=4in]{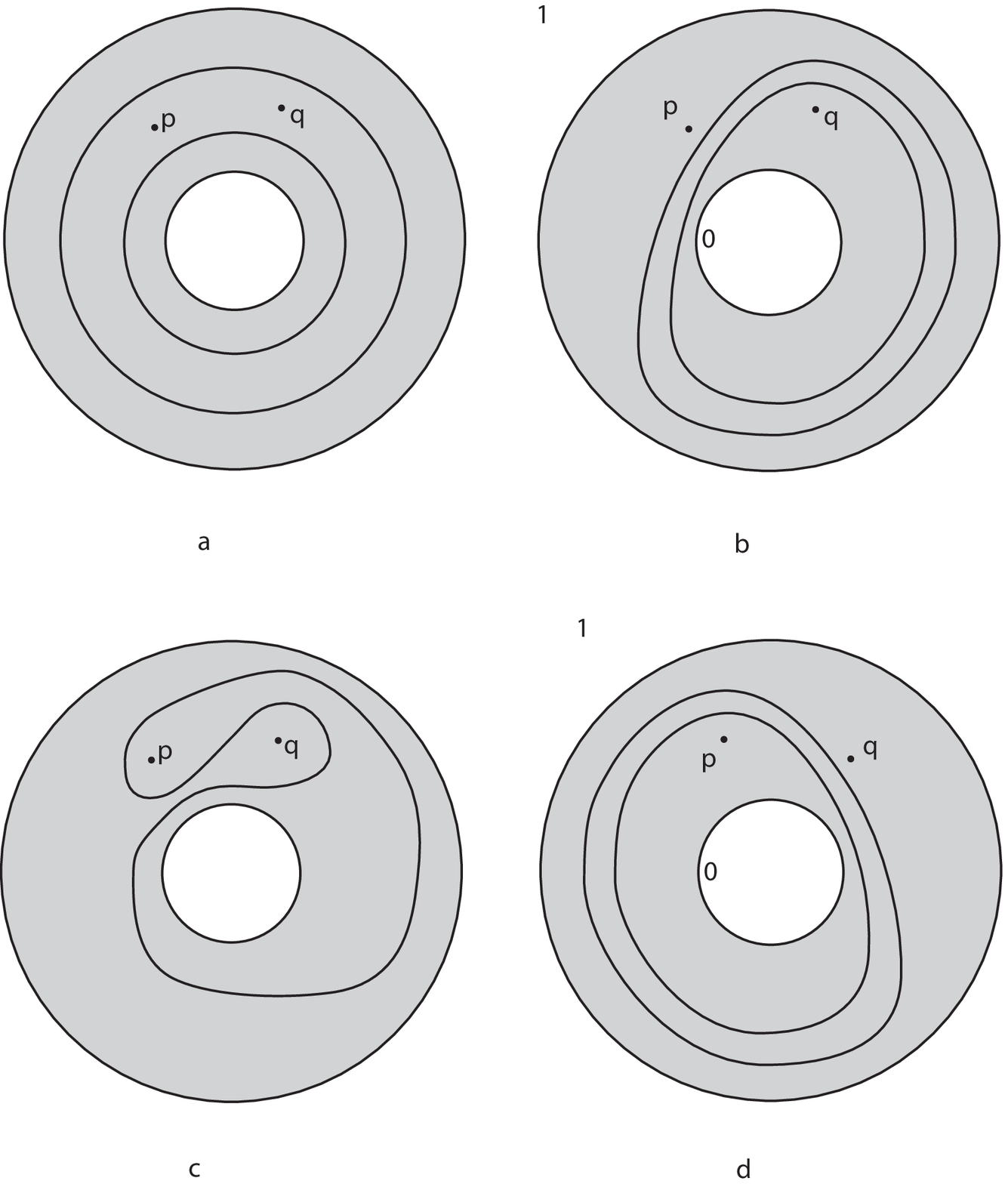}
  \caption{}
  \label{Ftype}
\end{figure}

\vspace{8pt}
\noindent
\emph{Surfaces of type $B$}: 
We first consider the arc $\gamma_i\subset F_i\subset N_i$.  Let $\pi\colon N_i\to F_i$ be the map collapsing each $I$-fiber to a point and let $V_i=\pi^{-1}(\gamma_i)$ be the vertical rectangle in $N_i$ containing $\gamma_i$.  The arc $\gamma_i$ divides $V_i$ into a pair of sub-rectangles $V_i^+$ and $V_i^-$ and we suppose $V_i^+$ is the one that intersects $A_i$.  The intersection of $V_i$ and a type $A$ surface in $N_i$ is a pair of arcs parallel to $\gamma_i$, one in each $V_i^\pm$.  Starting from a type $A$ surface $S$, we can push the arc $S\cap V_i^\pm$ across $\gamma_i$ and into $V_i^\mp$, and extend this operation to an isotopy of $S$ in $N_i$ pushing a neighborhood of $S\cap V_i^\pm$ across $\gamma_i$.  Let $S_B$ be the resulting surface.  $S_B$ also has two boundary circles in $\partial_vN_i$, but its configuration with respect to $\partial\gamma_i$ is different from $S$.  Since the arc $\gamma_i$ is orientation-reversing in $F_i$, $S_B\cap\partial_vN_i$ is as shown in Figure~\ref{Ftype}(b) or (d).  We call the resulting surface $S_B$ a type $B$ surface.

\vspace{8pt}
\noindent
\emph{Surfaces of type $B'$ and type $B''$}: 
Note that a type $B$ surface has two different possibilities.  In the isotopy above, if we push the arc $S\cap V_i^+$ into $V_i^-$, then $V_i^+$ is disjoint from the resulting type $B$ surface while $V_i^-$ intersects the surface in two arcs.  Similarly, if we push the arc $S\cap V_i^-$ into $V_i^+$, then $V_i^-$ is disjoint from the resulting type $B$ surface while $V_i^+$ intersects the surface in two arcs.  Recall that $A_i\cap V_i\subset\partial V_i^+$.  So, given a type $B$ surface $S_B$ with $V_i^-\cap S_B=\emptyset$ and $V_i^+\cap S_B$ consisting of two arcs, we can perform either a single annulus-compression along $A_i$ or two consecutive annulus-compressions along $A_i$, see Figure~\ref{Ftubing} and Figure~\ref{Fcap} for a schematic picture.  We call the surface after a single annulus-compression a type $B'$ surface and call the surface after two consecutive annulus-compressions a type $B''$ surface.  In particular, a surface of type $B'$ has 2 boundary circles in $A_i$ and a surface of type $B''$ has 4 boundary circles in $A_i$.  

Note that such an annulus-compression along $A_i$ can happen on a type $B$ surface $S_B$ only if $V_i^+\cap S_B\ne\emptyset$, in which case $V_i^+\cap S_B$ consists of two arcs and $V_i^-\cap S_B=\emptyset$. 
Recall that the two endpoints $p_i$ and $q_i$ of $\gamma_i$ in our construction are very specific.  In the construction in section~\ref{SY} and as shown in Figure~\ref{Fsurface}, the subarc of $\gamma_i$ between the endpoint $q_i$ and the point $\gamma_i\cap\alpha_i$ winds around the M\"{o}bius band $\mu_i$ once while the subarc of $\gamma_i$ bounded by $p_i$ and the point $\gamma_i\cap\alpha_i$ lies in the orientable sub-surface $\Gamma_i$ of $F_i$ ($i=1,2$). 
Let $K_i$ be the subarc of an $I$-fiber of $\partial_vN_i$ connecting $q_i$  to $\partial(\Gamma\times\{0\})$, see Figure~\ref{Fgamma}(b).  Since $A_i\subset\Gamma_i\times\{1\}$, by our construction of $\gamma_i$, $p_i$ and $q_i$, the arc $K_i$ is an edge of $V_i^+$.  So both boundary curves of a type $B''$ (or type $B'$) surface in $\partial_vN_i$ intersect this arc $K_i$.  In other words, the configuration of the pair of boundary curves of a type $B''$ (or type $B'$) surface in $\partial_vN_i$ is Figure~\ref{Ftype}(b), not Figure~\ref{Ftype}(d)

\vspace{8pt}
\noindent
\emph{Tubes around $\gamma_i$}:  
The last type of standard surfaces in $N_i-\gamma$ is a tube around $\gamma_i$, i.e. an annulus parallel to $T\cap N_i$, where $T=\partial \overline{N(\gamma)}$.

\begin{lemma}\label{Lstandard}
Let $P$ be a compact orientable incompressible surface properly embedded in $Y-\gamma$ and with $\partial P\subset A_1\cup A_2$.  Suppose $P$ is not a $\partial$-parallel annulus.  Then, after isotopy, each component of $P\cap N_K$, $P\cap (N_1-\gamma)$, and $P\cap (N_2-\gamma)$ is one of the standard surfaces described above.
\end{lemma}
\begin{proof}
First, since $\partial_vN_K$ is an incompressible annulus in $Y-\gamma$, after isotopy, we may assume that $P\cap \partial_vN_K$ consists of curves essential in both $P$ and $\partial_vN_K$.  Moreover, by assuming $|P\cap\partial_vN_K|$ is minimal, we may assume no component of $P\cap N_K$ is an annulus parallel to a sub-annulus of $\partial_vN_K$.  Thus, by Lemma~\ref{LIbundle}, $P\cap N_K$ is standard.

As the core curve of the annulus  $\partial_vN_i$ is essential in $Y$, $\partial_vN_i$ is incompressible in $Y$.  By our construction, $\gamma$ intersects $\partial_vN_i$ in two points $p_i$ and $q_i$.  Since $\gamma$ is essential in $Y$ and $\gamma$ minimally intersects $\partial_vN_i$, $\partial_vN_i-\gamma$ is incompressible in $Y-\gamma$.

We may assume our incompressible surface $P$ is transverse to $\partial_vN_i$.  Let $P_i=P\cap N_i$.  As both $P$ and $\partial_vN_i-\gamma$ are incompressible in $Y-\gamma$, after isotopy, we may assume
\begin{enumerate}
  \item no component of $P\cap \partial_vN_i$ is trivial in $\partial_vN_i-\gamma$, and
  \item $P_i$ contains no $\partial$-parallel annulus in $N_i-\gamma$ that is parallel to a sub-annulus of $\partial_vN_i-\gamma$ (if there is such a component, then one can push it across $\partial_vN_i-\gamma$ into $\Gamma\times I$, which reduces $|P\cap \partial_vN_i|$).
\end{enumerate} 

We may assume $\chi(P_i)$ is maximal up to isotopy on $P$ and under the two conditions above.  Since $P$ is incompressible in $Y-\gamma$, this implies that $P_i$ is incompressible in $N_i-\gamma$.  Note that in Condition (2), pushing a $\partial$-parallel annulus from one side of $\partial_vN_i$ to the other side does not change $\chi(P_i)$.

\begin{claim}\label{Claim0}
$P_i$ admits no $\partial$-compressing disk in $N_i-\gamma$ with base arc in $\partial_vN_i$, see Definition~\ref{Dbase} for the definition of base arc.
\end{claim}
\begin{proof}[Proof of Claim~\ref{Claim0}]
Suppose the claim is false and let $D$ be a $\partial$-compressing disk of $P_i$ with base arc $\alpha$ in $\partial_vN_i$.  Let $P_i^D$ be the surface obtained by  a $\partial$-compression on $P_i$ along $D$.  We will show next that no curve in $\partial P_i^D$ is trivial in $\partial_vN_i-\gamma$.

Suppose a curve in $\partial P_i^D$ is trivial in $\partial_vN_i-\gamma$.  Then this means that $\partial P_i\cap\partial_vN_i$ and the base arc $\alpha$ of $D$ have two possible configurations: 
\begin{enumerate}[\upshape (i)]
  \item  $\partial\alpha$ lies in the same circle of $\partial P_i$, and $\alpha$ is parallel in $\partial_vN_i-\gamma$ to a subarc of $\partial P_i$ that is bounded by $\partial\alpha$; 
  \item  the two endpoints of $\alpha$ lie in different circles of $\partial P_i$, and the two circles of $\partial P_i$ containing $\partial\alpha$ are parallel in $\partial_vN_i-\gamma$.  
\end{enumerate}
In case (i), one can push $\alpha$ into $\partial P_i$ and change $D$ into a compressing disk for $P_i$, which contradicts that $P_i$ is incompressible in $N_i-\gamma$.  In case (ii), similar to the proof of Lemma~\ref{Lmonogon}, the component of $P_i$ containing $\partial\alpha$ must be an annulus parallel to a sub-annulus of $\partial_vN_i-\gamma$, and this contradicts the Condition (2) above.  Thus no curve in $\partial P_i^D$ is trivial in $\partial_vN_i-\gamma$, i.e, $P_i^D$ satisfies Condition (1) above.

After pushing any possible $\partial$-parallel annulus from $N_i-\gamma$ into $\Gamma\times I$ as described in Condition (2) above, we may assume $P_i^D$ also satisfies Condition (2).  However, $\chi(P_i^D)>\chi(P_i)$ and this contradicts our assumption that $\chi(P_i)$ is maximal.
\end{proof}

Next we consider the curves in $\partial P_i\cap\partial_vN_i$.  Since $\gamma\cap\partial_vN_i$ contains two points, we have the following 4 types of curves:  a type I curve is a horizontal curve transverse to the $I$-fibers of $\partial_vN_i$; a type II curve is a curve bounding a disk in $\partial_vN_i$ and the disk contains exactly one endpoint of $\gamma_i$; a type III curve is a curve bounding a disk in $\partial_vN_i$ and the disk contains both endpoints of $\gamma_i$; a type IV curve is a curve that is essential in the annulus $\partial_vN_i$ but cannot be isotoped in $\partial_vN_i-\gamma$ to be transverse to the $I$-fibers of $\partial_vN_i$, see Figure~\ref{Ftype}(c) for a picture. 

\begin{claim}\label{Claimtype}
Type III and type IV curves cannot occur.  
\end{claim}
\begin{proof}[Proof of Claim~\ref{Claimtype}]
Recall that in our construction, the two endpoints of $\gamma_i$ are $p_i$ and $q_i$ in $\partial_vN_i$.  Let $I_1$ and $I_2$ be a pair of $I$-fibers of $N_i$ lying in $\partial_vN_i-\gamma$ such that  $I_1$ and $I_2$ divide $\partial_vN_i$ into a pair of rectangles $R_p$ and $R_q$ with $p_i\in R_p$ and $q_i\in R_q$.  

We may assume $|\partial P_i\cap I_1|$ and $|\partial P_i\cap I_2|$ are minimal up to isotopy on $\partial P_i$ in $\partial_vN_i-\gamma$.  If a component $\xi$ of $\partial P_i\cap R_p$ is an arc with both endpoints in $I_1$, then by the minimality assumption on $|\partial P_i\cap I_1|$, the subdisk of $R_p$ bounded by $\xi$ and a subarc of $I_1$  must contain the puncture $p_i$.  Moreover, if there is such an arc $\xi$, then every other component of $\partial P_i\cap R_p$ is either an arc transverse to the $I$-fibers or an arc parallel to $\xi$ in $R_p-p_i$ or a circle around $p_i$.  Therefore, if such an arc $\xi$ exists, then $|P\cap I_1|\ne |P\cap I_2|$.

After applying the argument above also to $P\cap R_q$, we can conclude that either (1) $|P\cap I_1|\ne |P\cap I_2|$, or (2) each component of $P\cap R_p$ and $P\cap R_q$ is either transverse to the $I$-fibers or a circle around an endpoint of $\gamma_i$.  In case (2), each component of $P\cap\partial_vN_i$ is of either type I or type II and the claim holds.  

Suppose the claim is false, then we have $|P\cap I_1|\ne |P\cap I_2|$. 
Since $g\ge 3$, there is a vertical rectangle $R_I$ properly embedded in $N_i$ such that $R_I\cap\gamma_i=\emptyset$, $R_I\cap A_i=\emptyset$, and $I_1\cup I_2$ is a pair of opposite edges of $R_I$.  Since $P_i$ is incompressible in $N_i-\gamma$, after isotopy, we may assume $P\cap R_I$ contains no closed curves.  As $R_I\cap A_i=\emptyset$, the endpoints of all the arcs in $P\cap R_I$ lie in $I_1\cup I_2$.  Since $|P\cap I_1|\ne|P\cap I_2|$, there must be an arc in $P\cap R_I$ having both endpoints in the same $I$-fiber $I_1$ or $I_2$.  As $R_I\cap\gamma=\emptyset$, this means that a subdisk of $R_I$ is a $\partial$-compressing disk for $P_i$ in $N_i-\gamma$ with its base arc in $I_1$ or $I_2$, contradicting Claim~\ref{Claim0}.  
\end{proof}

By Claim~\ref{Claimtype}, we see that $P\cap\partial_vN_i$ consists of type I and type II curves.

\begin{claim}\label{claimII}
Each type II curve in $\partial_vN_i$ is a boundary circle of a tube around $\gamma_i$.
\end{claim}
\begin{proof}[Proof of Claim~\ref{claimII}]
As in the construction in section~\ref{SY}, the two endpoints of $\gamma_i$ are $p_i$ and $q_i$.  So a type II curves in $\partial_vN_i$ is a circle around either $p_i$ or $q_i$. 
Let $I_p$ and $I_q$ be the pair of $I$-fibers of $N_i$ containing $p_i$ and $q_i$ respectively.  Each type I curve of $P\cap\partial_vN_i$ is transverse to the $I$-fibers and hence intersects $I_p$ (and $I_q$) in one point.  After isotopy, we may assume each type II curve around $p_i$ (resp. $q_i$) intersects $I_p$ (resp. $I_q$) in two points and is disjoint from $I_q$ (resp. $I_p$).

Let $V_i$ and $V_i^\pm$ be the vertical rectangles as in the discussions of type $B$ surfaces before the lemma.  So the $I$-fibers $I_p$ and $I_q$ above are a pair of opposite edges of $V_i$.  Since $P$ is incompressible, after isotopy, we may assume $P$ is transverse to $V_i$ and $P\cap V_i$ contains no trivial circle.   

By Claim~\ref{Claim0}, we may assume that there is no arc in $P\cap V_i$ having both endpoints in the same $I$-fiber $I_p$ or $I_q$.  By the definition of $V_i^\pm$, we know that $V_i^-\cap A_i=\emptyset$.  Hence each component of $P\cap V_i^-$ is an arc connecting $I_p$ to $I_q$.  

Note that if $P\cap V_i^-=\emptyset$, then $P\cap\partial_vN_i$ contains no type II curve and the claim holds vacuously. 
Suppose $P\cap V_i^-\ne\emptyset$ and let $c$ be the arc in $P\cap V_i^-$ that is the closest to $\gamma_i$, i.e., the sub-rectangle of $V_i^-$ between $c$ and $\gamma_i$ contains no other component of $P\cap V_i^-$.  

Let $c_1$ and $c_2$ be the components of $P\cap \partial_vN_i$ containing the two endpoints of $c$.  
The first case is that both $c_1$ and $c_2$ are of type I (note that this case includes the possibility $c_1=c_2$).  Since $c$ is the closest to $\gamma_i$ in $V_i^-$, this means that $P\cap\partial_vN_i$ contains no type II circles and the claim holds vacuously in this case.

The second case is that one circle, say $c_2$, is a type II circle and the other circle $c_1$ is of type I.  Without loss of generality, we may suppose $c_2$ is a circle around $q_i$.  Let $P_c$ the closure of a small neighborhood of $c_2\cup c$ in $P_i$. 
As shown in Figure~\ref{Ftube}(a), the arc $l_c=\overline{\partial P_c-\partial_vN_i}$ is properly embedded in $N_i$ with $\partial l_c\subset c_1$. Moreover, $l_c$ is parallel in $N_i-\gamma$ to an arc in $\partial_vN_i$ that goes around $p_i$.  This means that $l_c$ is a boundary arc of a $\partial$-compressing disk for $P_i$ in $N_i-\gamma$, contradicting Claim~\ref{Claim0}.

\begin{figure}
  \centering
\psfrag{a}{(a)}
\psfrag{b}{(b)}
\psfrag{p}{$p_i$}
\psfrag{q}{$q_i$}
\psfrag{1}{$c_1$}
\psfrag{2}{$c_2$}
\psfrag{c}{$c_q$}
\psfrag{g}{$\gamma_i$}
\psfrag{P}{$P_i$}
\psfrag{l}{$l_c$}
\psfrag{Q}{$Q_c$}
  \includegraphics[width=4in]{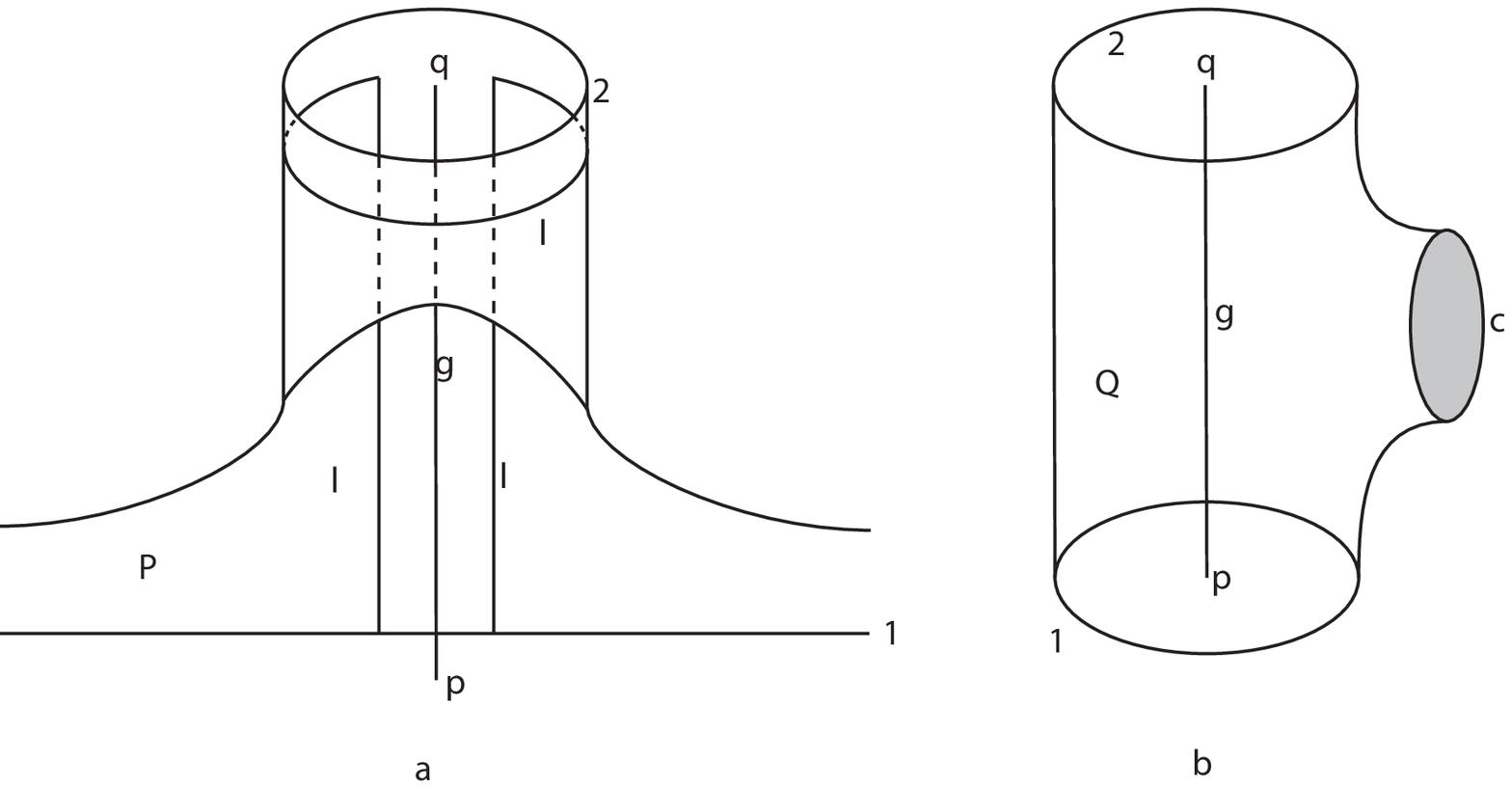}
  \caption{}
  \label{Ftube}
\end{figure}

The remaining case is that both $c_1$ and $c_2$ are type II curves.  Let $Q_c$ be the closure of a small neighborhood of $c\cup c_1\cup c_2$ in $P_i$.  So $Q_c$ is a pair of pants.  Let $c_q$ be the component of $\partial Q_c$ that lies in $\Int(N_i)$.
  As shown in Figure~\ref{Ftube}(b), $c_q$ bounds a disk in $N_i-\gamma$.  Since $P_i$ is incompressible in $N_i-\gamma$, $c_q$ also bounds a disk in $P_i$.  Hence the component of $P_i$ containing $c$ is a tube around $\gamma_i$.

After removing the innermost tube around $\gamma_i$ and repeating the argument above, we can inductively conclude that each type II curve in $P\cap\partial_vN_i$ is a boundary circle of a tube around $\gamma_i$.  Hence the claim holds.
\end{proof}

Now we delete all the components of $P_i$ that are tubes around $\gamma_i$ and denote the resulting surface by $P_i'$.  By Claim~\ref{claimII}, $\partial P_i'\cap\partial_vN_i$ consists of type I curves.  Hence $|I_p\cap\partial P_i'|=|I_q\cap\partial P_i'|$, where $I_p$ and $I_q$ are the $I$-fibers containing $p_i$ and $q_i$ as in the proof of Claim~\ref{claimII}.  Let $\lambda_i=A_i\cap\partial V_i$.  So $\lambda_i$ lies in $\partial V_i^+$ and is an essential arc in $A_i$.  

By Lemma~\ref{Lmonogon}, if there is an arc in $P_i'\cap V_i$ with both endpoints in $\lambda_i$, then $P_i'$ must be a $\partial$-parallel annulus, contradicting our hypotheses on $P$.  So each component of $P_i'\cap V_i$ is an arc either connecting $I_p$ to $I_q$ or connecting $\lambda_i$ to $I_p\cup I_q$.  So after isotopy, we may suppose the arcs in $P_i'\cap V_i$ are transverse to the $I$-fibers of $V_i$.  Moreover, since $|I_p\cap\partial P_i'|=|I_q\cap\partial P_i'|$, the number of arcs in $P_i'\cap V_i$ connecting $\lambda_i$ to $I_p$ is equal to the number of arcs in $P_i'\cap V_i$ connecting $\lambda_i$ to $I_q$.

Let $\Lambda_i=\pi^{-1}(\pi(A_i))$ be the union of the $I$-fibers of $N_i$ that meet $A_i$.  So $\Lambda_i$ is of the form $annulus\times I$ and its vertical boundary $\partial_v\Lambda_i$ is a pair of vertical annuli in $N_i$ which we denote by $E_1$ and $E_2$.  Let $J_1=E_1\cap V_i$, $J_2=E_2\cap V_i$ and $\Delta=V_i\cap\Lambda_i$.  So $\Delta$ is a vertical rectangle and the pair of $I$-fibers $J_1$ and $J_2$ are a pair of opposite edges of $\Delta$.  Moreover, $E_j\cap\gamma$ ($j=1,2$) is a single point lying in $J_j$.  

We may suppose $P_i'$ is transverse to $E_1$ and $E_2$.  Since $P_i'$ is incompressible, after isotopy, we may assume no component of $P_i'\cap (E_1-J_1)$ and $P_i'\cap (E_2-J_2)$ is a closed curve in the disks $E_1-J_1$ and $E_2-J_2$ respectively.  If there is a subarc $e$ of $P_i'\cap E_1$ such that $e\cap J_1=\partial e$, then as shown in Figure~\ref{Fcube}(a), the union of $e$ and the two arcs of $P_i'\cap \overline{V_i-\Delta}$ that are incident to $\partial e$ is an arc $l_e$ properly embedded in $N_i$ and parallel to a vertical arc in $\partial_vN_i$.  Since no component of $P_i'$ is a tube around $\gamma_i$, the two endpoints of $l_e$ lie in different curves in $\partial P_i'\cap \partial_vN_i$ and hence $l_e$ is a boundary arc of a $\partial$-compressing disk for $P_i'$ in $N_i-\gamma$, which contradicts Claim~\ref{Claim0}.  Thus $P_i'\cap E_1$ contains no such arc $e$ and this implies that each component of $P_i'\cap E_1$ intersects $J_1$ in a single point.  So after isotopy, each component of $P_i'\cap E_1$ is transverse to the $I$-fibers.  Similarly, after isotopy, each component of $P_i'\cap E_2$ is also transverse to the $I$-fibers.

\begin{figure}
  \centering
\psfrag{a}{(a)}
\psfrag{b}{(b)}
\psfrag{A}{$A_i$}
\psfrag{E}{$E_1$}
\psfrag{J}{$J_1$}
\psfrag{e}{$e$}
\psfrag{l}{$l_e$}
\psfrag{L}{$\Lambda_i$}
\psfrag{g}{$\gamma_i$}
\psfrag{V}{$V_i$}
\psfrag{+}{$\Delta^+$}
\psfrag{-}{$\Delta^-$}
  \includegraphics[width=4in]{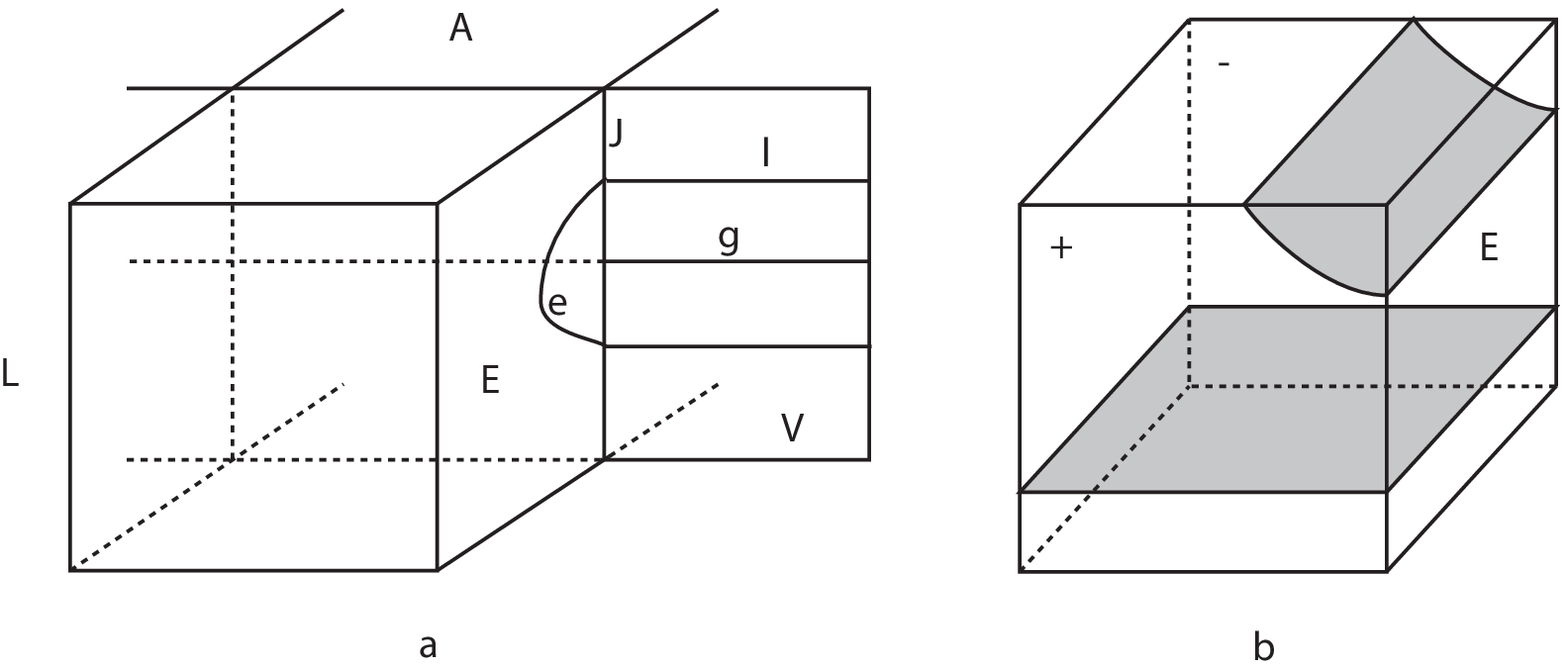}
  \caption{}
  \label{Fcube}
\end{figure}

Let $W_i$ be the closure (under path metric) of $N_i-(V_i\cup\Lambda_i)$.  So $W_i$ is an $I$-bundle over a compact surface.  By our assumptions above on $P_i'\cap V_i$, $P_i'\cap E_1$ and $P_i'\cap E_2$, we see that $P_i'\cap \partial_vW_i$ is a collection of curves transverse to the $I$-fibers of $W_i$.  By Claim~\ref{Claim0}, we may assume no component of $P_i'\cap W_i$ is an annulus parallel to a sub-annulus of $\partial_vW_i$.  So, by Lemma~\ref{LIbundle},  $P_i'\cap W_i$ is horizontal (i.e., transverse to the $I$-fibers). 

Next we consider $P_i'\cap\Lambda_i$.  By our assumption on $P_i'\cap V_i$ above, $P_i'\cap\Delta$ consists of arcs that either connect $J_1$ to $J_2$ or connect the edge $\lambda_i=A_i\cap\partial\Delta$ to $J_1\cup J_2$.  Now we cut $\Lambda_i$ open along $\Delta$, and the resulting manifold $C$ can be viewed as a cube with two opposite faces $\Delta^+$ and $\Delta^-$ corresponding to the two sides of $\Delta$.  Recall that $P_i'\cap A_i$, $P_i'\cap E_1$ and $P_i'\cap E_2$ all consist of essential simple closed curves in the annuli $A_i$, $E_1$, $E_2$ respectively.  Hence, as shown in Figure~\ref{Fcube}(b), each closed curve in $P_i'\cap\partial C$ consists of 4 edges with a pair of opposite edges lying in $\Delta^+$ and $\Delta^-$.  Since $\gamma\cap\Lambda_i\subset\Delta$, $\gamma\cap\Int(C)=\emptyset$.
  Since $P_i'$ is incompressible in $N_i-\gamma$, as shown in Figure~\ref{Fcube}(b), each curve of $P_i'\cap\partial C$ must bound a disk in $P_i'$, and by our assumptions on $P_i'\cap\partial\Lambda_i$, this disk bounded by $P_i'\cap\partial C$ must lie in $C$.  So $P_i'\cap C$ is a collection of quadrilaterals.  After gluing $\Delta^+$ to $\Delta^-$ and changing $C$ back to $\Lambda_i$, the pair of opposite edges in $\Delta^+\cup\Delta^-$ of each quadrilateral are identified, which yields an annulus.  Thus $P_i'\cap\Lambda_i$ consists of annuli, and after isotopy, these annuli are transverse to the $I$-fibers of $\Lambda_i$.  By our conclusion above on $P_i'\cap W_i$, this means that $P_i'$ is transverse to the $I$-fibers of $N_i$.

Recall that we have concluded earlier that the number of components of $P_i'\cap V_i$ connecting $\lambda_i=A_i\cap V_i$ to $I_p$ is equal to the number of components of $P_i'\cap V_i$ connecting $\lambda_i$ to $I_q$.  This implies that, as shown in Figure~\ref{Fcap} (also see Figure~\ref{Ftubing}), we can perform tubing on $P_i'$ along the annulus $A_i$ (see Definition~\ref{Dtubing}) to get an embedded surface $P_i''$ disjoint from $A_i$.  Moreover, by our conclusions on $P_i'\cap V_i$ above and as shown in Figure~\ref{Fcap}, we can perform the tubing in a nested way so that, after isotopy, each component of $P_i''$ is still transverse to the $I$-fibers of $N_i$.  Conversely,  $P_i'$ can be obtained from $P_i''$ by annulus-compressions along $A_i$. Since $P_i'$ is orientable, by our construction of $N_i$ and $A_i$, $P_i''$ must also be orientable.  So by the properties of $P_i'\cap V_i$ above, it is easy to see that each component of $P_i''$ is of either type $A$ or type $B$ which we described at the beginning of this section.   Thus each component of $P_i'$ is standard is this sense and Lemma~\ref{Lstandard} holds.
\end{proof}

\begin{figure}
  \centering
\psfrag{A}{$A_i$}
\psfrag{g}{tubing along $A_i$}
\psfrag{c}{annulus-compression}
  \includegraphics[width=4.5in]{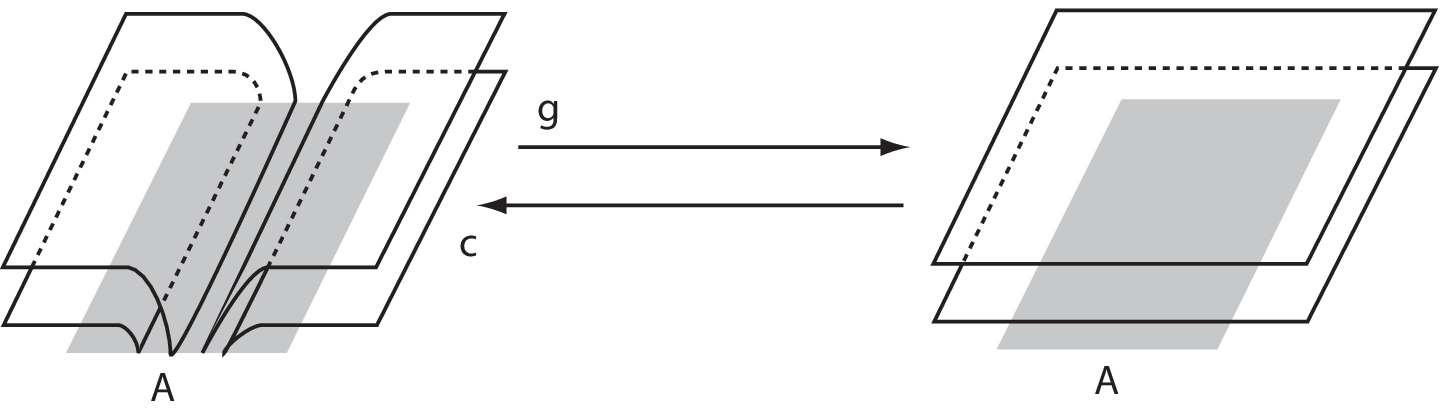}
  \caption{}
  \label{Fcap}
\end{figure}

Let $P$ be an incompressible surface in $Y-\gamma$ with $\partial P\subset A_1\cup A_2$.  By Lemma~\ref{Lstandard}, we may assume $P\cap N_1$, $P\cap N_2$ and $P\cap N_K$ are all standard.    Our next goal is to show that $P$ is also standard in the complement of $N_1\cup N_2\cup N_K$.

Recall that the complement of $N_1\cup N_2\cup N_K$ in $Y$ is the handlebody obtained by attaching the 1-handle $D^2\times I$ to the solid torus $\Gamma\times I$ at $\Gamma\times\{0\}$.  

We first describe a collection of standard surfaces properly embedded in $\Gamma\times I $.  Let $c_1$ and $c_2$ be a pair of disjoint essential simple closed curves in the annuli $\partial_vN_1$, $\partial_vN_2$ and $\partial_vN_K$, transverse to the $I$-fibers of these annuli.  Suppose $c_1\cup c_2$ is disjoint from $\gamma$.  Then $c_1\cup c_2$ bounds an annulus $\Gamma_c$ properly embedded in $\Gamma\times I$.   We call such an annulus $\Gamma_c$ a \emph{standard (punctured) annulus} in $\Gamma\times I$ if 
\begin{enumerate}
  \item $c_1$ and $c_2$ lie in two different annuli of $\partial_vN_1$, $\partial_vN_2$ and $\partial_vN_K$, and
  \item $|\Gamma_c\cap\gamma|$ is minimal up to isotopy on $\Gamma_c$ in $\Gamma\times I$ while fixing $c_1\cup c_2$.
\end{enumerate}    
Next we consider a collection of disjoint standard punctured annuli in $\Gamma\times I$.  As shown in Figure~\ref{Fpuncture}(a), we can add tubes to these annuli along $\gamma$ to get a surface properly embedded in $(\Gamma\times I)-\gamma$.  There are certainly more than one way to add such tubes.   As in Figure~\ref{Fpuncture}(a), at each puncture, there are two possible directions to locally add a tube, and these tubes may be nested.  We call the surface in $(\Gamma\times I)-\gamma$ after such tubing on a collection of standard punctured annuli a \emph{standard surface} in $(\Gamma\times I)-\gamma$.

\begin{figure}
  \centering
\psfrag{a}{(a)}
\psfrag{b}{(b)}
\psfrag{t}{$\gamma_t$}
\psfrag{g}{$\gamma$}
\psfrag{d}{$\widehat{\Delta}_t$}
  \includegraphics[width=4in]{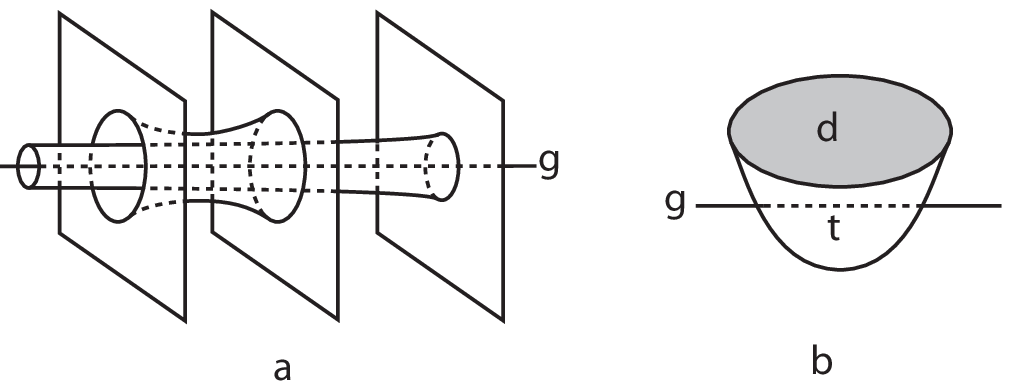}
  \caption{}
  \label{Fpuncture}
\end{figure}

\begin{lemma}\label{LAS}
Let $P$ be a compact orientable incompressible surface properly embedded in $Y-\gamma$ with $\partial P\subset A_1\cup A_2$.  Suppose $P$ is not a $\partial$-parallel annulus.  Then after isotopy,
\begin{enumerate}
  \item $P\cap N_1$, $P\cap N_2$, $P\cap N_K$ are standard as in Lemma~\ref{Lstandard}, and
  \item the intersection of $P$ with the added 1-handle $D^2\times I$ consists of tubes around the arc $\gamma_h=\gamma\cap (D^2\times I)$, and
  \item $P\cap (\Gamma\times I)$ is a standard surface in $(\Gamma\times I)-\gamma$.
\end{enumerate}  
\end{lemma}
\begin{proof}
By Lemma~\ref{Lstandard}, we may assume $P\cap N_1$, $P\cap N_2$, $P\cap N_K$ are standard.  Moreover, as $P$ is incompressible in $Y-\gamma$ and since $\gamma\cap (D^2\times I)$ is the core of the 1-handle, after isotopy, we may assume part (2) of the lemma also holds.

Next we consider a curve $c$ in $P\cap (\Gamma\times I)$ such that $c$ bounds a disk $D_c$ in $\Gamma\times I$ with the properties that $D_c\cap\gamma$ is a single point and $D_c\cap P=c$.  Note that $c$ must be an essential curve in $P$ because otherwise, the union of $D_c$ and the subdisk of $P$ bounded by $c$ is a 2-sphere intersecting $\gamma$ in a single point, which contradicts Lemma~\ref{LRi}.  We call the disk $D_c$ a once-punctured compressing disk.  We can compress $P$ along the disk $D_c$ to get a surface which intersects $\gamma$.  Note that we can add a tube as in Figure~\ref{Fpuncture}(a) on the surface after this compression to get back our surface $P$.  Now we maximally compress $P$ along once-punctured compressing disks in $\Gamma\times I$, and delete any resulting 2-sphere component which bounds a 3-ball $B$ with $B\cap\gamma$ an unknotted arc in $B$.  Let $P'$ be the surface after this operation.  It follows from our construction that $P$ can be obtained from $P'$ by tubing along $\gamma$.  Moreover, since $P$ is incompressible in $Y-\gamma$ and since the curve $c$ above is essential in $P$, we can inductively conclude that, after these compressions along once-punctured compressing disks, $P'-\gamma$ is incompressible in $Y-\gamma$.

Since the intersection of $P$ with the 1-handle $D^2\times I$ consists of tubes around the arc $\gamma\cap(D^2\times I)$, after the compressions above and some isotopy, we may assume $P'$ is disjoint from the 1-handle.  Similarly, since $P\cap N_1$, $P\cap N_2$, $P\cap N_K$ are standard, after the compressions above and some isotopy, we may assume $P'\cap N_1$, $P'\cap N_2$, $P'\cap N_K$ are standard but contain no tubes around $\gamma_i$.  In particular, $P'\cap \partial(\Gamma\times I)$ consists of essential curves in the 3 annuli $\partial_vN_1$, $\partial_vN_2$ and $\partial_vN_K$ transverse to the $I$-fibers of these annuli, and all the punctures of $P'\cap\gamma$ lie in $\Gamma\times I$.  

\begin{claimD}\label{claimD1}
$P'$ has no trivial intersection with $\gamma$.  More precisely, $\gamma$ has no subarc $\gamma_t$ such that $\gamma_t\cap P'=\partial\gamma_t$ and $\gamma_t$ can be isotoped into $P'$ (fixing $\partial\gamma_t$), see Figure~\ref{Fpuncture}(b) for a picture.
\end{claimD}
\begin{proof}[Proof of Claim~\ref{claimD1}]
Suppose the claim is false and there is such a subarc $\gamma_t$ of $\gamma$.  If this happens, then there is an embedded disk $\Delta_t$ with $\partial\Delta_t=\gamma_t\cup p_t$ where $p_t\subset P'$, $\partial\gamma_t=\partial p_t$, $\Delta_t\cap\gamma=\gamma_t$ and $\Delta_t\cap P'=p_t$.    As shown in Figure~\ref{Fpuncture}(b), we can take two copies of $\Delta_t$ and perturb them into a disk $\widehat{\Delta}_t$ in $Y-\gamma$ such that $\partial\widehat{\Delta}_t=\widehat{\Delta}_t\cap P'$ and $\partial\widehat{\Delta}_t$ bounds a disk in $P'$ containing the two punctures $\partial\gamma_t$.  Since $P'-\gamma$ is incompressible in $Y-\gamma$, $\partial\widehat{\Delta}_t$ must bound a disk in $P'-\gamma$ and hence the component of $P'$ containing $p_t$ must be a twice-punctured 2-sphere bounding a 3-ball in which $\gamma_t$ is an unknotted arc.  This contradicts our assumption that $P'$ has no such component.  
\end{proof}

\begin{claimD}\label{claimD2}
There is no $\partial$-compressing disk $\Delta'\subset (\Gamma\times I)-\gamma$ for $P'\cap (\Gamma\times I)$ such that its base arc (see Definition~\ref{Dbase}) $\delta'$ is a vertical arc in $\partial_vN_i$ (or $\partial_vN_K$)
\end{claimD}
\begin{proof}[Proof of Claim~\ref{claimD2}]
Suppose there is such a $\partial$-compressing disk $\Delta'$ whose base arc $\delta'$ is a vertical arc in $\partial_vN_i$ (or $\partial_vN_K$).
  Since $P'\cap N_i$ and $P'\cap N_K$ are assumed to be standard and since $P'\cap N_i$ contains no tube around the arc $\gamma_i$, $P'\cap N_i$ and $P'\cap N_K$ are transverse to the $I$-fibers of $N_i$ and $N_K$ respectively.  So $P'\cap N_i$ and $P'\cap N_K$ divide $N_i$ and $N_K$ respectively into a collection of sub-$I$-bundles.  As shown in Figure~\ref{Fmono}(a) and since $g\ge 3$, one can form a compressing disk $D$ for $P'-\gamma$ in $Y-\gamma$ by connecting two parallel copies of $\Delta'$ using an essential vertical rectangle (in these sub-$I$-bundles) that is disjoint from $\gamma_i$ and $A_i$ (note that in this construction, $\partial D$ must be essential in $P'-\gamma$ since the curves of $P'\cap\partial_vN_i$ and $P'\cap\partial_vN_K$ are essential in $Y-\gamma$ and hence essential in $P'$).  This contradicts our conclusion that $P'-\gamma$ is incompressible in $Y-\gamma$.
\end{proof}

Recall that $\gamma\cap(\Gamma\times I)$ consists of 3 arcs: the arc $\gamma_p$ connecting $p_1$ to $p_2$ and a pair of arcs $\delta_i$ ($i=1,2$) connecting $q_i$ to the center of the disk $D_j$ ($j=0,1$), see Figure~\ref{Fgamma}(a, b).  By our construction and as shown in Figure~\ref{Fgamma}(b), there is an embedded triangular disk $\Delta_i$ ($i=1,2$) such that the 3 edges in $\partial\Delta_i$ are 
$\delta_i$, $K_i=\Delta_i\cap\partial_vN_i$ and the arc $\Delta_i\cap(\Gamma\times\{0\})$, where $K_i$ is the subarc of an $I$-fiber of $\partial_vN_i$ connecting $q_i$ to $\partial(\Gamma\times\{0\})$.  Since $P'-\gamma$ is incompressible in $Y-\gamma$, after isotopy, we may assume that $P'\cap\Delta_i$ contains no closed curve. 
As $P'\cap (\Gamma\times\{0\})=\emptyset$, by Claim~\ref{claimD1} and Claim~\ref{claimD2}, each component of  $P'\cap\Delta_i$ must be an arc connecting $K_i$ to $\delta_i$. 

Similarly, the arc $\gamma_p$ is also $\partial$-parallel in $\Gamma\times I$.  As shown in Figure~\ref{Fgamma}(b), there is a rectangle $R$ embedded in $\Gamma\times I$ and the 4 edges of $\partial R$ are: $\gamma_p$, $\zeta_K=R\cap\partial_vN_K$, $\zeta_i=R\cap\partial_vN_i$ ($i=1,2$), where $\zeta_K$ is an essential arc properly embedded in $\partial_vN_K$ and $\zeta_i$ a vertical arc in $\partial_vN_i$.  We may assume $R$ is transverse to $P'$.  Since $P'-\gamma$ is incompressible in $Y-\gamma$, after isotopy, we may suppose $R\cap P'$ contains no closed curve.  By  Claim~\ref{claimD1} and Claim~\ref{claimD2}, no arc in $R\cap P'$ has both endpoints in the same edge of $\partial R$.

Let $B_1$, $B_2$ and $B_R$ be small neighborhoods of $\Delta_1$, $\Delta_2$ and $R$ respectively in $\Gamma\times I$. So $B_1$, $B_2$ and $B_R$ are 3-balls and we may assume each component of $P'\cap B_1$, $P'\cap B_2$ and $P'\cap B_R$ is a disk neighborhood in $P'\cap(\Gamma\times I)$ of an arc in $P'\cap\Delta_1$, $P\cap\Delta_2$ and $P\cap R$ respectively. 
Let $V=\overline{(\Gamma\times I)-(B_1\cup B_2\cup B_R)}$.  Clearly $V$ is a solid torus and $V\cap\gamma=\emptyset$.  

Recall that $P'\cap\partial(\Gamma\times I)$ is a collection curves parallel in $\Gamma\times I$ to the core of the solid torus $\Gamma\times I$.   By our conclusions on $P'$, each curve in $P'\cap\partial(\Gamma\times I)$ intersects $\partial R$ at most once and intersects $\Delta_1\cup\Delta_2$ at most once.  

Let $\xi$ be a component of $P'\cap (\Delta_1\cup\Delta_2\cup R)$.  By our construction, $\xi$ is either (1) an arc connecting $\partial(\Gamma\times I)$ to $\gamma$, or (2) an arc in $P'\cap R$ with $\partial\xi\subset\partial(\Gamma\times I)$.  In the first possibility, the operation removing $B_1\cup B_2\cup B_R$ from $\Gamma\times I$ changes the curve in $P'\cap\partial(\Gamma\times I)$ that is incident to $\partial\xi$ into a curve in $\partial V$ which is still parallel to the core of the solid torus $V$.  Recall that each curve in $P'\cap\partial(\Gamma\times I)$ intersects $\partial R$ at most once and intersects $\Delta_1\cup\Delta_2$ at most once, so in the possibility (2) above, the operation removing $B_1\cup B_2\cup B_R$ from $\Gamma\times I$ changes the two curves in $P'\cap\partial(\Gamma\times I)$ containing $\partial\xi$ into one trivial curve in $\partial V$.

Since $P'-\gamma$ is incompressible and $V\cap\gamma=\emptyset$, in possibility (2) above, each trivial curve in $P'\cap\partial V$ bounds a disk in $V$.  Moreover, the union of this disk (in possibility (2) above) and the component of $P'\cap B_R$ containing the arc $\xi$ above is an annulus.  
By our construction, all the essential curves in $P'\cap\partial V$ are parallel to the core of the solid torus $V$.  Since $P'-\gamma$ is incompressible in $Y-\gamma$ and since $V\cap\gamma=\emptyset$, the essential curves in $P'\cap\partial V$ must bound a collection of annuli in $V$.  The union of these annuli in $V$ and the disks in $P'\cap (B_1\cup B_2\cup B_R)$ is a collection of (punctured) annuli properly embedded in $\Gamma\times I$.  By Claim~\ref{claimD1}, Claim~\ref{claimD2} and our analysis on $P'\cap(\Delta_1\cup\Delta_2\cup R)$ above, these (punctured) annuli are standard (punctured) annuli.  Therefore, $P'\cap(\Gamma\times I)$ consists of standard (punctured) annuli in $\Gamma\times I$.  
By our construction of $P'$, $P$ can be obtained from $P'$ by tubing along $\gamma$.  So part (3) of the lemma also holds.
\end{proof}

\section{Computing Heegaard genus}\label{Sgenus}
The goal of this section is to prove the following lemma which together with Lemma~\ref{LRM} proves the main theorem.

\begin{lemma}\label{LHM}
Suppose $g\ge 3$.  Then there is an infinite set of slopes $\mathcal{S}$ such that if $s\in\mathcal{S}$, the Heegaard genus of $M=X_1\cup Y_s\cup X_2$ is at least $3g+4$.
\end{lemma}
\begin{proof}
We view $X_1$, $X_2$ and $Y_s$ as sub-manifolds of $M$.  Recall that $Y_s$ is obtained from $Y$ by Dehn surgery on $\gamma$.  Let $\gamma'$ be the core curve of the surgery solid torus (as in Lemma~\ref{LRi}) and we consider $M-N(\gamma')$.  Let $T$ be the torus boundary component of $M-N(\gamma')$.  So $M$ is the manifold after Dehn filling on $M-N(\gamma')$ along slope $s$.   Let $S_{min}$ be a minimal-genus Heegaard surface of $M$. 

By a theorem of Hatcher \cite{Ha}, embedded essential surfaces in $M-N(\gamma')$ with boundary in $T$ can realize only finitely many slopes.  Let $\mathcal{S}_F$ be this finite set of slopes.  
 By Lemma~\ref{LRi}, $M-N(\gamma')$ contains no essential annulus with boundary in $T$.  This means that the pair $(M-N(\gamma'), T)$ satisfies the hypotheses of the results in \cite{R, RSed} (this property is called a-cylindrical in \cite{R}), and by \cite{R, RSed} (also see \cite{MR}), there is a finite set of slopes $\mathcal{S}_L$ which depends on $M-N(\gamma')$ and $g(S_{min})$ such that if $s\notin\mathcal{S}_L$ and if the intersection number $\Delta(s,t)>1$ for every $t\in\mathcal{S}_F$, then after isotopy, (1) $S_{min}$ lies in $M-N(\gamma')$ and (2) $S_{min}$ is a Heegaard surface of $M-N(\gamma')$.

Let $\mathcal{S}$ be the set of slopes $s$ that satisfy (1) $s\notin\mathcal{S}_L$, (2) $\Delta(s,t)>1$ for every $t\in\mathcal{S}_F$.  Clearly $\mathcal{S}$ is an infinite set.  We assume our surgery slope $s\in\mathcal{S}$.  Hence $S_{min}$ is a Heegaard surface of both $M$ and $M-N(\gamma')$.

We consider the untelescoping of the Heegaard splitting of $M-N(\gamma')$ along $S_{min}$, i.e., a decomposition $M-N(\gamma')=\mathcal{N}_0\cup_{S_1}\mathcal{N}_1\cup_{S_2}\cdots\cup_{S_{n}}\mathcal{N}_n$, see Theorem~\ref{TST}, where each $S_i$ is incompressible in $M-N(\gamma')$ and each $\mathcal{N}_i$ has a strongly irreducible Heegaard surface $\Sigma_i\subset \mathcal{N}_i$.  If $\mathcal{N}_i$ is disconnected, all but one of its components are product regions, see \cite{S2}. The possible product regions do not affect our proof.   So, for simplicity, we will ignore the thick surfaces in the product regions and assume that each $\mathcal{N}_i$ is connected.  



Next we consider how these surfaces $S_i$ and $\Sigma_i$ intersect $A_1\cup A_2$.  Since each $S_i$ is incompressible, we may assume $S_i\cap (A_1\cup A_2)$ is a collection of essential curves in both $S_i$ and $A_1\cup A_2$ and assume $S_i\cap X_1$, $S_i\cap X_2$ and $S_i\cap Y_s$ contain no annulus parallel to a sub-annulus of $A_1\cup A_2$.   Let $\Sigma_i^{X_1}=\Sigma_i\cap X_1$, $\Sigma_i^{X_2}=\Sigma_i\cap X_2$ and $\Sigma_i^Y=\Sigma_i\cap (Y-\gamma)$. 

By Lemma~\ref{Lann}, for each $\Sigma_i$, we may assume that at most one component of $\Sigma_i^{X_1}\coprod \Sigma_i^{X_2}\coprod \Sigma_i^Y$ is strongly irreducible while all other components are incompressible in the respective manifolds $X_1$, $X_2$ and $Y-\gamma$.

\vspace{8pt}
\noindent
\textbf{Case (a)}.  $\partial X_1$ (or $\partial X_2$) is parallel to a component of $S_i$ for some $i$. 
\vspace{8pt}

Suppose $\partial X_1$ is parallel to a component of some $S_i$ (the case for $\partial X_2$ is the same).  

Let $H_{X}$ be the peripheral surface in the interior of $X_1$ and parallel to $\partial X_1$.  So we may view $H_X$ as a component of some $S_i$. 
We cut $M$ open along $H_X$ and get two components $X'$ and $M_X$, where $X'$ is the component bounded by $H_X$ and isotopic to $X_1$.  Let $M_2$ be the manifold obtained by gluing $Y_s$ to $X_2$ along $A_2$.  So  $M_X$ is the manifold obtained by gluing $M_2$ to a product neighborhood $\partial X_1\times I$ of $\partial X_1$ in $X_1$ along the annulus $A_1$.  Since $H_X$ is a component of some $S_i$, we can amalgamate the $\Sigma_j$'s that lie in $M_X$ (resp.~$X'$) along the $S_j$'s in $M_X$ (resp.~$X'$) to get a Heegaard surface $S_M$ (resp.~$S_X)$ of $M_X$ (resp.~$X'$). 
So our Heegaard surface $S_{min}$ is an amalgamation of the Heegaard surface $S_M$ of $M_X$ and the Heegaard surface $S_X$ of $X'$ along $H_X$.  Since $S_{min}$ is of minimal-genus, both $S_M$ and $S_X$ must be of minimal-genus in $M_X$ and $X'$ respectively.  As in section~\ref{S2}, this means that $g(S_{min})=g(M)=g(M_X)+g(X')-g(H_X)$.  Recall that $g(\partial X_1)=2$ and by Lemma~\ref{LXgenus}, $g(X_1)=3$.  So we have $g(H_X)=2$, $g(X')=3$.  Thus $g(M)=g(M_X)+1$.

To compute the Heegaard genus of our manifolds, it is useful to consider the corresponding manifolds before the Dehn surgery.  Let $M_2'$ be the manifold obtained by gluing $Y$ to $X_2$ along the annulus $A_2$, and let $M_X'$ be the manifold obtained by gluing $M_2'$ to a product neighborhood $\partial X_1\times I$ of $\partial X_1$ in $X_1$ along the annulus $A_1$.  So we may view $M_2$ and $M_X$ above as the manifolds obtained from $M_2'$ and $M_X'$ respectively by Dehn surgery on $\gamma$ with surgery slope $s$.

\begin{claimB}\label{claimB1}
The Heegaard genus of $M_X'$ is at least $3g+3$.
\end{claimB}
\begin{proof}[Proof of Claim~\ref{claimB1}]
We first consider $M_2'=Y\cup_{A_2}X_2$. 
Recall that the dimension of $H_1(Y;\mathbb{Z}_2)$ is $3g+1$.  Moreover, the core of $A_2$ represents a generator of $H_1(Y;\mathbb{Z}_2)$.  By Lemma~\ref{Llive}, the dimension of $H_1(X_2;\mathbb{Z}_2)$ is $2$.  Moreover, as $\pi_1(X_2)$ is generated by $\{x_2, h_2^{-1}x_2h_2, s_2\}$ and since $dim(H_1(X_2;\mathbb{Z}_2))=2$, we see that $x_2$ (i.e., the core of $A_2$) represents a generator of $H_1(X_2;\mathbb{Z}_2)$.  So, using Mayer-Vietoris sequence, it is easy to see that the dimension of $H_1(M_2';\mathbb{Z}_2)$ is $dim(H_1(Y;\mathbb{Z}_2))+dim(H_1(X_2;\mathbb{Z}_2))-1=3g+2$. 

Similarly, $M_X'=(\partial X_1\times I)\cup_{A_1} M_2'$, and the core of the annulus $A_1$ represents a generator for both $H_1(M_2';\mathbb{Z}_2)$ and $H_1(\partial X_1\times I;\mathbb{Z}_2)$.  So using Mayer-Vietoris sequence, it is easy to see that the dimension of $H_1(M_X';\mathbb{Z}_2)$ is $(3g+2)+2g(\partial X_1)-1=3g+5$.  Thus the rank of $\pi_1(M_X')$ is at least $3g+5$. 

Let $S$ be a minimal-genus Heegaard surface of $M_X'$ and let $M_X'=V\cup_S W$ be the Heegaard splitting along $S$.  Note that $\partial M_X'$ has two components, one of which is $H_X$  (parallel to $\partial X_1$) and has genus 2.  

If either $V$ or $W$ is a handlebody, then the splitting $M_X'=V\cup_S W$ gives a presentation of $\pi_1(M_X')$ with $g(S)$ generators.  Since the rank of $\pi_1(M_X'))$ is at least $3g+5$, we have $g(S)\ge 3g+5$ and the claim holds in this case.  

Now we consider the case that neither $V$ nor $W$ is a handlebody.  As $\partial M_X'$ has two components, this means that $\partial_-V$ and $\partial_-W$ are the two components of $\partial M_X'$.  Without loss of generality, we may suppose $\partial_-V=H_X$.  We may view $V$ as the manifold obtained by adding $h$ 1-handles to a product neighborhood of $\partial_-V$ ($\partial_-V=H_X$).  As $g(H_X)=2$, the number of 1-handles $h=g(S)-2$.  Note that $\pi_1(V)$ can be generated by $\pi_1(\partial_-V)$ plus the $h$ generators corresponding to the $h$ 1-handles.  Hence $rank(\pi_1(V))\le rank(\pi_1(H_X))+h=4+g(S)-2=g(S)+2$.  Moreover, since $M_X'$ can be obtained by adding 2-handles to $V$, $rank(\pi_1(M_X'))\le rank(\pi_1(V))\le g(S)+2$.  As $rank(\pi_1(M_X'))\ge 3g+5$, we have $g(S)+2\ge 3g+5$ and hence $g(S)\ge 3g+3$.  So the claim also holds in this case.
\end{proof}

Recall that our Heegaard surface $S_{min}$ of $M$ can be obtained by an amalgamation of $S_M$ and $S_X$ along $H_X$.  Moreover, since $S_{min}$ is a Heegaard surface for both $M$ and $M-N(\gamma')$ and since the untelescoping is with respect to the splitting of $M-N(\gamma')$, $S_M$ must be a Heegaard surface of both $M_X$ and $M_X -N(\gamma')$.

Note that $M_X'$ is the manifold obtained by a trivial Dehn filling on $M_X-N(\gamma')$, so $S_M$ is also a Heegaard surface of $M_X'$.  By Claim~\ref{claimB1}, $g(M_X')\ge 3g+3$. So we have $g(M_X)=g(S_M)\ge g(M_X')\ge 3g+3$.  Now by our earlier conclusion $g(M)=g(M_X)+1$, we have $g(M)\ge 3g+4$ and Lemma~\ref{LHM} holds in Case (a).

Next we suppose that we are not in Case (a).  As $X_1$ and $X_2$ are small, this implies that no $\Sigma_i$ lies totally in $X_1$ or $X_2$. 

\begin{claimB}\label{claimB15}
Either Lemma~\ref{LHM} holds, or there are $k$ and $j$ ($k\ne j$) such that a component of $\Sigma_k^{X_1}$ (resp.~$\Sigma_j^{X_2}$) is strongly irreducible in $X_1$ (resp.~$X_2$) and this component does not lie in a collar neighborhood of $\partial X_1$ (resp.~$\partial X_2$) in $X_1$ (resp.~$X_2$).  Moreover, every other component of $\Sigma_k^{X_1}$, $\Sigma_k^{X_2}$, $\Sigma_k^Y$,  $\Sigma_j^{X_1}$, $\Sigma_j^{X_2}$ and $\Sigma_j^Y$ is incompressible and is not a $\partial$-parallel annulus in the corresponding manifolds $X_1$, $X_2$ or $Y-\gamma$.
\end{claimB}
\begin{proof}[Proof of Claim~\ref{claimB15}]
We have assumed that for each $\Sigma_i$, at most one component of $\Sigma_i^{X_1}\coprod \Sigma_i^{X_2}\coprod \Sigma_i^Y$ is strongly irreducible while all other components are incompressible in the respective manifolds $X_1$, $X_2$ and $Y-\gamma$. 

By Lemma~\ref{LXsmall}, $X_1$ is both small and $A_1$-small.  As $S_i\cap X_1$ is incompressible in $X_1$ for all $i$, this means that $S_i\cap X_1$ and the incompressible components of $\Sigma_i^{X_1}$ consist of $\partial$-parallel surfaces which lie in a collar neighborhood of $\partial X_1$ in $X_1$.  If, for each $\Sigma_i$, the possible strongly irreducible component of $\Sigma_i^{X_1}$ always lies in a collar neighborhood of $\partial X_1$ in $X_1$, then the peripheral surface $H_X$ above can be isotoped disjoint from all the $S_i$'s and $\Sigma_i$'s.  This means that $H_X$ lies in a compression body in the Heegaard splitting of some block $\mathcal{N}_i$ in the untelescoping.  Since any closed incompressible surface in a compression body is parallel to its minus boundary, $H_X$ must be parallel to a component of some $S_i$.  By case (a), Lemma~\ref{LHM} holds.  

Suppose Lemma~\ref{LHM} is false, then by the argument above, there must be some $k$ such that a component of $\Sigma_k^{X_1}$ is strongly irreducible in $X_1$ and this component does not lie in a collar neighborhood of $\partial X_1$ in $X_1$.  As in part (2) of Lemma~\ref{Lann}, every other component of $\Sigma_k^{X_1}$, $\Sigma_k^{X_2}$ and $\Sigma_k^Y$ is incompressible and is not a $\partial$-parallel annulus in the respective manifolds $X_1$, $X_2$ and $Y-\gamma$.  Symmetrically, for $X_2$, there is some $j$ such that $\Sigma_j^{X_2}$ satisfies the claim.  

Recall that, for simplicity, we have assumed that each $\mathcal{N}_i$ and hence each $\Sigma_i$ in the untelescoping is connected. 
Since $\Sigma_k^{X_2}$ is incompressible in $X_2$ but $\Sigma_j^{X_2}$ has a strongly irreducible component, $\Sigma_k$ and $\Sigma_j$ are different surfaces.  Hence $k\ne j$.
\end{proof}

Let $P_X$ and $Q_X$ be the strongly irreducible components of $\Sigma_k^{X_1}$ and $\Sigma_j^{X_2}$ in $X_1$ and $X_2$ respectively as in Claim~\ref{claimB15}.

\begin{claimB}\label{claimB16}
Either Lemma~\ref{LHM} holds, or $|\partial P_X|\ge 4$ and $|\partial Q_X|\ge 4$.
\end{claimB}
\begin{proof}[Proof of Claim~\ref{claimB16}]
We prove $|\partial P_X|\ge 4$ and the case for $Q_X$ in $X_2$ is symmetric.

Recall that $S_{min}$ is a Heegaard surface for both $M$ and $M-N(\gamma')$.  Using our notation before Claim~\ref{claimB1}, $M-N(\gamma')$ is the annulus sum of $X_1$ and $M_2'-N(\gamma)$ along $A_1$, where $M_2'=Y\cup_{A_2}X_2$.  

By Lemma~\ref{Lstr}, $P_X$ is separating in $X_1$ and hence $|\partial P_X|$ is even.  So either the claim holds, or $\partial P_X$ has two circles.  Suppose $\partial P_X$ has two circles, then by Lemma~\ref{L2}, we have $g(M)=g(S_{min})\ge g(X_1)+g(M_2'-N(\gamma))-1$.  

Next we estimate $g(M_2'-N(\gamma))$.  Let $S'$ be a minimal-genus Heegaard surface of $M_2'-N(\gamma)$, then after a trivial Dehn surgery, $S'$ remains a Heegaard surface of $M_2'$.  This means that $g(M_2'-N(\gamma))\ge g(M_2')$.  At the beginning of the proof of Claim~\ref{claimB1}, we have shown that the dimension of $H_1(M_2';\mathbb{Z}_2)$ is $3g+2$.  So the rank of $\pi_1(M_2')$ is at least $3g+2$.  Since $\partial M_2'$ is connected, $r(M_2')\le g(M_2')$.  So the Heegaard genus of $M_2'$ is at least $3g+2$.  As $g(M_2'-N(\gamma))\ge g(M_2')$, we have $g(M_2'-N(\gamma))\ge g(M_2')\ge 3g+2$.  

Since $g(X_1)=3$, we have $g(M)\ge g(X_1)+g(M_2'-N(\gamma))-1\ge 3+(3g+2)-1=3g+4$ and Lemma~\ref{LHM} holds.  
\end{proof}

Let $P=\Sigma_k\cap (Y-\gamma)$, $P_1=P\cap N_1$, $P_2=P\cap N_2$ and $P_K=P\cap N_K$.  
As $|\partial P_X|\ge 4$, $P_1\cap A_1$ contains at least 4 circles.  Similarly, let $Q=\Sigma_j\cap (Y-\gamma)$, $Q_1=Q\cap N_1$, $Q_2=Q\cap N_2$ and $Q_K=Q\cap N_K$.  As $|\partial Q_X|\ge 4$, $Q_2\cap A_2$ contains at least 4 circles

By Claim~\ref{claimB15}, $P$ and $Q$ are incompressible surfaces in $Y-\gamma$ with boundary in $A_1\cup A_2$. 
By Lemma~\ref{Lstandard}, after isotopy, we may assume $P_1$, $P_2$, $P_K$, $Q_1$, $Q_2$, $Q_K$ are standard in the corresponding manifolds $N_1$, $N_2$ and $N_K$, which means a component of these surfaces is either (1) horizontal in the corresponding twisted $I$-bundles $N_1$, $N_2$ and $N_K$, or (2) a tube around $\gamma_1$ or $\gamma_2$.   

\begin{claimB}\label{claimB2}
Either Lemma~\ref{LHM} holds, or $P_1$ (resp.~$Q_2$) contains exactly one horizontal component in the $I$-bundle $N_1$ (resp.~$N_2$).
\end{claimB}
\begin{proof}[Proof of Claim~\ref{claimB2}]
We prove that $P_1$ contains one horizontal component in $N_1$, and the argument for $Q_2$ and $N_2$ is the same.

Since $\partial P_X\ne\emptyset$, $P_1$ contains at least one horizontal component.  Suppose the claim is false and the $P_1$ has more than one horizontal component.  As each connected horizontal surface in $N_1$ has two boundary curves in $\partial_vN_1$, $\partial P_1\cap\partial_vN_1$ contains at least 4 horizontal curves in $\partial_vN_1$.  

By Lemma~\ref{LAS}, $P$ can be obtained by first connecting the horizontal components of $P_1$, $P_2$ and $P_K$ using standard (punctured) annuli in $\Gamma\times I$ and then tubing along $\gamma$.  We use $P_0$ to denote the punctured surface in $Y-\gamma$ before tubing along $\gamma$, i.e., $P_0\cap(\Gamma\times I)$ is a collection of standard (punctured) annuli and $P$ is obtained from $P_0$ by tubing along $\gamma$.  

By the definition of standard (punctured) annuli in $\Gamma\times I$, no annulus of $P_0\cap(\Gamma\times I)$ connects two curves in the same annulus $\partial_vN_1$, $\partial_vN_2$ or $\partial_vN_K$. 
Since $\partial P_1\cap\partial_vN_1$ contains at least 4 horizontal curves in $\partial_vN_1$, there are at least 4 annuli in $P_0\cap(\Gamma\times I)$ connecting the horizontal curves in $\partial P_1\cap\partial_vN_1$ to the horizontal curves in $\partial P_K$ and $\partial P_2\cap \partial_vN_2$.  This means that $P_K\cup P_2$ has totally at least 2 horizontal components in $N_K\cup N_2$.

Next we show that $P_0\cap\gamma$ has at least 2 punctures.   

By Lemma~\ref{Lstandard}, a horizontal component of $P_1$ with a boundary curve in $A_1$ is of type $A'$ or  type $B'$ or type $B''$. A surface of type $A'$ or type $B'$ has two boundary circles in $A_1$ and a surface of type $B''$ has 4 boundary circles in $A_1$. By Claim~\ref{claimB16}, $\partial P_1\cap A_1$ has at least 4 curves.  So we have the following two possible cases for $P_1$: 
\begin{enumerate}[\upshape (i)]
  \item $P_1$ contains a horizontal component of type $B'$ or $B''$,
  \item  $P_1$ contains two horizontal components both of type $A'$.
\end{enumerate}
Let $K_i$ be the subarc of an $I$-fiber in $\partial_vN_i$ connecting $q_i$ to $\partial(\Gamma\times\{0\})$, see Figure~\ref{Fgamma}(b).  As we pointed out in the description of surfaces of type $B'$ and $B''$ before Lemma~\ref{Lstandard}, for any component $P_1'$ of type $B'$ or $B''$ in $N_1$, both curves of $\partial P_1'\cap \partial_vN_1$ intersect $K_1$.  
So in case (i), $\partial P_1$ intersects $K_1$ at least twice.  Similarly, each horizontal component of $P_1$ of type $A'$ must intersect $K_1$ once (see Figure~\ref{Ftype}(a)), and hence in case (ii), $\partial P_1$ also intersects $K_1$ at least twice.  Recall that $\gamma\cap(\Gamma\times I)$ consists of 3 arcs: $\gamma_p$, $\delta_1$ and $\delta_2$.  By our construction of $\delta_i$, see Figure~\ref{Fgamma}(b), the conclusion on $\partial P_1\cap K_1$ implies that, in both case (i) and case (ii), the annuli in $P_0\cap(\Gamma\times I)$ connecting $\partial P_1\cap\partial_vN_1$ to $\partial_vN_2\cup\partial_vN_K$ must intersect the arc $\delta_1$ at least twice (see Figure~\ref{Fgamma}(b) and Figure~\ref{Fstandard}).   Hence $|P_0\cap\gamma|\ge |P_0\cap\delta_1|\ge 2$.

\begin{figure}
  \centering
\psfrag{X}{$\Gamma\times I$}
\psfrag{g}{$\gamma_p$}
\psfrag{d}{$\delta_1$}
\psfrag{e}{$\delta_2$}
\psfrag{p}{punctures}
\psfrag{0}{$P_0$}
  \includegraphics[width=3in]{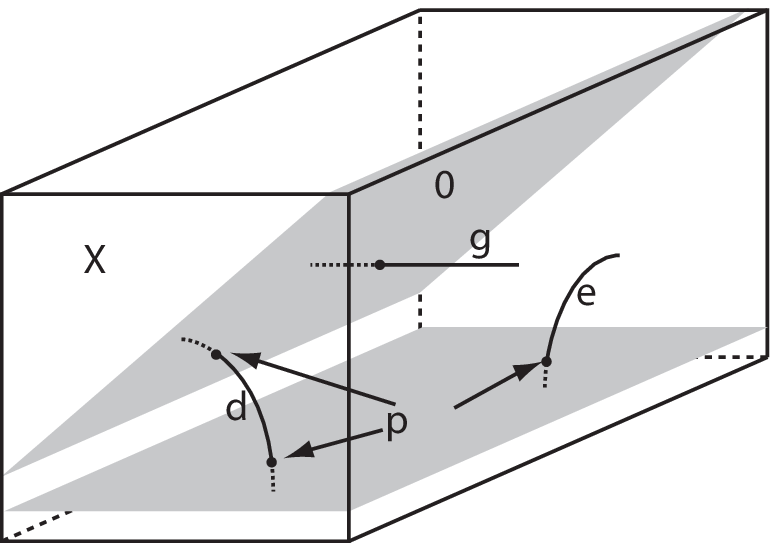}
  \caption{}
  \label{Fstandard}
\end{figure}

Now we calculate the Euler characteristic of $\Sigma_k$.  By Lemma~\ref{Lstr} and since $g(X_1)=3$, we have $\chi(P_X)\le 2-2g(X_1)=-4$.  Each orientable horizontal surface in the twisted $I$-bundles $N_1$, $N_2$ and $N_K$ has Euler characteristic $2(1-g)=2-2g$.  As $P_1$ has at least two horizontal components and $P_K\cup P_2$ has at least 2 horizontal components, $P_1$, $P_2$ and $P_K$ have totally at least 4 horizontal components.  Therefore, we have $\chi(\Sigma_k)\le \chi(P_X)+\chi(P_1)+\chi(P_2)+\chi(P_K)-|P_0\cap\gamma|\le -4+4(2-2g)-2=2-8g$.  This means that $g(\Sigma_k)\ge 4g$.  Since $g\ge 3$, $g(\Sigma_k)\ge 4g\ge 3g+3$.  

In the argument above, we have $\Sigma_j\ne\Sigma_k$.  So there are at least two blocks $\mathcal{N}_j$ and $\mathcal{N}_k$ in the untelescoping for the Heegaard surface $S_{min}$ of $M-N(\gamma')$.  The untelescoping construction can be viewed as a rearrangement of the handles in the Heegaard splitting.  As the Heegaard splitting of $\mathcal{N}_j$ in the untelescoping is assumed to be non-trivial, there is at least one 1-handle in $\mathcal{N}_j$, and this implies that $S_{min}\ne\Sigma_k$ and $g(S_{min})\ge g(\Sigma_k)+1$.  Since $g(\Sigma_k)\ge 3g+3$, we have $g(S_{min}) \ge 3g+4$ and Lemma~\ref{LHM} holds.
\end{proof}

By Claim~\ref{claimB2}, we may assume $P_1$ (resp. $Q_2$) contains exactly one horizontal component, which we denote by $P_1'$ (resp. $Q_2'$).   Since $|\partial P_1'\cap A_1|\ge |\partial P_X|\ge 4$ and $|\partial Q_2'\cap A_2|\ge |\partial Q_X|\ge 4$,  Lemma~\ref{Lstandard} implies that $P_1'$ and $Q_2'$ must be of type $B''$ and $|\partial P_1'\cap A_1|=|\partial Q_2'\cap A_2|=4$.

Next we consider $Q_1$ and $P_2$.  Let $\Pi_1$ and $\Pi_2$ be horizontal components of $Q_1$ and $P_2$ in $N_1$ and $N_2$ respectively.  
By Lemma~\ref{Lstandard},  $\partial P_1'\cap\partial_vN_1$ and $\partial \Pi_1\cap\partial_vN_1$ are pairs of essential curves in $\partial_vN_1$. 
Let $V_P$ and $V_{\Pi_1}$ be the sub-annuli of $\partial_vN_1$ bounded by $\partial P_1'\cap\partial_vN_1$ and $\partial \Pi_1\cap\partial_vN_1$ respectively.  Similarly, $\partial Q_2'\cap\partial_vN_2$ and $\partial \Pi_2\cap\partial_vN_2$ are pairs of essential curves in $\partial_vN_2$. 
Let $V_Q$ and $V_{\Pi_2}$ be the sub-annuli of $\partial_vN_2$ bounded by $\partial Q_2'\cap\partial_vN_2$ and $\partial \Pi_2\cap\partial_vN_2$ respectively.

\begin{claimB}\label{claimB3}
Let  $P_1'$, $Q_2'$, $\Pi_1$, $\Pi_2$, $V_P$, $V_Q$, $V_{\Pi_1}$ and $V_{\Pi_2}$ be as above.  Then 
\begin{enumerate}
  \item $V_P\subset V_{\Pi_1}$ and symmetrically $V_Q\subset V_{\Pi_2}$,
  \item $\partial \Pi_1\cap A_1\ne\emptyset$ and symmetrically $\partial \Pi_2\cap A_2\ne\emptyset$.
\end{enumerate}  
\end{claimB}
\begin{proof}[Proof of Claim~\ref{claimB3}]
We prove $V_P\subset V_{\Pi_1}$ and $\partial \Pi_1\cap A_1\ne\emptyset$ for $P_1$ and $\Pi_1$.  The argument for $Q_2$ and $\Pi_2$ is the same and symmetric.

In the argument above, we have concluded that $|\partial P_1'\cap A_1|=4$.  
Let $c_1$, $c_2$, $c_3$ and $c_4$ be the 4 circles of $\partial P_1'\cap A_1$ in consecutive order in $A_1$ (i.e. the sub-annulus of $A_1$ bounded by each $c_i\cup c_{i+1}$ contains no other circle $c_l$).  
Recall that in our construction, $c_1,\dots c_4$ are also the boundary circles of $P_X$.  By Lemma~\ref{Lstr}, if we maximally compress $P_X$ in $X_1$ on one side, we get a pair of $\partial$-parallel annuli bounded by $c_1\cup c_2$ and $c_3\cup c_4$ respectively, and if we maximally compress $P_X$ in $X_1$ on the other side, we get a $\partial$-parallel annulus bounded by $c_2\cup c_3$ (plus a surface parallel to $\partial X_1-\Int(A_1)$ and bounded by $c_1\cup c_4$).  Now we consider $\Sigma_j\supset Q_1\supset\Pi_1$.  Since any compressions on $P_X$ can be disjoint from other surfaces in the untelescoping, the conclusion above on $P_X$ means that if a curve $\alpha_H$ of $\Sigma_j\cap A_1$ lies in any of the sub-annuli bounded by $c_i\cup c_{i+1}$, then (since $\Sigma_j\cap X_1$ is incompressible in $X_1$) the component of $\Sigma_j\cap X_1$ containing $\alpha_H$ must be a $\partial$-parallel annulus in $X_1$, which contradicts our earlier assumption on $\Sigma_j$.  Thus $\Sigma_j\cap A_1$ and in particular $\partial\Pi_1\cap A_1$ (if not empty) are disjoint from the sub-annulus of $A_1$ bounded by $c_1\cup c_4$.

By our description of standard surfaces before Lemma~\ref{Lstandard}, we can perform tubing first on $P_1'$ and then on $\Pi_1$ along $A_1$ and get horizontal surfaces $\widehat{P}_1'$ and $\widehat{\Pi}_1$ in $N_1$ disjoint from $A_1$ (see Figure~\ref{Fcap}), such that 
(1) $P_1'$ can be obtained by two annulus-compressions on $\widehat{P}_1'$, and (2) either $\Pi_1=\widehat{\Pi}_1$ (i.e., $\Pi_1\cap A_1=\emptyset$) or $\Pi_1$ is obtained by one or two annulus-compressions on $\widehat{\Pi}_1$.  
Since $\partial\Pi_1\cap A_1$ is disjoint from the sub-annulus of $A_1$ bounded by $c_1\cup c_4$, by our tubing operation on $P_1'$, we may assume $\widehat{P}_1'\cap\Pi_1=\emptyset$.  Since $\widehat{P}_1'\cap A_1=\emptyset$ and $\widehat{P}_1'\cap\Pi_1=\emptyset$, the surface $\widehat{\Pi}_1$ after tubing $\Pi_1$ along $A_1$ remains disjoint from 
$\widehat{P}_1'$, i.e.  $\widehat{P}_1'\cap\widehat{\Pi}_1=\emptyset$.

$\widehat{P}_1'$ and $\widehat{\Pi}_1$ can be viewed as orientable sections of the twisted $I$-bundle $N_1$, so each $\widehat{P}_1'$ and $\widehat{\Pi}_1$ bounds a sub-twisted $I$-bundle of $N_1$, which we denoted by $W_P$ and $W_\Pi$ respectively.  We may view $\partial\widehat{P}_1'=\partial P_1'\cap\partial_vN_1$, $\partial\widehat{\Pi}_1=\partial \Pi_1\cap\partial_vN_1$, $V_P=\partial_vW_P$ and $V_{\Pi_1}=\partial_vW_\Pi$.   Since $\widehat{P}_1'\cap\widehat{\Pi}_1=\emptyset$, any pair of such sub-twisted $I$-bundles of $N_1$ are nested and in particular, $V_P$ and $V_{\Pi_1}$ are nested.

The surface $P_1'$ is obtained by two annulus-compressions on $\widehat{P}_1'$.  As shown in Figure~\ref{Fcap}, we may view the two annulus-compressions on $\widehat{P}_1'$ to be the operation that pushes a neighborhood of a vertical annulus in $W_P$ into $A_1$.  This means that, as illustrated in Figure~\ref{Fintersect}(a), the region $W_P'$ bounded by $P_1'\cup V_P$ (and the two sub-annuli of $A_1$ bounded by $c_1\cup c_2$ and $c_3\cup c_4$) is also an $I$-bundle (though the fiber structure of $W_P'$ is different from that of $W_P$ near $A_1$).  In particular, the two sub-annuli of $A_1$ bounded by $c_1\cup c_2$ and $c_3\cup c_4$ are two components of the vertical boundary $\partial_vW_P'$.

\begin{figure}
  \centering
\psfrag{W}{$W_P'$}
\psfrag{L}{$W_\Lambda$}
\psfrag{P}{$P_1$}
\psfrag{Q}{$Q_2$}
\psfrag{K}{$P_K$}
\psfrag{S}{$Q_K$}
\psfrag{v}{$\partial_vN_K$}
\psfrag{a}{(a)}
\psfrag{b}{(b)}
\psfrag{c}{(c)}
  \includegraphics[width=5in]{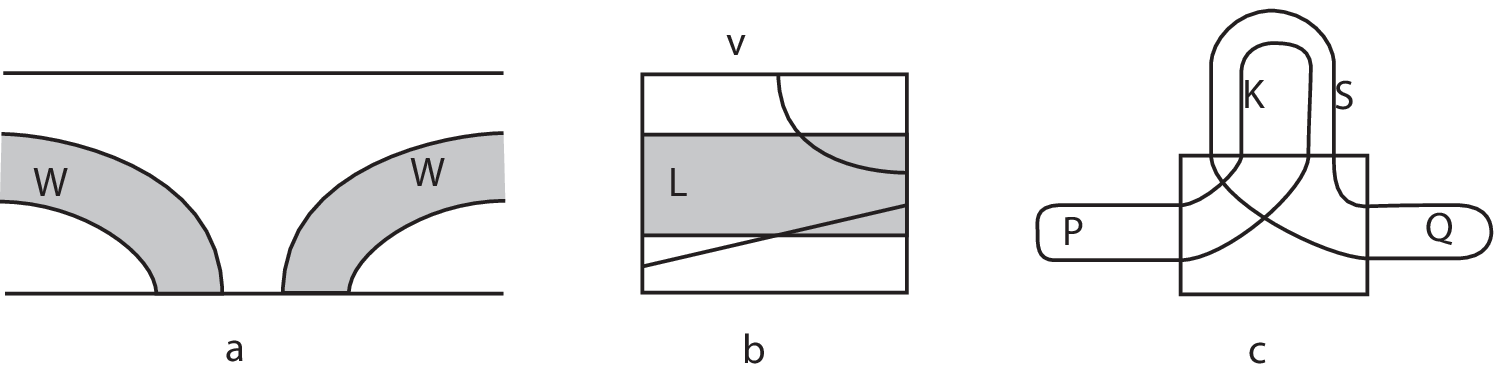}
  \caption{}
  \label{Fintersect}
\end{figure}

Since $\Pi_1$ is incompressible in $N_1$ and is not an annulus, if $\Pi_1$ lies in $W_P'$, it must be incompressible in $W_P'$ and hence can be isotoped to be horizontal in the $I$-bundle $W_P'$.  This implies that if $\Pi_1$ lies in $W_P'$, then it must have boundary curves in each component of $\partial_vW_P'$ and hence 
have boundary curves in the two sub-annuli of $A_1$ bounded by $c_1\cup c_2$ and $c_3\cup c_4$.  This contradicts our conclusion at the beginning of the proof that $\partial\Pi_1\cap A_1$ is disjoint from the sub-annulus of $A_1$ bounded by $c_1\cup c_4$.  Thus $\Pi_1$ cannot lie inside $W_P'$.  Since $P_1'\cap\Pi_1=\emptyset$, $\Pi_1$ must lie outside $W_P'$.  Furthermore, since $\Pi_1$ has no boundary curve in the annulus bounded by $c_1\cup c_4$, after tubing $P_1'$ to get $\widehat{P}_1'$, we see that $\Pi_1$ also lies outside $W_P$.  So after we perform tubing on $\Pi_1$ to get $\widehat{\Pi}_1$, $\widehat{\Pi}_1$ also lies outside $W_P$.  As $W_P$ and $W_\Pi$ are nested, this means that $W_P\subset W_\Pi$. As $V_P=\partial_vW_P$ and $V_{\Pi_1}=\partial_vW_\Pi$, we have $V_P\subset V_{\Pi_1}$ and part (1) of the claim holds.

Suppose part (2) of the claim is false.  By the conclusion above,  this happens only if $\Pi_1=\widehat{\Pi}_1$.  So we may view $\partial_h W_\Pi=\widehat{\Pi}_1=\Pi_1$.  Since $P_1'$ is disjoint from $\Pi_1$ and since $\partial P_1'\cap A_1\ne\emptyset$, if $\partial_h W_\Pi=\Pi_1$, $P_1'$ must lie outside the $I$-bundle $W_\Pi$.  Hence after tubing $P_1'$, we get our surface $\widehat{P}_1'$ outside $W_\Pi$.  As $W_P$ and $W_\Pi$ are nested, this means that $W_\Pi\subset W_P$, contradicting out conclusion above $W_P\subset W_\Pi$.  So $\Pi_1\ne\widehat{\Pi}_1$ and part (2) of the claim also holds.
\end{proof}

By Lemma~\ref{LAS}, $\Sigma_k\cap (Y-\gamma)$ (resp. $\Sigma_j\cap(Y-\gamma)$) can be obtained by first connecting the horizontal components of $P_1$, $P_2$ and $P_K$ (resp. $Q_1$, $Q_2$ and $Q_K$)  by standard (punctured) annuli in $\Gamma\times I$ and then tubing the resulting (punctured) surface along $\gamma$.  Let $\Lambda_P$ and $\Lambda_Q$ be the collections of standard (punctured) annuli in $\Gamma\times I$ connecting the horizontal components of $P_1\cup P_2\cup P_K$ and $Q_1\cup Q_2\cup Q_K$ respectively.  As $\Sigma_k\ne\Sigma_j$, $\Lambda_P$ and $\Lambda_Q$ are 2 disjoint sets of punctured annuli.

\vspace{8pt}
\noindent
\textbf{Case (b)}.  Both $P_2$ and $P_K$ contain a horizontal component in the respective twisted $I$-bundles $N_2$ and $N_K$, or symmetrically both $Q_1$ and $Q_K$ contain a horizontal component in the respective twisted $I$-bundles $N_1$ and $N_K$.
\vspace{8pt}

Suppose both $P_2$ and $P_K$ contain a horizontal component (the case that  both $Q_1$ and $Q_K$ contain a horizontal component is symmetric).  We will show that Lemma~\ref{LHM} holds in Case (b).

By Claim~\ref{claimB2}, we may assume $P_1$ has exactly one horizontal component.  So it follows from Claim~\ref{claimB16} and Lemma~\ref{Lstandard} that the horizontal component of $P_1$ must be of type $B''$.  So the pair of horizontal curves in $\partial P_1\cap\partial_vN_1$ is as shown in Figure~\ref{Ftype}(b).  Recall that the two endpoints of $\gamma_1$ are $p_1$ and $q_1$, and $q_1$ is also the endpoint of the arc $\delta_1$ in $\gamma\cap(\Gamma\times I)$.  Let $K_i$ ($i=1,2$) be the subarc of an $I$-fiber of $\partial_vN_i$ connecting $q_i$ to $\partial(\Gamma\times\{0\})$, see Figure~\ref{Fgamma}(b).  By our description of type $B''$ surfaces before Lemma~\ref{Lstandard}, both horizontal curves in $\partial P_1\cap\partial_vN_1$ must intersect $K_1$.  As shown in Figure~\ref{Fstandard}, this means that the two annuli in $\Lambda_P$ containing the two horizontal curves in $\partial P_1\cap\partial_vN_1$ must both intersect the arc $\delta_1$.  So $|\Lambda_P\cap\delta_1|\ge 2$.

Similarly, by Claim~\ref{claimB2}, we may assume $Q_2$ has exactly one horizontal component in $N_2$ which is of type $B''$ in $N_2$.  The argument above implies that both horizontal curves in $\partial Q_2\cap\partial_vN_2$ must intersect $K_2$.  Now we consider a horizontal component $\Pi_2$ of $P_2$ in $N_2$.  Let $V_{\Pi_2}$ be the sub-annulus of $\partial_vN_2$ bounded by the pair of curves $\partial\Pi_2\cap\partial_vN_2$.  By part(1) of Claim~\ref{claimB3}, $V_{\Pi_2}$ contains the two horizontal curves of $\partial Q_2\cap\partial_vN_2$.  Since both horizontal curves in $\partial Q_2\cap\partial_vN_2$ intersect $K_2$, this implies that $\partial\Pi_2\cap K_2\ne\emptyset$.  Similar to the argument above, this means that the annuli in $\Lambda_P$ connecting $\Pi_2$ must intersect the arc $\delta_2$ at least once.  

As $|\Lambda_P\cap\delta_1|\ge 2$ and $|\Lambda_P\cap\delta_2|\ge 1$, we have $|\Lambda_P\cap\gamma|\ge 3$.  Since $P=\Sigma_k\cap (Y-\gamma)$ can be obtained by adding tubes to a punctured surface along $\gamma$, $|\Lambda_P\cap\gamma|$ is an even number.  So we have $|\Lambda_P\cap\gamma|\ge 4$.

Since $P_2$ contains a horizontal component $\Pi_2$, by part (2) of Claim~\ref{claimB3}, $\Pi_2\cap A_2\ne\emptyset$.  As $\Pi_2\subset P_2\subset\Sigma_k$, this means $\Sigma_k\cap X_2\ne\emptyset$.  By our earlier conclusions, $\Sigma_k\cap X_2$ consists of incompressible surfaces and no component of $\Sigma_k\cap X_2$ is a $\partial$-parallel annulus.  Since $X_2$ is $A_2$-small, each component of $\Sigma_k\cap X_2$ must be parallel to $\partial X_2-\Int(A_2)$.   Since $g(\partial X_2)=2$, we have $\chi(\Sigma_k\cap X_2)\le -2$.

Now we estimate $\chi(\Sigma_k)$.  We have $\chi(\Sigma_k)\le\chi(P_X)+\chi(P_1)+\chi(P_K)+\chi(P_2)+\chi(\Sigma_k\cap X_2)-|\Lambda_P\cap\gamma|$.  By Lemma~\ref{LXgenus}, $g(X_1)=3$.  So by Lemma~\ref{Lstr}, $\chi(P_X)\le 2-2g(X_1)=-4$.  
Each horizontal surface in $N_1$, $N_2$ and $N_k$ has Euler characteristic $2(1-g)$.  By our hypothesis of case (b), each $P_1$, $P_2$ and $P_K$ has a horizontal component in $N_1$, $N_2$ and $N_k$ respectively, so we have $\chi(P_1)+\chi(P_K)+\chi(P_2)\le 3(2-2g)$.  Moreover, we have concluded above that $|\Lambda_P\cap\gamma|\ge 4$ and $\chi(\Sigma_k\cap X_2)\le -2$.  So we have $\chi(\Sigma_k)\le\chi(P_X)+\chi(P_1)+\chi(P_K)+\chi(P_2)+\chi(\Sigma_k\cap X_2)-|\Lambda_P\cap\gamma|\le -4+ 3(2-2g)-2-4=-4-6g$.  Thus $g(\Sigma_k)\ge 3g+3$.  

Now the argument is the same as the last part of the proof of Claim~\ref{claimB2}.
Since $\Sigma_j\ne\Sigma_k$, there are at least two blocks $\mathcal{N}_j$ and $\mathcal{N}_k$ in the untelescoping for the Heegaard surface $S_{min}$ of $M-N(\gamma')$.  The untelescoping construction can be viewed as a rearrangement of the handles in the Heegaard splitting along $S_{min}$.  As the Heegaard splitting of $\mathcal{N}_j$ in the untelescoping is assumed to be non-trivial, there is at least one 1-handle in $\mathcal{N}_j$, and this implies that $S_{min}\ne\Sigma_k$ and $g(S_{min})\ge g(\Sigma_k)+1$.  Since $g(\Sigma_k)\ge 3g+3$, we have $g(S_{min}) \ge 3g+4$ and Lemma~\ref{LHM} holds.

\vspace{8pt}
\noindent
\textbf{Case (c)}.   $P_K=\emptyset$ or $Q_K=\emptyset$.
\vspace{8pt}

We will show that Case (c) cannot happen. 
Suppose $P_K=\emptyset$ (the case  $Q_K=\emptyset$ is the same).  Then $\Lambda_P$ consists of annuli connecting $\partial P_1\cap\partial_vN_1$ to $\partial P_2\cap\partial_vN_2$.  Since $P_1$ has only one horizontal component $P_1'$, $\partial P_1\cap\partial_vN_1$ has exactly two horizontal curves in $\partial_vN_1$.  Recall that the two boundary curves of a standard (punctured) annulus lie in two different annuli of $\partial_vN_1$, $\partial_vN_2$ and $\partial_vN_K$.  As $P_K=\emptyset$, this means that $\partial P_2\cap\partial_vN_2$ must also have two horizontal curves and $\Lambda_P$ consists of two standard (punctured) annuli connecting $\partial_vN_1$ to $\partial_vN_2$.  We use $W_\Lambda$ to denote the sub-manifold of $\Gamma\times I$ between the pair of annuli $\Lambda_P$.  So $W_\Lambda$ is a product region of the form $annulus\times I$.  Moreover, the sub-annulus $V_P$ of $\partial_vN_1$ bounded by $\partial P_1'\cap\partial_vN_1$ is an annulus in $\partial W_\Lambda$.  By part (1) of Claim~\ref{claimB3}, any horizontal curves of $\partial Q_1\cap \partial_vN_1$ must lie outside the annulus $V_P$ and hence outside $W_\Lambda$.

Since $\partial P_2\cap\partial_vN_2$ contains two horizontal curves, $P_2$ has only one horizontal component in $N_2$.  Let $\Pi_2$ be the horizontal component of $P_2$  and let $V_{\Pi_2}$ be the sub-annulus of $\partial_vN_2$ bounded by $\partial\Pi_2\cap\partial_vN_2$.  So $V_{\Pi_2}\subset\partial W_\Lambda$.  
By part(1) of Claim~\ref{claimB3}, $V_Q\subset V_{\Pi_2}$ and in particular, $V_{\Pi_2}$ (and hence $\partial W_\Lambda$) contains the pair of horizontal curves in $\partial Q_2\cap\partial_vN_2$.  However, since both $\partial_vN_K$ and the horizontal curves of $\partial Q_1\cap \partial_vN_1$ lie outside $W_\Lambda$, as illustrated in the 1-dimensional schematic picture Figure~\ref{Fintersect}(b), any standard (punctured) annulus in $\Lambda_Q$ connecting 
$\partial Q_2\cap\partial_vN_2$ to $\partial Q_K$ or to $\partial Q_1\cap \partial_vN_1$ must intersect the pair of annuli $\Lambda_P$, which contradicts that $\Sigma_k\cap\Sigma_j=\emptyset$.  Thus Case (c) cannot happen.

By Case (b) and Case (c) above, the remaining case to consider is:

\vspace{8pt}
\noindent
\textbf{Case (d)}.  $P_K\ne\emptyset$, $Q_K\ne\emptyset$, $P_2$ contains no horizontal component in $N_2$, and $Q_1$ contains no horizontal component in $N_1$.
\vspace{8pt}

We will show that Case (d) cannot happen either.  Suppose on the contrary that Case (d) occurs.

As above, we consider the two sets of standard (punctured) annuli $\Lambda_P$ and $\Lambda_Q$ in $\Gamma\times I$ connecting horizontal curves in $\partial P_1\cap\partial_vN_1$ and $\partial Q_2\cap \partial_vN_2$ to $P_K$ and $Q_K$ respectively.

As $\Sigma_j\ne\Sigma_k$, we have $P_K\cap Q_K=\emptyset$.  By Claim~\ref{claimB2}, there are only two horizontal curves in $\partial P_1\cap\partial_vN_1$ and only two horizontal curves in $\partial Q_2\cap \partial_vN_2$.  Similar to Case (c), since $P_2$ (resp. $Q_1$) has no horizontal component, $\Lambda_P$ (resp. $\Lambda_Q$) consists of 2 annuli.  Hence $|\partial P_K|=2$ and $|\partial Q_K|=2$, and this means that $P_K$ and $Q_K$ are connected horizontal surfaces in $N_K$.  

We may view $P_K$ and $Q_K$ as orientable sections of the twisted $I$-bundle $N_K$.  Each orientable section of $N_K$ bounds a sub-twisted $I$-bundle of $N_K$, and these sub-twisted $I$-bundles are nested.  So the two sub-annuli of $\partial_vN_K$ bounded by $\partial P_K$ and $\partial Q_K$ are nested.  However, as illustrated in the 1-dimension schematic picture Figure~\ref{Fintersect}(c), this implies that $\Lambda_P$ must intersect $\Lambda_Q$ in $\Gamma\times I$, which is impossible as $\Sigma_k\cap\Sigma_j=\emptyset$.  Thus this case cannot happen either.

Therefore, in all possible cases, we have $g(M)\ge 3g+4$ and Lemma~\ref{LHM} holds.
\end{proof}

\begin{proof}[Proof of Theorem~\ref{Tmain}]
By Lemma~\ref{LRi}, $M-N(\gamma')$ is irreducible and atoroidal.  So if $M$ is reducible or toroidal, then an essential 2-sphere or torus in $M$ must non-trivially intersect the curve $\gamma'$, which means that our Dehn surgery slope $s$ is a boundary slope of an essential punctured 2-sphere or torus.  
 Recall that our slope $s$ is assumed not to be a boundary slope of an essential surface with boundary in the boundary torus $T$ of $M-N(\gamma')$.  So this cannot happen and $M$ is irreducible and atoroidal.

By  Lemma~\ref{LRM}, the rank $r(M)\le 3g+3$, and by Lemma~\ref{LHM}, $g(M)\ge 3g+4$.  Hence Theorem~\ref{Tmain} holds.  
\end{proof}

Although we only need $g(M)\ge 3g+4$ to prove the main theorem, it is not hard to find a Heegaard surface of $M$ with genus equal to $3g+4$.  For completeness, we briefly describe a weakly reducible Heegaard splitting of $M$ with genus $3g+4$.  This Heegaard surface corresponds to Case (b) in the proof of Lemma~\ref{LHM}.  By Lemma~\ref{LHM}, this means that the Heegaard genus of $M$ is in fact equal to $3g+4$.  

First, we construct a properly embedded surface in $X_1$ with 4 boundary circles in $A_1$.  Recall that $X_1$ is the exterior of a graph $G=K\cup \beta$, where $K$ is a 2-bridge knot, see section~\ref{SX}.  By our construction, the arc $\beta$ lies in a bridge sphere.  We take a bridge 2-sphere $S_b$ of $K$ disjoint from $\beta$ which corresponds to a 4-hole sphere $S_1'$ properly embedded in $X_1$ with all 4 boundary circles in $A_1$.  We can add a tube to $S_1'$ along an arc (see the arc $s$ in Figure~\ref{Fknot}(c)) that goes around the arc $\beta$ and get a 4-hole torus $S_1$.  Note that, as $\beta$ is parallel to an arc in the bridge sphere, if one performs tubing on $S_1$ (see Definition~\ref{Dtubing}) along the two arcs of $K$ in the 3-ball that is bounded by the bridge sphere $S_b$ and that does not contain $\beta$, then the resulting closed surface is a genus 3 Heegaard surface of $X_1$.  In particular, $S_1$ is strongly irreducible in $X_1$ and has 4 boundary circles in $A_1$. 

Let $S_2$ be the surface obtained from the genus 3 Heegaard surface of $X_2$ (described in Lemma~\ref{LXgenus}) by one annulus-compression along $A_2$.  In fact, $S_2$ is a strongly irreducible surface in $X_2$ with 2 boundary circles in $A_2$.

Let $P_1$ be a horizontal surface in $N_1$ of type $B''$, let $P_2$ be a horizontal surface in $N_2$ of type $A'$, and let $P_K$ be a horizontal surface in $N_K$.  We can connect $P_1$ to $S_1$ in $A_1$ and connect $P_2$ to $S_2$ in $A_2$.  Then we use 3 standard (punctured) annuli in $\Gamma\times I$ to connect $P_K$ to $P_1$, $P_K$ to $P_2$ and $P_1$ to $P_2$.  Let $S$ be the resulting (punctured) surface.  As in Case (b) in the proof of Lemma~\ref{LHM} and illustrated in Figure~\ref{Fstandard}, it is easy to see that $S$ has exactly 4 intersection points with $\gamma$.  Then we add two (nested) tubes to $S$, both tubes going through the 1-handle.  Let $\widehat{S}$ be the resulting closed surface.  Similar to the calculation in Case (b), we have $\chi(\widehat{S})=\chi(S_1)+\chi(P_1)+\chi(P_K)+\chi(P_2)+\chi(S_2)-|S\cap\gamma|=-4+3(2-2g)-4-4=-6-6g$.  So $g(\widehat{S})=3g+4$.  Although not obvious, it is not hard to see that $\widehat{S}$ is a Heegaard surface of $M$.

We conclude this section by explaining why Theorem~\ref{TJSJ} follows from the same proof.  

\begin{proof}[Proof of Theorem~\ref{TJSJ}]
We first modify the construction of $M$ a little. 
In the construction of $M$, instead of using $X_1$ and $X_2$, we glue a pair of 2-bridge knot exteriors $X_1'$ and $X_2'$ to $Y_s$ along $A_1$ and $A_2$ respectively, where the core curve of $A_i$ in $\partial X_i'$ is a meridional loop of the 2-bridge knot.  We denote the resulting manifold by $M_T$.

By \cite{HT}, a 2-bridge knot exterior is both small and meridionally small.  So $X_i'$ has the same topological properties that we need as $X_i$.  Moreover, $\pi_1(X_i')$ is generated by a pair of conjugate elements $x_i$ and $h_i^{-1}x_ih_i$, where $x_i$ is represented by a meridional loop of the knot and $h_i\in\pi_1(X_i')$.  Now we repeat the argument in Lemma~\ref{LRM} and see that the rank of $\pi_1(M_T)$ is at most $3g+1$ (the only difference here is that $\pi_1(X_1')$ and $\pi_1(X_2')$ do not have the generators $s_1$ and $s_2$ in $\pi_1(X_1)$ and $\pi_1(X_2)$ respectively).

Our arguments for $M$ above can all be applied to $M_T$, except that since $M_T$ is toroidal, part (2) of Lemma~\ref{LRi} is not true for $M_T$.  Nonetheless, parts (1) and (3) of Lemma~\ref{LRi} still hold for $M_T$ and we can still apply the results in \cite{R, RSed} on $M_T$.

By applying the argument in Lemma~\ref{LHM} on $M_T$, we see that the Heegaard genus of $M_T$ is at least $3g+2$ (the only difference here is that $g(\partial X_i)=2$ and $g(X_i)=3$ but $\partial X_i'$ is a torus and $g(X_i')=2$).  Note that the reason that the same calculation also works on $X_i'$ is that both $X_i$ and $X_i'$ have tunnel number one.  In fact, if we compute the tunnel number instead of Heegaard genus, then both $M$ and $M_T$ have tunnel number 3.  Thus $r(M_T)<g(M_T)$.

The pair of 2-bridge knot exteriors are JSJ pieces of $M_T$.  After capping off $\partial M_T$ by a handlebody and using a complicated gluing map, $X_1'$ and $X_2'$ remain JSJ pieces of the resulting closed 3-manifold.  Now Theorem~\ref{TJSJ} follows from Corollary~\ref{CL5}.
\end{proof}

\section{Some open questions}\label{Sopen}

In this section, we discuss some interesting open questions related to rank and genus.  Some of these questions were asked earlier by other mathematicians but the author could not find in the literature who are the first to raise the questions.

\begin{question}\label{Q0}
Is there a non-Haken 3-manifold $M$ with $r(M)<g(M)$?
\end{question}

The manifolds constructed in this paper are Haken manifolds. It is not clear whether one can modify the methods in this paper to produce a non-Haken example.  Some obvious changes on the construction may not work, since one would need each piece in the annulus sum to be a handlebody to get a non-Haken manifold.  On the other hand, non-Haken manifolds are more rigid than Haken manifolds, e.g., see \cite{L1, L2}.  So there may be a chance that rank equals genus for non-Haken manifolds.

Another related question is:

\begin{question}\label{Q1}
Is there a knot $k$ in $S^3$ such that $r(S^3-N(k))<g(S^3-N(k))$?  How about a prime knot $k$?
\end{question}

It is conceivable that one can use the methods in this paper to produce a composite knot $k$ in $S^3$ whose exterior has rank smaller than Heegaard genus.  But it is much harder to find a prime-knot example.

\begin{question}\label{Q2}
Among all hyperbolic 3-manifolds $M$ with $r(M)<g(M)$, what is the minimal value for $r(M)$? 
\end{question}

By Proposition~\ref{Prank}, the rank of $\pi_1(M)$ in our construction is $3g+3$ with $g\ge 3$.  Note that the genus of the non-orientable surface $F_K$ in our construction does not need to be at least 3.  So we can choose $g(F_1)=g(F_2)=3$ and choose $F_K$ to be a M\"{o}bius band.  This gives an example $M$ with $r(M)=(3+3+1)+3=10$ and $g(M)=11$.  It is conceivable that one can modify this construction to get a manifold with smaller rank, but it is not clear how small it can be.  

Note that by Proposition~\ref{Prank}, the rank $r(M)$ equals the dimension of $H_1(M;\mathbb{Z}_2)$ for some surgery slopes.  
Namazi informed the author that, for such manifolds, one can apply the proof in \cite{NS} to show that if one caps off $\partial M$ using a handlebody and via a high power of a generic pseudo-Anosov map, then the resulting closed 3-manifold $\widehat{M}$ has the same rank as $M$.  So for the closed hyperbolic 3-manifold $\widehat{M}$ in our construction, $r(\widehat{M})$ can be as low as 10.

This construction of gluing a handlebody immediately brings another interesting question:

\begin{question}
Is there an analogue of Theorem~\ref{TL5} for the rank of fundamental group? 
\end{question}

The following question is more specific.

\begin{question}\label{Q4}
Let $M_1$ and $M_2$ be compact 3-manifolds with connected boundary and $\partial M_1\cong\partial M_2$. Let $\phi\colon \partial M_1\to \partial M_2$ be a homeomorphism and let $M$ be the closed manifold obtained by gluing $M_1$ to $M_2$ via $\phi$.  If $\phi$ is sufficiently complicated, then is it true that $r(M)=r(M_1)+r(M_2)-g(\partial M_i)$?
\end{question}

Question~\ref{Q4} is true if both $M_1$ and $M_2$ are handlebodies \cite{NS}.  However, the question is unknown if only one of the two manifolds is a handlebody, and it is not even known in the case of Dehn filling, i.e., $M_2$ is a solid torus.

Another interesting question related to Question~\ref{Q2} is whether the minimal value for $r(M)$ can be 2 for hyperbolic 3-manifolds.

\begin{question}
Let $M$ be a hyperbolic 3-manifold with $r(M)=2$.  Is $g(M)=2$?
\end{question}

We have shown in Theorem~\ref{Tclosed} that the discrepancy $g(M)-r(M)$ can be arbitrarily large for hyperbolic 3-manifolds.  But our construction of boundary connected sum does not change the ratio $\frac{r(M)}{g(M)}$.  

\begin{question}\label{Q6}
How small can the ratio $\frac{r(M)}{g(M)}$ be?  Can the infimum of the ratio $\frac{r(M)}{g(M)}$ be zero for 3-manifolds? 
\end{question}

In 3-manifold topology, questions on finite covering spaces are always difficult to answer.

\begin{question}\label{Q7}
Does every closed hyperbolic 3-manifold have a finite-sheeted cover $\widetilde{M}$ with $r(\widetilde{M})<g(\widetilde{M})$? 
\end{question}

Note that if the Virtually Fibered Conjecture is true, then by \cite{So}, every closed hyperbolic 3-manifold has a finite cover $M'$ with $r(M')=g(M')$.  
Question~\ref{Q6} and Question~\ref{Q7} are also related to the Heegaard gradient defined by Lackenby \cite{La2}.

\end{document}